\documentclass[10pt]{amsart}
\usepackage[a4paper, margin=1.2in]{geometry}
\usepackage{cite}
\usepackage{amsmath, amsthm, amssymb,tikz}
\usepackage{accents}
\usepackage{csquotes}
\usepackage{tikz-cd}
\usepackage{MnSymbol}
\usepackage[hidelinks]{hyperref}
\usetikzlibrary{arrows}
\usepackage{longtable}
\usepackage{mathtools}
\usepackage{bbm}
\usepackage[shortlabels]{enumitem}

\newcommand{\R}{\mathbb{R}}
\newcommand{\Q}{\mathbb{Q}}
\newcommand{\Z}{\mathbb{Z}}

\newcommand{\A}{\mathbb{A}}
\newcommand{\C}{\mathbb{C}}

\newcommand{\F}{\mathbb{F}}
\newcommand{\G}{\mathbb{G}}

\newcommand{\Gal}{\mathrm{Gal}}
\newcommand{\GL}{\mathrm{GL}}

\newcommand{\Spec}{\mathrm{Spec}}

\newcommand{\tr}{\mathrm{tr}}
\newcommand{\Res}{\mathrm{Res}}
\newcommand{\Aut}{\mathrm{Aut}}

\newtheorem{lem}{Lemma}[section]
\newtheorem{thm}[lem]{Theorem}
\newtheorem{defn}[lem]{Definition}
\newtheorem{prop}[lem]{Proposition}
\newtheorem{cor}[lem]{Corollary}
\newtheorem{con}[lem]{Conclusion}
\newtheorem{rem}[lem]{Remark}

\begin{document}
\title[Base change fundamental lemma over local function fields]{Base Change Fundamental Lemma FOR CENTRAL ELEMENTS IN DEPTH-ZERO HECKE ALGEBRAS OVER LOCAL FUNCTION FIELDS}
\author{Weimin Jiang}
\address{University of Maryland, College Park}
\begin{abstract}
Let $G$ be an unramified group over a local function field of characteristic $p>0$. This article introduces an abstract norm map between the twisted conjugacy classes and conjugacy classes of $G$ in the positive characteristics setting, applies the process of stabilization of (twisted) trace formula and proves the corresponding base change fundamental lemma for regular semisimple elements.
\end{abstract}
\maketitle

\tableofcontents
\vspace{-1cm}
\section{Introduction}
\subsection{Main Results}
Let $F$ denote a local function field of characteristic $p>0$, and let $F_r\supset F$ denote the unique degree $r$ unramified extension of $F$ contained in a fixed separable closure $F^s$ of $F$, and let $\theta$ denote a generator of $\Gal(F_r/F)$. Let $G$ denote an unramified connected reductive group over $F$. The automorphism $\theta$ determines an $F$-automorphism of $G(F_r)$, which is still denoted by $\theta$. For simplicity, we will assume that $G_{\text{der}}=G_{\text{sc}}$ in the remainder of the introduction.

From the concrete norm map 
\begin{equation*}
\begin{split}
N:G(F_r)&\to G(F_r)\\
g&\mapsto g\theta(g)\dots\theta^{r-1}(g),
\end{split}
\end{equation*}
Kottwitz \cite{kot82} was able to define an abstract norm map $\mathcal{N}$ from the set of stable $\theta$-conjugacy classes in $G(F_r)$ to the set of stable conjugacy classes in $G(F)$ for perfect fields $F$. We will say that $\gamma\in G(F)$ \textit{is a norm} if $\gamma$ is stably conjugate to $N\delta$ for some $\delta\in G(F_r)$, which means that there exist $g\in G(F^s)$ such that $g^{-1}(N\delta)g=\gamma$.

Fix a semisimple element $\gamma\in G(F)$ and we let $G_\gamma\subset G$ denote its centralizer in $G$. For $f\in C^\infty_c(G(F))$, the algebra of locally constant and compactly supported $\C$-valued functions on $G(F)$, we can define the \textit{orbital integral} 
\begin{equation*}
    O^{G(F)}_{\gamma}(f)=\int_{G_\gamma(F)\backslash G(F)}f(g^{-1}\gamma g)\,dg/dt
\end{equation*}
depending on the choice of Haar measures $dg$ and $dt$ on $G(F)$ and $G_\gamma(F)$ respectively. We also consider the \textit{stable orbital integral}
\begin{equation*}
SO^{G(F)}_\gamma(f)=\sum_{\gamma'}e(G_{\gamma'})O^{G(F)}_{\gamma'}(f),
\end{equation*}
where $\gamma'$ sums over the conjugacy classes in $G(F)$ within the stable conjugacy class of $\gamma$, and $e(G_{\gamma'})\in\{\pm1\}$ is the sign attached by Kottwitz \cite{kottwitz1983sign} to connected reductive groups. Notice that our assumption that $G_\mathrm{der}=G_\mathrm{sc}$ implies that the centralizers $G_{\gamma'}$ are connected. If $\gamma$ and $\gamma'$ are stably conjugate, their centralizers are inner forms of each form, therefore we may require the Haar measures are compatible with each other in the definition of the orbital integrals $TO^{G(F)}_{\gamma'}(f)$ in the sense of \cite{kottwitz1988tamagawa} p. 631.
Similarly, for an element $\delta\in G(F_r)$ such that $N\delta$ is semisimple, we define its twisted centralizer to be the connected reductive group $G_{\delta\theta}$ such that
\begin{equation*}
    G_{\delta\theta}(F)=\{g\in G(F_r): g^{-1}\delta\theta(g)=\delta\}.
\end{equation*}
Then for $\phi\in C^\infty_c(G(F_r))$, similarly we can define the \textit{twisted orbital integral}
\begin{equation*}
    TO^{G(F_r)}_{\delta\theta}(\phi)=\int_{G_{\delta\theta}(F)\backslash G(F_r)}\phi(g^{-1}\delta\theta(g))\, dg/dt
\end{equation*}
and its stable version
\begin{equation*}
SO^{G(F_r)}_{\delta\theta}(\phi)=\sum_{\delta'}e(G_{\delta'\theta})TO^{G(F_r)}_{\delta'\theta}(\phi),
\end{equation*}
where the sum is over $\theta$-conjugacy classes $\delta'$ in $G(F_r)$ inside the stable $\theta$-conjugacy class of $\delta$, or equivalently (Proposition \ref{prop of abstract norm map}), whose norm down to $G(F)$ is in the same stable conjugacy class as $N\delta$. If $\gamma\in G(F)$ is stably conjugate to $N\delta$, then $G_{\delta\theta}$ is an inner form of $G_\gamma$ (Lemma \ref{lemma of inner form}). Therefore we may also require the Haar measures on these groups to be compatible.

We say that $\phi\in C^\infty_c(G(F_r))$ and $f\in C^\infty_c(G(F))$ are \textit{associated} or have \textit{matching orbital integrals} (resp., at regular semisimple elements) if for every (resp., regular) semisimple element $\gamma\in G(F)$ we have 
\begin{equation*}\tag{1.1.1}\label{1.1.1}
SO^{G(F)}_\gamma(f)= \left\{ \begin{array}{rcl}
SO^{G(F_r)}_{\delta\theta}(\phi) & \mbox{if}
& \gamma=\mathcal{N}\delta \\ 0 & \mbox{if} &\gamma \mbox{ is not a norm}
\end{array}\right.
\end{equation*}

The case of spherical Hecke algebras was studied first. Suppose $K_r\subset G(F_r)$ and $K\subset G(F)$ are hyperspecial maximal compact subgroups associated to a hyperspecial vertex in the Bruhat-Tits building $\mathcal{B}(G(F))$ of $G(F)$, and suppose $\phi\in C^\infty_c(G(F_r))$ belongs to the spherical Hecke algebra $\mathcal{H}(G(F_r), K_r)$ of bi-invariant functions under $K_r$. In particular, the spherical Hecke algebras are commutative. The Satake isomorphisms gives rise to a natural algebra homomorphism
\begin{equation*}\tag{1.1.2}\label{1.1.2}
    b_r: \mathcal{H}(G(F_r), K_r)\to\mathcal{H}(G(F), K)
\end{equation*}
which will be called the \textit{base change homomorphism}. The base change fundamental lemma for spherical functions asserts that $\phi$ and $b_r(\phi)$ are associated. This was proved in the special cases of $\GL_2$ \cite{langlands1980base} and $\GL_n$ \cite{arthur1989simple}, which gave rise to global and local applications, and for general unramified reductive groups it was proved by Clozel \cite{Clozel90} and Labesse \cite{labesse1990fonctions}.

In \cite{haines2009base}, an analogous base change fundamental lemma is proved for centers of parahoric Hecke algebras by Haines, which replaced the hyperspecial maximal compact subgroup $K_r$ (resp., $K$) above by a general parahoric subgroup $J_r$ (resp., $J$) defined as intersection of the group of elements in $G(F_r)$ (resp. $G(F)$) that fix a facet in $\mathcal{B}(G(F))$ with the kernel of Kottwitz homomorphism(see section 2 of \textit{loc. cit.}). In particular, when the facet is a hyperspecial vertex, we have recovered $J_r=K_r$. However, 
since the resulting parahoric Hecke algebras $\mathcal{H}(G(F_r), J_r)$ and $\mathcal{H}(G(F), J)$ are in general no longer commutative, the base change homomorphism will be restricted to the center of those Hecke algebras:
\begin{equation*}
    b_r: \mathcal{Z}(G(F_r), K_r)\to\mathcal{Z}(G(F), K)    
\end{equation*}
\begin{rem}
From now on and in this article, when we speak of base change fundamental lemma, we refer to the matching of elements in the centers of  linked by the base change homomorphisms.     
\end{rem}

Later in \cite{haines2012base}, Haines generalizes this result to so called \textit{depth-zero} principal series blocks in the Bernstein decomposition. Generalizing his method, Shenghao Li\cite{li2026base} is able to show the base change fundamental lemma for Bernstein centers of principal series blocks (not just depth-zero ones).

All the results \cite{Clozel90, labesse1990fonctions, haines2009base, haines2012base, li2026base} stated above only work over groups over $p$-adic fields $F$ of $\text{char}\,F =0$. Results in positive characteristic are more ad-hoc for specific groups. In his thesis \cite{ray2010base}, Ray-Dulany proved the base change fundamental lemma for Iwahori-Hecke algebra of $\GL_2$ for any local field with characteristic not equal to $2$ by explicit calculations. In \cite{feng2020nearby}, Feng gave a geometric proof of the base change fundamental lemma for regular semisimple elements for parahoric Hecke algebras of $\GL_n$ over local function fields using cohomology of certain moduli stacks of shtukas and the Langlands-Kottwitz method, which generalized the previous results by \cite{ngo2006d} for spherical Hecke algebras for  $\GL_n$ over local function fields. Most recently, Bartling and Ito\cite{Ito} prove base change fundamental lemma for central elements in the parahoric Hecke
algebras for general unramified reductive groups over local function fields, using the technique of close local fields.

In this article, we prove the base change fundamental lemma for Bernstein centers of depth-zero principal
series blocks at regular semisimple elements in unramified groups $G$ over local function fields, following and generalizing the strategies in \cite{haines2009base,haines2012base}. To state the theorem, let us fix some more notations. Denote the ring of integers of $F$ by $\mathcal{O}_F$. Let $A$ denote a maximal $F$-split torus of $G$ and set $T=Z_G(A)$, a maximal torus of $G$ which is defined and unramified over $F$. We can choose an Iwahori subgroup $I\subset G(F)$ which is in good position relative to $T$. Let $I^+$ denote the pro-unipotent radical of $I$. Furthermore, let $T(F)_1$ denote the maximal compact open subgroup of $T(F)$, and let $T(F)^+_1=T(F)_1\cap I^+$ denotes its pro-unipotent radical. We will denote $\chi$ to be a character on $T(F)_1/T(F)^+_1$. Such characters are called \textit{depth-zero} characters. Via the canonical isomorphism
\begin{equation*}
 T(F)_1/T(F)^+_1\xrightarrow{\sim} I/I^+   
\end{equation*}
we see that $\chi$ induces a character $\rho:=\rho_\chi$ on $I$, trivial on $I^+$.
Then we can consider the Hecke algebra $\mathcal{H}(G,\rho)$ of our study:
\begin{equation*}
\mathcal{H}(G,\rho):=\{f\in C^\infty_c(G): f(igi')=\rho^{-1}(i)f(g)\rho^{-1}(i'), \forall\,i, i'\in I,\, \forall g\in G\},   
\end{equation*}
where the convolution is defined using the Haar measure which gives $I$ volume $1$. 
Write $\mathcal{Z}(G,\rho)$ for $Z(\mathcal{H}(G,\rho))$.

Let $N_r: T(F_r)_1\to T(F)_1$ denote the naive norm homomorphism. Therefore, the character $\chi_r:=\chi\circ N_r$ is a depth-zero character on $T(F_r)_1$, and it gives rise to the character $\rho_r$ on $I_r$ and the Hecke algebra $\mathcal{H}(G_r,\rho_r)$ with center $\mathcal{Z}(G_r,\rho_r)$. Here $G_r:=G(F_r)$ and $I_r\subset G_r$ is the Iwahori subgroup corresponding to $I$.

We will define a \textit{base change homomorphism}
\begin{equation*}
    b_r: \mathcal{Z}(G_r,\rho_r)\to\mathcal{Z}(G,\rho) 
\end{equation*}
Finally we can state the main theorem of this article:
\begin{thm}\label{main thm}
Let $F$ be a local function field of arbitrary positive characteristic. Let $G$ be a connected unramified reductive group defined over $F$. Denote $F_r/F$ to be an unramified extension of local function fields of degree $r$.
Then for any $\phi\in\mathcal{Z}(G_r,\rho_r)$, $\phi$ and $b_r(\phi)$ have the matching stable orbital integrals at every \textbf{regular semisimple element}.
\end{thm}

We write $\mathcal{H}(G(F_r),I^+_r)$ (resp., $\mathcal{H}(G(F),I^+)$)for the Hecke algebra of locally constant compactly supported functions on $G(F_r)$ (resp., $G(F)$) that are bi-invariant under the open compact subgroups $I^+_r$ (resp., $I^+$), the pro-unipotent radical of $I_r$ (resp.,  $I$), and we write $\mathcal{Z}(G(F_r),I^+_r)$ (resp., $\mathcal{Z}(G(F),I^+)$) for the center of these Hecke algebras.
In \cite{haines2012base}, \S 10, a base change homomorphism 
\begin{equation*}
    b_r:\mathcal{Z}(G(F_r),I^+_r)\to\mathcal{Z}(G(F),I^+)
\end{equation*}
was constructed and characterized using the injection
\begin{equation*}
 \mathcal{Z}(G(F_r),I^+_r)\hookrightarrow\prod_{\chi'}\mathcal{Z}(G(F_r), I_r, \chi'), 
\end{equation*}
where $\chi'$ ranges over the finite set of all depth-zero characters on $T(F_r)_1$. As a result, we are able to show a base change fundamental lemma for pro-$p$ Iwahori subgroups from Theorem \ref{main thm}:
\begin{cor}\label{pro-p Iwahori result}
If $\phi\in\mathcal{Z}(G(F_r),I^+_r)$, then the functions $\phi$ and $b_r(\phi)$ are associated.    
\end{cor}

This article overlaps with \cite{feng2020nearby} in the case when $G=\GL_n$. However, we are able to prove the vanishing results when $\gamma$ is not a norm, which was not done in \textit{loc. cit.} Moreover, this article overlaps with \cite{Ito} for central elements of Iwahori-Hecke algebras, where the results are proved for regular semisimple elements in \textit{loc. cit.}.
\subsection{Motivations}
The analogs of Theorem \ref{main thm} and Corollary \ref{pro-p Iwahori result} for groups over characteristic $0$ fields already could be useful in the pseudo-stabilization of the Lefschetz trace formula for certain simple Shimura varieties (see for example, \cite{haines2009base}, p.4 and \cite{haines2012shimura}, \S 13). For local function fields, similar counting points formula has been estabilished in \cite{song2024langlands}, Theorem 5.9.2. We will explain how Corollary \ref{pro-p Iwahori result} could be used to stabilize the corresponding Lefschetz trace formula on certain moduli space of shtukas.

Let $X$ be a smooth projective geometrically irreducible curve over $\F_q$, let $F=\F_q(X)$ be the function field of $X$. For a closed point $x\in X$, we write $F_x$ the completion of $F$ at $x$ and $\mathcal{O}_x$ the corresponding valuation ring. Let $\A=\A_F$ be the group of adeles of the global field $F$. Let $G$ be a connected reductive group over $F$ and let $\mathcal{G}$ be a parahoric model of $G$ over $X$. We also assume that $G_x$ is unramified. Fix two distinct closed points $x$ and $\infty$ in $X$. Let $K\subset G(\A)$ be a compact open subgroup of the form $K=K_xK^x\mathcal{G}(\mathcal{O}_\infty)$, where $K^x$ is a sufficiently small open compact subgroup of $G(\A^{\infty,x})$ and $K_x\subset\mathcal{G}(\mathcal{O}_x)$ is the pro-$p$ unipotent radical of some Iwahori subgroup $I_x\subset\mathcal{G}(\mathcal{O}_x)$.

Consider a proper moduli space of shtukas $\mathrm{Sht}:=\mathrm{Sht}^{\underline{\mu}}_{\mathcal{G}, \Xi, K}$ with two legs and with one leg fixed, over the local ring $\Spec\,\mathcal{O}_{y,\infty'}$ (defined in \S4.7 in \cite{song2024langlands}). We assume that the space satisfies some boundedness condition as in \textit{loc. cit.}. Let $f=f_x\otimes f^x$ where $f_x\in C^\infty_c(\mathcal{G}(\mathcal{O}_x))$ and $f^x\in C^\infty_c(G(\A^x))$ where the $\infty$-component of $f^x$ is given by the characteristic function of $\mathcal{G}(\mathcal{O}_\infty)$. The elliptic semisimple part of the alternating sum of the traces of the Hecke operator $f$ composed with a power $\sigma^r$ of Frobenius at $x$ is given by 
\begin{equation*}\tag{1.2.1}\label{counting points formula}
\begin{split}
\mathrm{tr}&(\sigma^r\times f\,|\,H^*(\mathrm{Sht}\otimes \overline{F}_{y,\infty'}, \bar{\Q}_\ell))^{\mathrm{ell,reg}}\\&=\sum_{(\gamma_0;\gamma,\delta)}|\ker^1(F, G_{\gamma_0})|\mathrm{vol}(G_{\gamma_0}(F)\backslash G_{\gamma_0}(\A)/\Xi)O^{G(\A^{x,\infty})}_\gamma(f^{x,\infty}) TO^{G(F_{x^r})}_{\delta_x,\sigma}(\phi_{r,x})TO^{G(F_{\infty^r})}_{\delta_{\infty},\sigma}(\phi_{r,\infty})
\end{split}
\end{equation*}
where the sum runs over the so called elliptic Kottwitz triples $(\gamma_0;\gamma,\delta)\in G(F)\times G(\A^{x,\infty})\times G(F_{x^r}\times F_{\infty^r})$ attached to a fixed point with some vanishing properties. Here $F_{x^r}$ is the degree-$r$ unramified extension of $F_x$.

In particular, it is conjectured that we can find the test function $\phi'_{r,x}$ (resp. $\phi'_{r,\infty}$) inside the center $\mathcal{Z}(G(F_{x^r}), K_{x^r})$ (resp. $\mathcal{Z}(G(F_{\infty^r}), K_{\infty^r})$) of the Hecke algebra $\mathcal{H}(G(F_{x^r}), K_{x^r})$ (resp. $\mathcal{H}(G(F_{\infty^r}), K_{\infty^r})$) with the same stable twisted orbital integral as $\phi_{r,x}$ (resp. $\phi_{r,\infty}$). We may simplify the right hand side to be 
\begin{equation*}\tag{1.2.2}\label{1.2.2}  \tau(G)\sum_{\gamma_0}SO^{G(\A^{x,\infty})}_\gamma(f^{x,\infty})SO^{G(F_{x^r})}_{\delta_x,\sigma}(\phi'_{r,x})SO^{G(F_{\infty^r})}_{\delta_{\infty},\sigma}(\phi'_{r,\infty})
\end{equation*}
At the place $x$, by Corollary \ref{pro-p Iwahori result}, we can write 
\begin{equation*}\tag{1.2.3}\label{1.2.3}
  SO^{G(F_{x^r})}_{\delta_x,\sigma}(\phi'_{r,x})= SO^{G(F_{x})}_{\gamma_x}(b_{r,x}(\phi'_{r,x})) 
\end{equation*}
where $b_{r,x}$ is the base change homomorphism $b_{r,x}:\mathcal{Z}(G(F_{x^r}), K_{x^r})\to\mathcal{Z}(G(F_{x}), K_{x})$. Similar base change identity holds for the place $\infty$ from the results in \cite{kottwitz1986base} since $K_\infty$ is hyperspecial maximal compact open subgroup. 

With the extra assumptions that $\mathrm{Sht}$ has no global endoscopy and that the matching (\ref{1.2.3}) can be extended to all semisimple elements, we may eventually write (\ref{1.2.2}) as 
\begin{equation*}\tag{1.2.4}\label{1.2.4}
\tau(G) \sum_{\gamma_0}SO^{G(\A)}_{\gamma}(f^{x,\infty}b_{r,x}(\phi'_{r,x})b_{r,\infty}(\phi'_{r,\infty})),    
\end{equation*}
which resembles the geometric side of the (stabilized) global trace formula for the regular action $R$ of $f':=f^{x,\infty}b_{r,x}(\phi'_{r,x})b_{r,\infty}(\phi'_{r,\infty})$ on the space $L^2(G(F)A_G\backslash G(\A))$:
\begin{equation*}\tag{1.2.5}\label{1.2.5}
\tr\,R(f)=\sum_\pi m(\pi)\,\tr\,\pi(f'),    
\end{equation*}
where $\pi$ ranges over irreducible representations of $L^2(G(F)A_G\backslash G(\A))$. Therefore, we can
relate the local factor of the zeta function with automorphic representations, eventually the $L$-function. (See \cite{haines2012shimura}, Cor. 1.4 for more details)
\subsection{Outline of the paper}
In \S\ref{sec 2}, we will define the abstract norm homomorphism for all semisimple elements in a quasi-split reductive group over $F$. This generalizes the classical results of Kottwitz \cite{kot82} to non-perfect fields, using the results from \cite{HaKi22}. The abstract norm homomorphism relates the stable twisted conjugacy classes of $G(E)$ with stable conjugacy classes of $G(F)$, providing the "transfer of conjugacy classes" needed in (\ref{1.1.1}). In \S \ref{sec 3}, we briefly recall the necessary background for the base change fundamental lemma, including the definitions of Iwahori subgroups, Bernstein center for depth-zero principal series blocks and base change homomorphisms.

The core technical content starts from \S\ref{sec 4}. In order to prove Theorem \ref{main thm}, we will reduce the problem to the case that the group has simply connected derived group and the element is strongly regular semisimple both in the adjoint group and the group itself (see \ref{conclusion of reduction steps}). In order to turn the identities between integrals into identities between traces of certain representations, we have to prove a version of very simple trace formulas (\S \ref{sec 5}), then stabilize them (\S \ref{sec 6}) in order to compare them on different groups. Finally, by producing the \textit{local data} adapted to our situation (\S\ref{sec 7.3}, \ref{sec 7.4}), we are able to reduce to the cases considered in \cite{haines2012base}. Using the same technique of Labesse elementary functions, we are finally able to prove the main theorem \ref{main thm}.
\subsection*{Acknowledgments}
I would like to thank my advisor Thomas Haines for suggesting this topic and his continuous encouragement and interest. I would like to thank Shenghao Li for spotting a mistake in the original reduction steps. I also thank Shin Eui Song for informing me the results of Hamacher-Kim. I would like to thank Peter Dillery and Kazuhiro Ito for helpful discussions.
This research was partially supported by NSF grants DMS 2200873.
\section{Norm Mapping and Conjugacy Completeness}\label{sec 2}
Let $F$ be a field of characteristic $p>0$ (not necessarily perfect). Let $G$ be a connected reductive group over $F$. We fix once for all a separable closure $F^{\text{s}}$ of $F$ inside an algebraic closure $\bar{F}$ of $F$, and let $\Gamma= \text{Gal}(F^s/F)$ be the absolute Galois group. Consider a conjugacy class $C$ of $G$. By conjugacy class we mean the conjugacy class in $G(F^s)$ (rather than in $G(\bar{F})$). We denote $C^\sigma:=\sigma(C)$ and $x^\sigma:=\sigma(x)$ for any $x\in G$. The conjugacy class $C$ is defined over $F$ if and only if $C^\sigma = C$ for all $\sigma\in\Gamma$, and the conjugacy class of an element $x\in G$ is defined over $F$ if and only if $x^\sigma$ is conjugate to $x$ under $G(F^s)$ for all $\sigma\in\Gamma$.

Let $E/F$ be a cyclic extension of degree $r$. We further assume $G$ to be an unramified connected reductive group over $F$, let $\theta\in\text{Gal}(E/F)$ be a generator. Let $R_{E/F} G_E$ denote the Weil restriction of scalars of $G_E:=G\times_FE$ to a group over $F$. The automorphism $\theta$ of $E$ determines an $F$-automorphism of $R_{E/F}G_E$ and an automorphism on its $F$-points $G(E)$, which will be also denoted by $\theta$.

Consider the \textit{concrete norm} $N_r:G(E)\to G(E)$ given by 
\begin{equation*}
N_r\delta = \delta\theta(\delta)\dots\theta^{r-1}(\delta)
\end{equation*}
It is easy to see that the conjugacy class of $N_r\delta$ in $G(F^s)$ is defined over $F$ since we can calculate that
\begin{equation*}
\theta(N_r\delta)=\delta^{-1}(N_r\delta)\delta    
\end{equation*}
The goal is to define a well-defined \textit{abstract norm map}.
\begin{equation*}\tag{2.0.1}\label{abstract norm}
\mathcal{N}:\{\text{stable }\theta\text{-conjugacy classes in }G(E)\}\to\{\text{stable conjugacy classes in }G(F)\}   
\end{equation*}
To make sense of all that, we will have to define \textit{stable $\theta$-conjugacy class} first, then we will show that the conjugacy class of $N_r\delta$ actually has a rational point over $F$.

\subsection{Stable Twisted Conjugacy}\label{sec 2.1}
We follow \cite{kot82}, \S 5. We assume $G_{\rm der}=G_{\rm sc}$ for the moment, later we will remove this assumption by using the $z$-extension of $G$.

Let $I:=R_{E/F}G_E$, there is a natural isomorphism $I_E\xrightarrow{\sim}G_E\times\dots\times G_E$, with the factors indexed by $\Gal(E/F)$. The element $\theta\in\Gal(E/F)$ determines an automorphism $s\in\text{Aut}_F(I)$, which takes the form
\begin{equation*}
s:(x_1,\dots,x_r)\mapsto(x_r,x_1,\dots,x_{r-1})    
\end{equation*}
on $E$-valued points. We can identify $G$ with $I^s$ by embedding $G$ into $I$ diagonally. 
The composition
\begin{equation*}
G(E)=I(F)\to I(E)=G(E)\times\dots\times G(E)    
\end{equation*}
is given by 
\begin{equation*}
x\mapsto (x^{\theta^{r-1}},\dots,x^\theta,x).
\end{equation*}
We can also define the concrete norm map on $I$ by
$Nx=x^{s^{r-1}}\dots x^s\cdot x$. It is defined over $F$ and it is the same on $G(E)=I(F)$ as the norm map $N:G(E)\to G(E)$ we defined before. We say that $x,y\in I(F^s)$ are \textit{$F^s$-$\theta$-conjugate} if there exists an element $g\in I(F^s)$ such that $x=g^{-s}yg$. We say that $x,y\in G(E)$ are \textit{$F^s$-$\theta$-conjugate} if they are $F^s$-$\theta$-conjugate as elements $G(E)=I(F)\subset I(F^s)$.

\begin{lem}\label{Lemma 2.1}
\begin{enumerate}[(a)]
    \item For $g,y\in I(F^s)$, we have $N(g^{-s}yg)=g^{-1}(Ny)g$.
    \item Let $x,y\in I(F^s)$. Then $x,y$ are $F^s$-$\theta$-conjugate if and only if $Nx, Ny$ are conjugate in $I(F^s)$.
    \item The embedding $G\hookrightarrow I$ induces a $\Gamma=\Gal(F^s/F)$-equivariant injection from the set of conjugacy classes in $G(F^s)$ to the set of conjugacy classes in $I(F^s)$.
\end{enumerate}    
\end{lem}
\begin{proof}
The calculations of \cite{kot82}, Lemma 5.2 are still valid here. The $\Gamma$-equivariant part is clear since the embedding $G\hookrightarrow I$ is defined over $F$. Let $x,y\in G(F^s)$, then they embed as $(x,\dots,x)$ and $(y,\dots,y)$ in $I(F^s)$. Then it is clear that they are conjugate in $I(F^s)$ if and only if they are conjugate in $G(F^s)$.  

\end{proof} 
We get an immediate corollary from (b) and (c) of Lemma \ref{Lemma 2.1} above.
\begin{cor}\label{cor 2.2}
Let $x,y\in G(E)$. Then $Nx, Ny$ are conjugate in $G(F^s)$ if and only if $x,y$ are $\theta$-conjugate by an element in $G(F^s)$.
\end{cor}

To define stable $\theta$-conjugacy, we have to define $\theta$-centralizers.
\begin{defn}\label{defn of twisted centralizer}
For $x\in G(E)=I(F)$, let $I_{sx}=\{g\in I: g^{-s}xg=x\}$. This group will be called the $\theta$\textit{-centralizer} of $x$.     
\end{defn}
It is an $F$-subgroup of $I$. Since we have $I_E=G_E\times\dots\times G_E$, where the factors are indexed by elements of $\Gal(E/F)$: $\theta^{r-1},\dots, \theta, 1$, let $p: I_E\to G_E$ be the projection onto the last factor: $(x_1,\dots,x_r)\mapsto x_r$.
\begin{rem}\label{rem 2.4}
Consider the $F$-points of $I_{s\delta}$ for $\delta\in G(E)=I(F)$, this can be given as the $F$-point of a connected reductive group as follows:
\begin{equation*}
\begin{split}
    I_{s\delta}(F)&=\{g\in I(F): g^{-s}\delta g=\delta\}\\
    &=\{g\in G(E): g^{-\theta}\delta g=\delta\}
\end{split}
\end{equation*}
We will denote by $G_{\delta\theta}$ in the remainder of the sections.
\end{rem}

\begin{lem}\label{lemma of inner form}
Let $x\in G(E)$. Then the projection $p$ induces an isomorphism $(I_{sx})_E\xrightarrow{\sim}(G_E)_{Nx}$. This makes $I_{sx}$ an $F$-form of the centralizer of $Nx$ in $G$.    
\end{lem}
\begin{proof}
The proof of Lemma 5.4 in \cite{kot82} are still valid here. 
\end{proof}
We define the notion of $\sigma$-(strongly regular) semisimplicity below.
\begin{defn}\label{defn of sigma semisimplicity}
We say that $\delta\in G(E)$ is $\sigma$-semisimple, $\sigma$-regular semisimple or $\sigma$-strongly regular semisimple if $N\delta\in G(E)$ is semisimple, regular semisimple or strongly regular semisimple.    
\end{defn}
We can show that twisted orbital integral converges for $\sigma$-regular semisimple elements. Let $\phi\in \mathcal{H}(G,\rho)$ be a function in the depth-zero Hecke algebra, where as before $\rho$ is a character on $I$, trivial on $I^+$. Since $\rho$ is a character on a finite abelian group, it must be unitary. Therefore, $|\phi|\in\mathcal{H}(G,I)$, the Hecke algebra of compactly supported bi-invariant functions under $I$. Using the results of Proposition 5.2 in \cite{Ito}, we see that the twisted orbital integrals of $\phi$ converge absolutely for $\sigma$-regular semisimple elements.

Let us assume $G$ is a general quasi-split reductive group defined over $F$, we need to define yet another subgroup $I^\circ_{sx}$. We define $I^\circ_{sx}$ as the inverse image under the $E$-isomorphism $p: I_{sx}\to G_{Nx}$ of the subgroup $G^\circ_{Nx}$ of $G_{Nx}$. We will see that $I^\circ_{sx}$ is defined over $F$ in the following lemma. It is clear that $I^\circ_{sx}=I_{sx}$ when $G_{\text{der}}$ is simply connected. In general, we consider a $z$-extension $\alpha: G'\to G$, whose definitions are given below. We will also get a corresponding $z$-extension $\gamma: I'\to I$ where $I'=R_{E/F}((G')_E)$.

\begin{defn}\label{defn z-ext}
Given a connected reductive group $G$ over $F$, we say that a homomorphism $\alpha:G'\to G$ is a \textit{$z$-extension} of $G$ if
\begin{enumerate}
    \item $G'$ is also a connected reductive group over $F$, whose derived group is simply connected.
    \item $\ker(\alpha)\subset Z(G')$ and is isomorphic to a product of tori of the form $R_{L/F}\G_m$, where each $L$ is a finite separable extension of $F$.
    \item $\alpha$ is surjective.
\end{enumerate}
\end{defn}
The existence of $z$-extensions is proved in \cite{Milne1982conjugates} Prop 3.1.

\begin{lem}\label{lem 2.8}
For any $y\in G'(E)$ such that $\alpha(y)=x$, we have $\gamma(I'_{sy})=I^\circ_{sx}$. In particular, $I^\circ_{sx}$ is defined over $F$.     
\end{lem}
\begin{proof}
It is quite easy to see that $\gamma(I'_{sy})\subset I_{sx}$. Therefore, the composition $p\circ\gamma: (I'_{sy})_E\to G_{Nx}$ makes sense. Since $I'_{sy}$ is connected, we have $p(\gamma(I'_{sy}))\subset G^\circ_{Nx}$, therefore $\gamma(I'_{sy})\subset I^\circ_{sx}$. Conversely, we use the fact that the map $\alpha:G'_{Ny}\to G^\circ_{Nx}$ is surjective, therefore, the pullback to the isomorphic groups $\gamma: I'_{sy}\to I^\circ_{sx}$ should still be surjective.     
\end{proof}

\begin{defn}\label{defn stable twisted conjugacy}
We say that $x,y\in G(E)=I(F)$ are \textit{stably $\theta$-conjugate} if there exists $g\in I(F^s)$ such that $g^{-s}xg=y$ and $g^\tau g^{-1}\in I^\circ_{sx}$ for all $\tau\in\Gamma$.    
\end{defn} 
Now since $x,y\in G(E)=I(F)$, we have $(g^{-s}xg)^\tau=g^{-\tau s}xg^\tau=g^{-s}xg$, hence $g^\tau g^{-1}\in I_{sx}$ for all $\tau\in\Gamma$. Therefore stable $\theta$-conjugacy and $F^s$-$\theta$-conjugacy coincide whenever $I^\circ_{sx}=I_{sx}$, in particular whenever $G_{\text{der}}$ is simply connected.

We introduce a special kind of $z$-extension $\alpha: G'\to G$.
\begin{defn}\label{defn of adapted z-extension}
We will say that the $z$-extension $\alpha:G'\to G$ is \textit{adapted to $E$ (or $E/F$)} if the norm map $\ker(\alpha)(E)\to\ker(\alpha)(F)$ is surjective.    
\end{defn}

\begin{lem}\label{lem 2.11}
\begin{enumerate}[(a)]
    \item Let $G$ be a connected quasi-split reductive group. There exists a $z$-extension $\alpha:G'\to G$ adapted to $E$.
    \item Let $\alpha:G'\to G$ be a $z$-extension adapted to $E$. Then $x,y\in G(E)$ are stably $\theta$-conjugate if and only if there exist $F^s$-$\theta$-conjugate elements $x',y'\in G'(E)$ such that $\alpha(x')=x$ and $\alpha(y')=y$.
\end{enumerate}  
\end{lem}
\begin{proof}
The proof of Lemma 5.6 \cite{kot82} can be adapted to the equal characteristic situation. 
\end{proof}

\subsection{Definition of the Norm Homomorphism}\label{sec 2.2}
To show that for any $\delta\in G(E)$, $N_r(\delta)$ is stably conjugate to an element in $G(F)$, therefore defines a stable conjugacy class in $G(F)$, we will have to use the classical results from \cite{kot85} and recent progress from \cite{HaKi22}. Denote $\breve{F}$ to be the maximal unramified extension of $F$, in particular we have $E\subset \breve{F}$.

\begin{prop}\label{Prop in Isocrystal}
Let $\delta\in G(\breve{F})$. Then there exists an $F$-Levi subgroup $M$ and a pair $(J,u)$ such that
\begin{enumerate}
    \item $J$ is an inner twist of $M$;
    \item $u$ is an $\breve{F}$-isomorphism $J_{\breve{F}}\xrightarrow{\sim}M_{\breve{F}}$;
    \item $u(\theta(x))=b\cdot\theta(u(x))\cdot b^{-1}$.
    \item The inner twist $u$ identifies $J(F)$ with the set 
\begin{equation*}
    \{g\in G(\breve{F}): g^{-1}\delta\theta(g)=\delta\}
\end{equation*}
\end{enumerate}
\end{prop}
\begin{proof}
This is from \cite{kot85} Remark 6.5 and 5.2. See also \cite{HaKi22}, Prop. 6.2.   
\end{proof}

Now we use Lemma 8.1 from \cite{HaKi22}.
\begin{lem}\label{lem 2.13}
Let $G^*$ be a quasi-split reductive group over an infinite field $F$ with a simply connected derived subgroup, and $H^*$ an $F$-subgroup of $G^*$ containing a maximal torus. Let $H$ be an inner form of $H^*$, and fix an inner twist $H_{F^s}\xrightarrow{\sim}H^*_{F^s}$, hence we can view $H(F)$ as a subgroup of $G^*(F^s)$.

Then for any semisimple element $a\in H(F)$, the $G^*(F^s)$-conjugacy class of $a$ contains an element $\gamma\in G^*(F)$.
\end{lem}
\begin{proof}
The proof in \cite{HaKi22} relies on their earlier results (Theorem A.1.1) in \cite{hamacher2021g} on the rationality of regular semisimple conjugacy classes, which we will document below.
\end{proof}
\begin{lem}\label{rationality of regular semisimple}
Let $F$ be any field. Let $G$ be a connected quasi-split reductive group over $F$. Let $C\subset G$ be a regular semisimple conjugacy class defined over $F$. Then there exists an element $x\in G(F)\cap C$.    
\end{lem}
\begin{rem}\label{rem 2.15}
The lemma above, however, does not show that every semisimple conjugacy class has the same property, but it suffices for the purpose of defining the abstract norm map.   
\end{rem}
Now we consider an element $\delta\in G(E)$ and whose concrete norm $N_r(\delta)$ is semisimple. Since $\delta\,\theta(N_r\delta)=(N_r\delta)\delta$, by the above proposition, with $G^*=G$, $H^*=M$ and $H=J$. We see that $N_r\delta\in H(F)$, and the stable conjugacy class of $N_r\delta$ contains an element of $G(F)$. Combined with Corollary \ref{cor 2.2}, this finishes the definition of the abstract norm map \ref{abstract norm} in the case when $G_{\text{der}}$ is simply connected.

For the next step, we need to follow the argument after Lemma 5.1 in \cite{kot82} to define the abstract norm map for all quasi-split reductive groups. To be specific, we find a $z$-extension $\alpha: G'\to G$ adapted to $E$. Then we can define the abstract norm map on $G'(E)$ by the previous results, and we have the following commutative diagram  
\begin{equation*}
\begin{tikzcd}
G'(E) \arrow[r, "\mathcal{N}_r"] \arrow[d, twoheadrightarrow,"\alpha"]
& G'(F)/\sim \arrow[d,"\alpha"] \\
G(E) \arrow[r, dashrightarrow, "\mathcal{N}_r"]
&  G(F)/\sim
\end{tikzcd}
\end{equation*}
and we claim that there exists a unique homomorphism, still denoted by $\mathcal{N}_r$, mapping stable-$\theta$-conjugate classes in $G(E)$ to stable conjugate classes in $G(F)$. Indeed, uniqueness of the map comes from the fact that $\alpha: G'(E)\to G(E)$ is surjective. For existence, we have to show that if $x,y\in G'(E)$ and $\alpha(x)=\alpha(y)$, then $\alpha(\mathcal{N}_r(x))=\alpha(\mathcal{N}_r(y))$ as stable classes. We can find $z\in\ker(\alpha)(E)$ such that $x=yz$. Since $\ker(\alpha)$ is central, we have $N_r(x)=N_r(y)N_r(z)$ with $N_r(z)\in\ker(\alpha)(F)$. Hence $\alpha(N_r(x))=\alpha(N_r(y))$ as desired.

Next, we show that the definition of $\mathcal{N}_r$ does not depend on the choice of the $z$-extension. As in the \cite{kot82}, this reduces to show that the following diagram is commutative
\begin{equation*}
\begin{tikzcd}
G_1(E) \arrow[r, "\mathcal{N}_r"] \arrow[d]
& G_1(F)/\sim \arrow[d] \\
G_2(E) \arrow[r, "\mathcal{N}_r"]
&  G_2(F)/\sim
\end{tikzcd}
\end{equation*}
for $G_1, G_2$ two $z$-extensions of $G$. But this comes from the definition of $\mathcal{N}_r$ when the derived group is simply connected.

This completes the definition of $\mathcal{N}_r$ for \textbf{all elements in a quasi-split reductive group $G$ with semisimple norms}. We are only left to show the following proposition.

\begin{prop}\label{prop of abstract norm map}
Let $x,y\in G(E)$. Then $x,y$ are stably $\theta$-conjugate if and only if $\mathcal{N}_r(x)=\mathcal{N}_r(y)$.    
\end{prop}
\begin{proof}
The proof of \cite{kot82} Proposition 5.7 adapts here.
\end{proof}
Although it is not needed in this article, we may extend the definition of the abstract norm map to regular semisimple classes in general connected reductive groups $G$ over $F$, that is, we no longer assume that $G$ is quasi-split. Following \cite{kot82}, \S5, we choose an inner twisting $\psi: G\to G^*$, and we will define a norm mapping $\mathcal{N}_r$ from stable $\theta$-conjugacy classes in $G(E)$ to stable conjugacy classes in $G^*(F)$.

We first assume that $G_{\rm der}=G_{\rm sc}$
We assume that $\delta\in G(E)$ such that $N_r(\delta)$ is regular semisimple. As before, the conjugacy class of $N_r\delta$ in $G$ is defined over $F$. Consider the conjugacy class $C$ of the regular semisimple element $\psi(N_r(\delta))$ in $G^*$, we see that $\tau(C)=C$ for any $\tau\in \Gal(F^s/F)$, therefore $C$ is defined over $F$. Lemma \ref{rationality of regular semisimple} tells us that there exists a $G^*(F^s)$-conjugacy class, or stable conjugacy class in $G(F)$ consisting of elements that are conjugate to $\psi(N_r(\delta))$ in $G(F^s)$, moreover, by the first paragraph of the proof of Corollary A.1.2 in \cite{hamacher2021g}, this class is unique. We define $\mathcal{N}_r(\delta)$ to be this class.

Using the same idea of $z$-extensions as before, we can extend the definition of $\mathcal{N}_r$ to general connected reductive groups and prove similar properties as Prop. \ref{prop of abstract norm map}. Therefore, we conclude this section by the following:

\begin{con}\label{concl of norm map}
Let $G$ be a connected reductive group defined over local field $F$ of characteristic $p>0$ and let $E/F$ be an unramified extension of degree $r$. Fix an inner twisting $\psi: G\to G^*$. For \textbf{elements $\delta\in G(E)$ such that $N_r\delta$ is regular semisimple}, we can define an injective abstract norm map
\begin{equation*}
\mathcal{N}_r:\{\text{stable }\theta\text{-conjugacy classes in }G(E)\}\to\{\text{stable conjugacy classes in }G^*(F)\}. 
\end{equation*}

If moreover $G$ is quasi-split, then for \textbf{elements $\delta\in G(E)$ such that $N_r\delta$ is semisimple}, we can also define an injective abstract norm map
\begin{equation*}
\mathcal{N}_r:\{\text{stable }\theta\text{-conjugacy classes in }G(E)\}\to\{\text{stable conjugacy classes in }G(F)\}. 
\end{equation*}
\end{con}

\section{Background}\label{sec 3}
In this section, we briefly introduce and review various objects needed to state Theorem \ref{main thm} for completeness.
\subsection{Iwahori subgroups}
Let $F$ denote a nonarchimedean local field of characteristics $p\geq0$. Let $\mathcal{O}_F$ denote the ring of integers in $F$, and let $\varpi\in\mathcal{O}_F$ be a uniformizer. Let $q=p^n$ denote the cardinality of the residue field of $F$. Fix an algebraic closure $\overline{F}$ for $F$ and a separable closure $F^s\subset \overline{F}$, and let $L$ denote the completion of the maximal unramified extension of $F$ inside $F^s$. Let $\sigma\in\Aut(L/F)$ denote the Frobenius automorphism of $L$ over $F$. Let $\mathcal{O}_L$ denote the ring of integers in $L$. The valuation $\text{val}_F:F^{\times}\to\Z$ is normalized such that $\text{val}_F(\varpi)=1$. Define $|x|_F:=q^{-\text{val}_F(x)}$ for $x\in F^{\times}$.

Let $G$ denote a connected reductive group that is defined and unramified over $F$. Let $A$ denote a maximal $F$-split torus in $G$ and set $T:=\text{Cent}_G(A)$, a maximal torus in $G$ defined over $F$ and split over $L$. We will use the symbol $G$ to denote the group $G(F)$ of $F$-points.

Let $E/F$ be an unramified extension of degree $r$ contained in $L$. We fix a generator $\theta\in \Gal(E/F)$ and we use the same symbol $\theta$ to denote the induced automorphisms of groups of $E$-points $T(E), G(E)$, etc.

We consider the Bruhat-Tits building $\mathcal{B}(G(L))$ (resp., $\mathcal{B}(G)$) for $G(L)$ (resp., $G(F)$). The Bruhat-Tits buildings associated to the semisimple $F$-groups $G_{\text{ad}}$ and $G_{\text{der}}$ can be canonically identified and are denoted $\mathcal{B}_{ss}(G)$. The group $G(L)\rtimes\Aut(L/F)$ (resp., $G(F)$) acts on $\mathcal{B}(G(L))$ (resp., $\mathcal{B}(G)$). Via this action, we can identify $\mathcal{B}(G)$ with the $\sigma$-fixed subset $\mathcal{B}(G(L))^\sigma\subset\mathcal{B}(G(L))$.

Let $\mathcal{A}^L$ (resp., $\mathcal{A}$) denote the apartment of $\mathcal{B}(G(L))$ (resp., $\mathcal{B}(G)$) corresponding to the torus $T$ (resp., $A$). Then $\mathcal{A}^L$ (resp., $\mathcal{A}$) is endowed with a family of hyperplanes given by the vanishing of the affine roots $\Phi_{\text{aff}}(G,T,L)$ (resp., $\Phi_{\text{aff}}(G,A,F)$).
Under the identification $\mathcal{B}(G)=\mathcal{B}(G(L))^\sigma$, the apartment $\mathcal{A}$ is identified with $(\mathcal{A}^L)^\sigma$. Moreover, the affine roots $\Phi_{\text{aff}}(G,A,F)$ are nonconstant restrictions to $\mathcal{A}=(\mathcal{A}^L)^\sigma$ of the affine roots $\Phi_{\text{aff}}(G,T,L)$. These affine roots determine the notions of alcoves, facets and Weyl chambers used throughout this article.

We fix once for all a $\sigma$-invariant alcove $\mathbf{a}\in\mathcal{A}^L$, moreover, within $\bar{\mathbf{a}}$ we fix a $\sigma$-invariant facet $\mathbf{a}_J$ and a $\sigma$-invariant hyperspecial vertex $\mathbf{a}_0$. Kottwitz defines a functorial surjective homomorphism
\begin{equation*}
    \kappa_G:G(L)\longrightarrow X^*(\hat{Z}(G)^I)
\end{equation*}
where $I=\text{Gal}(L^s/L)$ denotes the inertia group. We denote $G(L)_1:=\ker\kappa_G$ by convention.


\begin{defn}\label{defn of parahoric subgroup}
A \textit{parahoric subgroup} of $G(L)$ associated to an arbitrary facet $F$ of $\mathcal{B}_{\text{ss}}(G(L))$ is a subgroup of the form
\begin{equation*}
    K_F=\mathrm{Fix}(F)\cap G(L)_1,
\end{equation*}
where $\mathrm{Fix}(F)\leq G(L)$ is the subgroup of elements which fix $F$ pointwise. An \textit{Iwahori subgroup} of $G(L)$ is the parahoric subgroup associated to an alcove of $\mathcal{B}_{\text{ss}}(G(L))$. A parahoric subgroup of $G$ will be a subgroup of $G$ of the form $K_F\cap G$.
\end{defn}
Therefore, the facets $\mathbf{a}_0, \mathbf{a}$ and $\mathbf{a}_J$ give rise to parahoric subgroups of $G(L)$: a hyperspecial maximal parahoric subgroup $K(L)$, an Iwahori subgroup $I(L)$ and a general parahoric subgroup $J(L)$. Moreover, we have that $K(L)\supset I(L)\subset J(L)$. Moreover, since the facets are $\sigma$-invariant, we get the corresponding parahoric subgroups $K, I$ and $J$ of $G(F)$ with similar relations.
\begin{rem}\label{rem 3.2}
This definition is taken from \cite{haines2008parahoric}, which is different from the classical definition of Bruhat-Tits. However, they are equivalent by Prop. 3 in \textit{loc. cit.}      
\end{rem}
\begin{rem}\label{rem 3.3}
When $G=T$ is a torus, then there is exactly one parahoric subgroup $K$ of $T(L)$, namely, 
\begin{equation*}
K=\mathcal{T}^0(\mathcal{O}_L),    
\end{equation*}
where $\mathcal{T}^0$ is the identity component of the Neron model of $T$. Moreover, we have $T(L)_1=\mathcal{T}^0(\mathcal{O}_L)$.
\end{rem}

\subsection{Bernstein center for depth-zero principal series}
We keep assuming that $G$ is an unramified connected reductive group over $F$. We denote $\mathfrak{R}(G)$ to be the category of smooth representations of $G(F)$ on $\C$-vector spaces. All the representations in this article will be on complex vector spaces. The \textbf{Bernstein center} $\mathfrak{Z}(G)$ of the group $G$ is defined as the ring of endomorphisms of the identity functor on $\mathfrak{R}(G)$.
If $G$ contains an $F$-rational parabolic subgroup $P$ with $F$-Levi factor $M$ and unipotent radical $N$ (such that $P=MN$), we define
the modulus function $\delta_P:M(F)\to\R_{>0}$ by
\begin{equation*}
\delta_P(m):=|\det(\text{Ad}(m);\text{Lie}(N(F)))|_F,    
\end{equation*}
where $|\cdot|_F$ is the normalized absolute value on $F$. For $\sigma\in\mathfrak{R}(M)$, we will often consider the normalized induced representation
\begin{equation*}
i^G_P(\sigma)=\text{Ind}^{G(F)}_{P(F)}(\delta^{1/2}_P\sigma),    
\end{equation*}
where $\delta^{1/2}_P(m)$ means the positive square-root of the positive real number $\delta_P(m)$. The normalization will preserve unitary representations.

Let $T$ be a maximal unramified $F$-torus, which means that $T$ splits over an unramified extension of $F$. In this case $T(F)_1=T(F)^1$  is the unique maximal compact open subgroup of $T(F)$ (\cite{Kaletha_Prasad_2023}, Lemma 2.5.18). 

Let $T(F)_1$ denote the maximal compact open subgroup of $T(F)$, and let $T(F)^+_1:=T(F)_1\cap I^+$ denote its pro-unipotent radical.
Let $\chi$ denote a \textit{depth-zero} character on $T(F)_1$, which means that $\chi$ factors through the quotient $T(F)_1/T(F)^+_1$. Let $\Tilde{\chi}$ denote any extension of $\chi$ to a character $\Tilde{\chi}: T(F)\to\C^\times$. Consider the inertial class
\begin{equation*}
\mathfrak{s}=\mathfrak{s}_\chi=[T(F),\Tilde{\chi}]_G.    
\end{equation*}
The inertial class $\mathfrak{s}$ depends only on the $W(F)$-orbit of $\chi$.

Let $N$ denote the norm homomorphism $T(E)\to T(F)$. It maps $T(E)_1\to T(F)_1$ and $T(E)^+_1\to T(F)^+_1 $. Therefore we will use the same notation to denote the norm map on the quotients $N:T(E)_1/T(E)^+_1\to T(F)_1/T(F)^+_1$. 

Let $\chi_N:=\chi\circ N$. This will give a depth-zero character on $T(E)_1$. The set of all smooth characters on $T(F)$ carries a left action under the Weyl group $W(F)=N_GT(F)/T(F)$. Let $W_\chi$ denote the subgroup of elements in $W(F)$ which fix $\chi$. Similarly we define $W_{\chi_N}$ in the Weyl group $W(E)$ of $G(E)$.

Let $\mathfrak{R}_\mathfrak{s}(G)$ denote the Bernstein component indexed by $\mathfrak{s}$, in other words, this is the full subcategory of $\mathfrak{R}(G)$ whose objects have the property that each of their subquotients is a subquotient of a principal series representation $i^G_B(\Tilde{\chi}\eta)$ for some unramified character $\eta$ of $T(F)$. Sometimes we denote $\mathfrak{R}_\mathfrak{s}(G)$ by $\mathfrak{R}_\chi(G)$.

Now fix $\chi$ and $\mathfrak{s}=\mathfrak{s}_\chi$ as above. We have 
\begin{equation*}
\mathfrak{X}_\mathfrak{s}=\{(T,\xi)_G\}    
\end{equation*}
of supercuspidal supports $(T,\xi)_G$ of irreducible representations $\pi$ in the category $\mathfrak{R}_\chi(G)$. This means that $\pi$ is a subquotient of $i^G_B(\xi)$, where $B$ is any Borel subgroup of $G$ containing $T$ as the Levi factor. Here $\xi:T(F)\to\C^\times$ is a smooth character extending some $W(F)$-conjugate of $\chi$. 

\begin{rem}\label{remark extension of chracter}
 There exists at least one $W_\chi$-invariant extension $\Tilde{\chi}$ of $\chi$, by using the canonical isomorphism $X_*(T)\xrightarrow{\sim} T(F)/T(F)_1$ given by $\nu\mapsto\varpi^\nu$, where $\varpi$ is a choice of uniformizer of $F$. Hence the extension $\Tilde{\chi}^\varpi$ is defined by the formula
\begin{equation*}
\Tilde{\chi}^\varpi(\varpi^\nu t_0)=\chi(t_0), \forall\nu\in X_*(T), t_0\in T(F)_1.  
\end{equation*}
\end{rem}

Fix one such extension $\Tilde{\chi}$, we have a bijection 
\begin{equation*}
\begin{split}
\Hat{A}/W_\chi &  \xrightarrow{\sim} \mathfrak{X}_\mathfrak{s}\\
\eta &\mapsto (T,\Tilde{\chi}\eta)_G
\end{split}
\end{equation*}
where $\eta\in\hat{A}$ can be viewed as an unramified character on $T(F)$ as in \cite{haines2009base}, Lemma 2.4.2. Up to isomorphism, this structure does not depend on the choice of $\chi$ in its $W(F)$-orbit, nor on the choice of the extension $\Tilde{\chi}$ of $\chi$. We have 
\begin{equation*}
\mathfrak{X}_\mathfrak{s}=\Spec(\C[X_*(A)]^{W_\chi})   
\end{equation*}

We have a isomorphism $I/I^+\xrightarrow{\sim} T(F)_1/T(F)^+_1$, hence the depth zero character $\chi:T(F)_1/T(F)^+_1\to\C^\times$ induces 
\begin{equation*}
\rho=\rho_\chi: I/I^+\to\C^\times.    
\end{equation*}
We have the following proposition. By definition, $(I,\rho)$ is a \textit{Bushnell-Kutzko type} for $\mathfrak{R}_\chi(G)$ means that an irreducible representation $\pi\in\mathfrak{R}(G)$ belongs to $\mathfrak{R}_\chi(G)$ if and only if $\rho\subset \pi|_I$.

\begin{prop}\label{prop Iwahori type}
The pair $(I,\rho)$ is a Bushnell-Kutzko type for $\mathfrak{R}_\chi(G)$.     
\end{prop}

\begin{proof}
The proof of Theorem 3.0.2 in \cite{haineshecke} works for depth-zero case and any characteristic.    
\end{proof}

Given Proposition \ref{prop Iwahori type}, the category $\mathfrak{R}_\chi(G)$ is equivalent to the category of $\mathcal{H}(G,\rho)$-modules. (\cite{haineshecke}, Prop. 2.0.3). Let $\mathcal{Z}(G,\rho):=Z(\mathcal{H}(G,\rho))$, this also means that there is a canonical algebra isomorphism
\begin{equation*}\tag{3.2.1}\label{Bernstein isomorphism}
 \beta:\C[\mathfrak{X_\mathfrak{s}}]\xrightarrow{\sim} \mathcal{Z}(G,\rho),  
\end{equation*}
which can be characterized as follows. For an extension $\xi$ of some $W$-conjugate of $\chi$, we consider the space $i^G_B(\xi)^\rho$ of functions $f\in i^G_B(\xi)$ such that 
\begin{equation*}
f(gi)=f(g)\rho(i)    
\end{equation*}
for all $g\in G$ and $i\in I$. Then 
\begin{equation*} 
z\in \mathcal{Z}(G,\rho)\text{ acts on the left on } i^G_B(\xi)^\rho\text{ by the scalar }[\beta^{-1}(z)](\xi), 
\end{equation*}
where $\beta^{-1}(z)$ is viewed as a regular function on the variety $\mathfrak{X}_\mathfrak{s}$ and by $\xi$ we mean the point $(T,\xi)_G\in\mathfrak{X}_\mathfrak{s}$, and the action is convolution of functions on $G$ with suitable Haar measure.

\begin{rem}\label{rem 3.6}
The fact that $\beta^{-1}(z)$ is well-defined as a function on the class $(T,\xi)_G$ means that $[\beta^{-1}(z)](\,^w\xi)=[\beta^{-1}(z)](\xi)$ for all $w\in W$. In other words, $\beta^{-1}(z)$ is $W$-invariant as a function of $\xi$.
\end{rem}

\subsection{Base change homomorphism}
We fix a depth-zero character $\chi$ on $T(F)_1$, and let $(I,\rho)$ be the $\mathfrak{s}=\mathfrak{s}_\chi$-type described above. Let $A^E$ denote the unique maximal $E$-split torus in $G$ containing $A$. It is easy to see that $T=Z_G(A^E)$. Consider $\chi_N:=\chi\circ N$ as a depth-zero character on $T(E)_1$ and consider the corresponding inertial class $\mathfrak{s}_E:=\mathfrak{s}_{\chi_N}$ for $G(E)$. Let $(I_E, \rho_E)$ denote the $\mathfrak{s}_E$-type associated to the character $\chi_N$. We denote the corresponding Hecke algebra and its center by $\mathcal{H}(G(E),\rho_N)$ and $\mathcal{Z}(G(E),\rho_N)$, respectively.

There is a canonical morphism of algebraic varieties
\begin{equation*}
\begin{split}
N^*: \mathfrak{X}_\mathfrak{s}&\to \mathfrak{X}_{\mathfrak{s}_E}\\
(T,\xi)_G &\mapsto (T(E),\xi_N)_{G(E)},
\end{split}    
\end{equation*}
where $\xi$ denotes an extension of some $W(F)$-conjugate of $\chi$. This induces an algebra homomorphism
\begin{equation*}\tag{3.3.1}\label{3.3.1}
N:\C[\mathfrak{X}_{\mathfrak{s}_E}]\to\C[\mathfrak{X}_\mathfrak{s}]    
\end{equation*}
Now we can define the \textit{base change homomorphism} $b_E: \mathcal{Z}(G(E),\rho_N)\to \mathcal{Z}(G,\rho)$ to be the unique morphism making the following diagram commute:
\begin{equation*}\tag{3.3.2}\label{3.3.2}
\begin{tikzcd}
\C[\mathfrak{X}_{\mathfrak{s}_E}] \arrow[r, "\sim"',"\beta", rightarrow] \arrow[d, "N"]
& \mathcal{Z}(G(E),\rho_N)\arrow[d, "b_E"] \\
\C[\mathfrak{X}_{\mathfrak{s}}] \arrow[r, "\sim"',"\beta", rightarrow] 
&  \mathcal{Z}(G,\rho)
\end{tikzcd}
\end{equation*}
We can interpret the diagram in terms of the actions on principal series to get the following lemma:
\begin{lem}\label{lem characterization of base change hom}
For any character $\xi:T(F)\to\C^\times$ which extends some $W(F)$-conjugate of $\chi$, define $\xi_N:=\xi\circ N$, a character on $T(E)$ which extends some $W(E)$-conjugate of $\chi_N$. Let $z\in\mathcal{Z}(G(E), \rho_E)$. Then $b(z)$ is the unique element in $\mathcal{Z}(G,\rho)$ which acts on every module $i^G_B(\xi)^\rho$ by the same scalar by which $z$ acts on $i^{G(E)}_{B(E)}(\xi_N)^{\rho_N}$. In other words,
\begin{equation*}
[\beta^{-1}(b(z))](\xi)=[\beta^{-1}(z)](\xi_N)    
\end{equation*}
\end{lem}
We can rephrase the above lemma in terms of right actions: for $z\in\mathcal{Z}(G(E), \rho_E)$, $b(z)$ acts on the right on $i^G_B(\xi^{-1})^{\rho^{-1}}$ by the scalar by which $z$ acts on the right on $i^{G(E)}_{B(E)}(\xi^{-1}_N)^{\rho^{-1}_N}$, in other words,
\begin{equation*}\tag{3.3.3}\label{3.3.3}
ch_{\xi^{-1}}(b(z))=ch_{\xi^{-1}_N}(z).    
\end{equation*}
We fix $w\in W(F)$ and use the same symbol to denote its lift in $N_G(T)$. The character $^w\chi$ is defined by $^w\chi(t)=\chi(w^{-1}tw)$. Similarly, for any suitable function $\Phi$, we define $^w\Phi(\cdot)=\Phi(w^{-1}\cdot w)$. We write $^wI:=wIw^{-1}$ and $^wI^+:=wI^+w^{-1}$. We extend $^w\chi$ to the character $^w\rho:\,^wI/^wI^+\to\C^\times $ using Iwasawa decomposition of $^wI$, and therefore we have $^w\rho(wiw^{-1})=\rho(i)$ for $i\in I$. There is also an isomorphism of algebras  $\mathcal{H}(G,I,\rho)\xrightarrow{\sim}\mathcal{H}(G,\,^wI,\,^w\rho)$ given by $h\mapsto \,^wh$.

As usual, we let $\xi$ denote an extension of a $W(F)$-conjugate of $\chi$. We write $^wB:=wBw^{-1}$, then we have an isomorphism
\begin{equation*}
i^G_B(\xi^{-1})^{\rho^{-1}}\xrightarrow{\sim}i^G_{^wB}(^w\xi^{-1})^{^w\rho^{-1}}    
\end{equation*}
given by $\Psi\mapsto\,^w\Psi$. This isomorphism also intertwines the right actions of $h\in\mathcal{H}(G,I,\rho)$ and $^wh\in\mathcal{H}(G,\,^wI,\,^w\rho)$, in the sense that
\begin{equation*}
^w(\Psi\cdot h)=\,^w\Psi\cdot\,^wh.    
\end{equation*}
Taking $h=z\in\mathcal{Z}(G,I,\rho)$, we have $\beta^{-1}(z)(\xi)=\beta^{-1}(\,^wz)(\,^w\xi)=\beta^{-1}(\,^wz)(\xi)$. The last equality comes from the Remark \ref{rem 3.6}. In other words, we have the following commutative diagram:
\begin{equation*}
\begin{tikzcd}
\C[\mathfrak{X}_\chi] \arrow[r, "\sim"',"\beta", rightarrow] \arrow[d, "="]
& \mathcal{Z}(G, I, \rho)\arrow[d, "z\mapsto\,^wz"] \\
\C[\mathfrak{X}_\chi] \arrow[r, "\sim"',"\beta", rightarrow] 
&  \mathcal{Z}(G, \,^wI, \,^w\rho)
\end{tikzcd}
\end{equation*}
Combining this diagram with the diagram (\ref{3.3.2}) for $G$ and $G(E)$ respectively, we have the following lemma, which shows the compatibility of $z\mapsto\,^wz$ with the base change homomorphism:
\begin{lem}\label{lem base change compatible with Weyl action}
For any $w\in W(F)$, the following diagram is commutative:
\begin{equation*}
\begin{tikzcd}
\mathcal{Z}(G(E),\rho_N) \arrow[r, "z\mapsto\,^wz", rightarrow] \arrow[d, "b"]
& \mathcal{Z}(G(E),\,^w\rho_E)\arrow[d, "b"] \\
\mathcal{Z}(G,\rho)\arrow[r, "z\mapsto\,^wz", rightarrow] 
&  \mathcal{Z}(G,\,^w\rho)
\end{tikzcd}
\end{equation*}
\end{lem}

\section{Reduction Steps}\label{sec 4}
In this section, we follow closely the reduction steps of \cite{haines2009base}, \S5 and \cite{haines2012base}, \S7 with a few adaptations.

\subsection{Definition of stable twisted orbital integral}\label{sec 4.1}
We still denote $G$ to be an unramified connected reductive $F$-group. Let $\delta\in G(E)$ be such

that $\mathcal{N}\delta\in G(F)$ (as a stable class) is semisimple. Let $e(\delta):=e(G^\circ_{\delta\sigma})$ denote the sign attached by Kottwitz \cite{kottwitz1983sign} to the connected reductive $F$-group $G^\circ_{\delta\sigma}$. (There is no assumption on the characteristics of the ground field $F$ in loc. cit.) And define $a(\delta)$ to be the cardinality of the set
\begin{equation*}
\ker[H^1(F,G^\circ_{\delta\sigma})\to H^1(F,G_{\delta\sigma}) ]    
\end{equation*}
Therefore $a(\delta)=1$ for those $G$'s with simply-connected derived group. Now for any function $\phi\in C^\infty_c(G(E))$, we can define its \textit{stable} twisted orbital integral by 
\begin{equation*}\tag{4.1.1}\label{4.1.1}
SO_{\delta\sigma}(\phi):=\sum_{\delta'}e(\delta')\,a(\delta')TO_{\delta'\sigma}(\phi)  
\end{equation*}
where $TO_{\delta'\sigma}(\phi)$ is defined in \cite{haines2009base} (4.4.1). Here $\delta'$ ranges over $\sigma$-conjugacy classes in $G(E)$ which are stably $\sigma$-conjugate to $\delta$. As a special case when $E=F$ and $\sigma=id$, we have also defined  the stable orbital integral $SO_\gamma(f)$ for a semisimple element $\gamma\in G(F)$ and for $f\in C^\infty_c(G(F))$.

\subsection{Vanishing statements for non-norms}\label{sec 4.2}
\begin{lem}\label{lem vanishing non-norm}
Let $\phi\in \mathcal{Z}(G(E),\rho_r)$. If $\gamma$ is not a norm from $G(E)$, then $SO_\gamma(b\phi)=0$.   
\end{lem}
\begin{proof}

Recall that $b\phi\in\mathcal{Z}(G,\rho)$. First we assume $\gamma$ is elliptic. We will prove the stronger statement that if $\gamma'\in G(F)$ is stably conjugate to $\gamma$, then $b\phi(g^{-1}\gamma'g)=0$ for every $g\in G(F)$. From this we see that $SO_\gamma(b\phi)=0$.

Consider the canonical map $p: G_{\text{sc}}\to G$ and the abelian group $H^0_{\text{ab}}(F,G)=G(F)/p(G_\text{sc}(F))$. Proposition 2.5.3 of \cite{LBC1999} shows that an elliptic element $\gamma$ is a norm from $G(E)$ if and only if its image $p(\gamma)$ in $H^0_{\text{ab}}(F,G)$ is a norm. (We notice that Labesse didn't assume the ground field $F$ to be of characteristic $0$ in CHAPITRE 1 and 2 in \textit{loc. cit.}) Now the required vanishing follows from the following lemma.

\begin{lem}\label{lem 4.2}
Let $f=b\phi$ for some $\phi\in \mathcal{Z}(G(E),\rho_r)$. Let $x\in G(F)$ be any element such that, for some character $\eta$ on the group $H^0_{\text{ab}}(F,G)$ which is trivial on the norms, we have $\eta(x)\neq1$. Then $f(x)=0$.
\end{lem}
\begin{proof}
We can view $\eta$ as a character $\eta:G(F)\to\C^\times$ by pulling back along the quotient map $G(F)\twoheadrightarrow H^0_{\text{ab}}(F,G)$, thus the condition $\eta(x)\neq1$ makes sense. 

Since $\eta$ vanishes on $G(F)^\natural:=p(G_\text{sc}(F))$, it is trivial on $G(F)_1=G(F)^\natural\cdot T(F)_1$, where $G(F)_1$ is also the kernel of the Kottwitz homomorphism $\kappa_G:G(F)\to\pi_1(G)^{\Gal(\breve{F}/F)}_I$ (see \cite{Kaletha_Prasad_2023}, Definition 2.6.23 and Proposition 11.5.4). Here $\breve{F}$ denotes the maximal unramified extension of $F$ inside a fixed algebraic closure $F^s$ and $I:=\Gal(F^s/\breve{F})$, and $T(F)_1$ is the unique maximal compact subgroup of $T(F)$. Hence the restriction of $\eta$ to $T(F)$ is trivial on $T(F)_1$. Thus we have $f\eta\in\mathcal{H}(G,\rho)$. By examining the right convolution action of $f\eta$ on $i^G_B(\xi^{-1})^{\rho^{-1}}$, we see that $f\eta\in Z(G,\rho)$ (also see the proof of Lemma \ref{lem 4.8}), and 
\begin{equation*}\tag{4.2.1}\label{4.2.1}
ch_{\xi^{-1}}(f\eta)=ch_{(\eta\xi)^{-1}}(f).    
\end{equation*}
But by (\ref{3.3.3}), this is
\begin{equation*}
ch_{(\eta_r\xi_r)^{-1}}(\phi)=ch_{\xi_r^{-1}}(\phi)=ch_{\xi^{-1}}(f),   
\end{equation*}
the first equality holds since $\eta_r:=\eta\,\circ N_r$ is trivial by assumption. This implies that $f\eta=f$ since $ch_{\xi^{-1}}(f\eta)=ch_{\xi^{-1}}(f)$ implies that $\beta^{-1}(f\eta)=\beta^{-1}(f)$. Now we have 
\begin{equation*}
f(x)(\eta(x)-1)=0    
\end{equation*}
for such $x$. Since $\eta(x)\neq1$, we have $f(x)=0$ as desired.
\end{proof}

Now consider $x:=g^{-1}\gamma g$. Since it is not a norm by assumptions, there must exist one such character $\eta$ satisfying the conditions of Lemma \ref{lem 4.2}, therefore $f(x)=0$ as desired. We have proved Lemma \ref{lem vanishing non-norm} in the case that $\gamma$ is elliptic, in the general case, we can apply the descent formula \cite{haines2009base} (4.4.6) to reduce to the elliptic case.
\end{proof}

\subsection{The case where the derived group is not simply connected}\label{sec 4.3}
The strategy is the same as in \cite{haines2009base}, \S 5.2 and \cite{haines2012base} \S 7.2. Choose a finite unramified extension $F'\supset F$, which contains $E$ and splits $G$. Consider a $z$-extension of $F$-groups adapted to $E/F$:
\begin{equation*}
1\longrightarrow Z\longrightarrow H\overset{p}{\longrightarrow} G\longrightarrow 1,    
\end{equation*}
where $Z$ is a finite product of copies of $\Res_{F'/F}\G_m$, $H_{\text{der}}=H_{\text{sc}}$ and $p$ is surjective on $E$- and $F$-points since $Z$ is an induced torus. Choose an extension of $\sigma$ to an element in $\Gal(F'/F)$, which is still denoted $\sigma$. Let $Z(E)_1$ (resp. $Z(F)_1$) denote the maximal compact subgroup of $Z(E)$ (resp. $Z(F)$). We endow $Z(E)$ (resp. $Z(F)$)  
with the Haar measure giving $Z(E)_1$ (resp. $Z(F)_1$) volume $1$. The norm homomorphism $N: Z(E)\to Z(F)$ is surjective and determines a measure-preserving isomorphism 
\begin{equation*}
N:\overline{Z(E)}:=(Z(E)/(1-\sigma)(Z(E)))\to Z(F)    
\end{equation*}
and $N: Z(E)_1\twoheadrightarrow Z(F)_1$ is surjective. Here we give the compact subgroup $(1-\sigma)(Z(E))=(1-\sigma)(Z(E)_1)$ measure $1$.

Let $\lambda:Z(F)\to\C^\times$ denote a smooth character, and for $f\in C^\infty_c(H(F))$, we set
\begin{equation*}
f_\lambda(h):=\int_{Z(F)}f(hz)\lambda^{-1}(z)\,dz    
\end{equation*}
Write $\lambda=1$ for the trivial character, then we can show that $f_1(hz')=f_1(h)$ for any $h\in H(F)$ and $z'\in Z(F)$. Therefore we can view $f_1$ as an element in $C^\infty_c(G(F))$, which is denoted $\bar{f}$. Notice that $f_\lambda$ is only compactly supported modulo $Z(F)$

We write $\lambda N$ for the character $\lambda\circ N:Z(E)\to\C^\times$. The depth-zero character $\chi$ on $T(F)_1$ determines a depth-zero character $\chi_H$ on $T_H(F)_1$, where $T_H:=p^{-1}(T)$. Let $I_H$ be the Iwahori subgroup in $H$ corresponding to the Iwahori subgroup $I$ in $G$, and let $\rho_H:I_H\to\C^\times$ denote the character constructed from $\chi_H$ by the isomorphism $I_H/I^+_H\cong T_H(F)_1/T_H(F)^+_1$. For $\phi\in C^\infty_c(H(E))$, we define analogously the function $\phi_{\lambda N}$ by the formula
\begin{equation*}
\phi_{\lambda N}(h):=\int_{Z(E)}\phi(hz)\lambda^{-1}(Nz)\,dz,
\end{equation*}
and similarly we use $\bar{\phi}$ when viewing $\phi_1$ as an element $C^\infty_c(G(E))$. It is straightforward to check that the functions $\phi_{\lambda N}$ (resp. $f_\lambda$) are transformed by $\lambda^{-1}N$ (resp. $\lambda^{-1}$) under $Z(E)$ (resp. $Z(F)$). These functions are not longer compactly supported on $H$, nonetheless we can show that the (twisted) orbital integrals of $f_\lambda$ (resp. $\phi_{\lambda N}$) exist at $(\sigma)$-semisimple element $\delta$ by the following simple calculation:
\begin{equation*}\tag{4.3.1}\label{4.3.1}
\begin{split}
TO^H_{\delta\sigma}(\phi_{\lambda N}):&=  \int_{H_{\delta\sigma}\backslash H(E)}\phi_{\lambda N}(h^{-1}\delta\sigma(h))\,d\bar{h}\\
&=\int_{H_{\delta\sigma}Z(E)\backslash H(E)}\int_{Z(F)\backslash Z(E)}\phi_{\lambda N}((vh)^{-1}\delta\sigma(hv))\,dv\,d\bar{h}\\
&= \text{vol}(Z(F)\backslash Z(E))\int_{H_{\delta\sigma}Z(E)\backslash H(E)}\phi_{\lambda N}(\bar{h}^{-1}\delta\sigma(\bar{h}))\,d\bar{h}\\
&=\int_{H_{\delta\sigma}Z(E)\backslash H(E)}\phi_{\lambda N}(\bar{h}^{-1}\delta\sigma(\bar{h}))\,d\bar{h}
\end{split}    
\end{equation*}
Since $Z(F)\backslash Z(E)\cong (1-\sigma)Z(E)$ which has measure $1$ and as a function, $\phi_{\lambda N}(\cdot\,\delta\sigma(\cdot))$ is compactly supported modulo $Z(E)$, the convergence of the (twisted) orbital integral is shown.

We also have the following useful lemma:
\begin{lem}\label{lem 4.3}
Assume $\phi\in\mathcal{H}(H(E),\Tilde{K}, \Tilde{\rho}_H)$ for some compact open subgroup $\Tilde{K}\leq H(E)$ and there exists compact open $K\leq G$ and a character $\rho: K\to\C^\times$ such that
\begin{itemize}
    \item $\Tilde{K}\supset (1-\sigma)(Z(E))$, and
    \item $\Tilde{\rho}|_{\Tilde{K}\cap Z(E)}=\rho\circ N|_{\Tilde{K}\cap Z(E)}$.
\end{itemize}
Then we have
\begin{equation*}
\phi_{\lambda N}(h)=\int_{\overline{Z(E)}}\phi(hz)\lambda^{-1}(Nz)\,d\bar{z} \end{equation*}
\end{lem}
\begin{proof}
We have 
\begin{equation*}
\begin{split}
\phi_{\lambda N}(h)&=\int_{Z(E)}\phi(hz)\lambda^{-1}(Nz)\,dz\\
&=\int_{Z(E)/(1-\sigma)Z(E)}\int_{(1-\sigma)Z(E)}\phi(vhz)\lambda^{-1}(N(vz))\,dv\,d\bar{z}\\
&=\int_{Z(E)/(1-\sigma)Z(E)}\int_{(1-\sigma)Z(E)}\phi(hz)\lambda^{-1}(Nz)\,dv\,d\bar{z}\\
&=\int_{\overline{Z(E)}}\phi(hz)\lambda^{-1}(Nz)\,d\bar{z}  
\end{split}
\end{equation*}
where the second line is \cite{folland2016course} Theorem 2.51, the third line follows from that we can write $v=v'\sigma(v'^{-1})$ for some $v'\in Z(E)$ so that $Nv=1$ and $\phi(vhz)=\phi(hz)\rho(Nv)=\phi(hz)$, and in the last line we use our choice of measure that gives $(1-\sigma)Z(E)$ volume $1$.
\end{proof}
\begin{lem}\label{lem 4.4}
Suppose that $\phi\in\mathcal{H}(H(E), \Tilde{K}, \Tilde{\rho})$ (resp. $f\in\mathcal{H}(H(F), K, \rho)$) for a compact open subgroup $\Tilde{K}\subset H(E)$ and a character $\Tilde{\rho}:\Tilde{K}\to\C^\times$ (resp. $K\subset H$ and $\rho:K\to\C^\times$) such that 
 \begin{itemize}
        \item $N(\Tilde{K}\cap Z(E))=K\cap Z(F)$,
        \item $\Tilde{K}\supset(1-\sigma)Z(E)$,
        \item $\Tilde{\rho}|_{\Tilde{K}\cap Z(E)}=\rho\circ N|_{\Tilde{K}\cap Z(E)}$.
\end{itemize}

We have the following statements:
\begin{enumerate}[(i)]
    \item The functions $\phi$ and $f$ are associated 
    if and only if $\phi_{\lambda N}$ and $f_\lambda$ are associated for every $\lambda$.
    \item In (i) we only need to consider characters $\lambda$ such that 
    \begin{equation*}
    \lambda|_{K\cap Z(F)}=\rho^{-1}|_{K\cap Z(F)}.    
    \end{equation*}
    \item If $\phi\in\mathcal{H}(H(E),\Tilde{\rho}_H)$ and $f\in\mathcal{H}(H(F), \rho_H)$, then in (i) we only need to consider characters $\lambda$ with $\lambda|_{Z(F)_1}=\chi^{-1}|_{Z(F)_1}$.
    \item The pair $(\phi_1,f_1)$ are associated if and only if $(\bar{\phi},\bar{f})$ are associated.
\end{enumerate}
\end{lem}
\begin{proof}
We follow the proof of Lemma 5.3.1 in \cite{haines2009base}. Suppose that $\delta\in H(E)$ (resp. $\gamma\in H(F)$), and we write $\bar{\delta}:=p(\delta)$ (resp. $\bar{\gamma}:=p(\gamma)$). We have the following formulas for all $\chi$:
\begin{equation*}\tag{4.3.2}\label{4.3.2}
\begin{split}
SO^H_{\delta\sigma}(\phi_{\lambda N}) & := \sum_{\delta'}e(\delta')\,TO^H_{\delta'\sigma}(\phi_{\lambda N})  \\
& = \sum_{\delta'}e(\delta')\,\int_{H_{\delta'\sigma}\backslash H(E)}\phi_{\lambda N}(h^{-1}\delta'\sigma(h))d \bar{h}\\
& = 
\sum_{\delta'}e(\delta')\,\int_{H_{\delta'\sigma}\backslash H(E)}\int_{\overline{Z(E)}}\phi(h^{-1}(\delta'v)\sigma(h))\,\lambda^{-1} (Nv)dv\, d \bar{h}\\
& = \int_{\overline{Z(E)}}\Bigl[\sum_{\delta'}e(\delta')\,\int_{H_{\delta'\sigma}\backslash H(E)}\phi(h^{-1}(\delta'v)\sigma(h))\,d\bar{h}\Bigr]\lambda^{-1} (Nv)dv\\
& = 
\int_{\overline{Z(E)}}\lambda^{-1} (Nv) \,SO^H_{v\delta,\sigma}(\phi)\,dv.
\end{split}
\end{equation*}
and similarly we have 
\begin{equation*}\tag{4.3.3}\label{4.3.3}
SO^H_\gamma(f_\lambda)=\int_{Z(F)}\lambda^{-1}(z)\,SO^H_{z\gamma}(f)\,dz    
\end{equation*}
Then it is clear that if $\phi$ and $f$ are associated, then $\phi_{\lambda N}$ and $f_\lambda$ are associated for every $\lambda$. For the inverse, we apply the Fourier inversion formula to get $SO^H_{v\delta,\sigma}(\phi)=SO^H_{z\gamma}(f)$ for $v\in Z(E)$ and $z\in Z(F)$ such that $Nv=z$. (as functions of $v$ and $z$, resp.) In particular, let $v=1$, we have 
\begin{equation*}\tag{4.3.4}\label{4.3.4}
SO^H_{\delta\sigma}(\phi)=SO^H_{\gamma}(f)   \end{equation*}

For (ii), we assume $f\in\mathcal{H}(H(F), K, \rho)$, then for any $i\in K\cap Z(F)$ and $h\in H(F)$, we have 
\begin{equation*}
\begin{split}
f_\lambda(hi)&=\int_{Z(F)}f(hiz)\lambda^{-1}(z)\,dz\\
&=\rho^{-1}(i)f_\lambda(h),
\end{split}    
\end{equation*}
also we have 
\begin{equation*}
\begin{split}
f_\lambda(hi)&=\int_{Z(F)}f(hiz)\lambda^{-1}(z)\,dz\\
&=\lambda(i)\int_{Z(F)}f(hu)\lambda^{-1}(u)\,du=\lambda(i)f_\lambda(h)
\end{split}    
\end{equation*}
Therefore either $\lambda|_{Z(F)\cap K}=\rho^{-1}|_{Z(F)\cap K}$ or $f_\lambda$ vanishes completely on $H$.

Part (iii) follows from part (ii) by taking $K=I_H, \Tilde{K}=\Tilde{I}_H$ and $\Tilde{\rho}=\Tilde{\rho}_H$. 

Finally, we prove part (iv). For $\lambda=1$, we claim that for $\gamma=\mathcal{N}\delta$, we have
\begin{equation*}\tag{4.3.5}\label{4.3.5}
\begin{split}
SO^H_{\delta\sigma}(\phi_1) &=SO^G_{\bar{\delta}\sigma}(\bar{\phi})\\
SO^H_{\gamma}(f_1) &=SO^G_{\bar{\gamma}}(\bar{f})
\end{split}
\end{equation*}
Indeed, we have by definition
\begin{equation*}
SO^H_{\delta\sigma}(\phi_1):= \sum_{\delta'}e(\delta')\,TO^H_{\delta'\sigma}(\phi_1)   
\end{equation*}
and 
\begin{equation*}
SO^G_{\bar{\delta}\sigma}(\bar{\phi}):= \sum_{\bar{\delta'}}e(\bar{\delta'})\,a(\bar{\delta'})TO^G_{\bar{\delta'}\sigma}(\bar{\phi})    
\end{equation*}

We compare those two equalities term by term. First we notice that $p: H_{\delta'\sigma}(F)\to G^\circ_{\bar{\delta'}\sigma}(F)$ with kernel $Z(E)\cap H_{\delta'\sigma}(F)=Z(F)$, hence $e(\delta'):=e(H_{\delta'\sigma})=e(G_{\bar{\delta'}\sigma})=:e(\bar{\delta'})$ by the Corollary on Page 295 of \cite{kottwitz1983sign}. Recall that $a(\bar{\delta'}):=|\ker[H^1(F, G^\circ_{\bar{\delta'}\sigma})\to H^1(F, G_{\bar{\delta'}\sigma})]|$. We claim that $p$ induces a surjective map from the set
\begin{equation*}
\{\sigma\text{-conjugacy classes }\delta'\in H(E)\text{ stably }\sigma\text{-conjugate to }\delta\} \end{equation*}
onto the set
\begin{equation*}
\{\sigma\text{-conjugacy classes }\bar{\delta'}\in G(E)\text{ stably }\sigma\text{-conjugate to }\bar{\delta}\} 
\end{equation*}
with the fiber over the class of $\bar{\delta'}$ identified with the set
\begin{equation*}
\ker[H^1(F, G^\circ_{\bar{\delta'}\sigma})\to H^1(F, G_{\bar{\delta'}\sigma})]    
\end{equation*}
Indeed, we denote $\Tilde{G}:=R_{E/F}G_E$, and similarly $\Tilde{H}$. The set of $\sigma$-conjugacy classes in $G(E)$ which are stably $\sigma$-conjugate to $\bar{\delta}$ corresponds to the image of
\begin{equation*}
\ker[H^1(F, G^\circ_{\bar{\delta}\sigma})\to H^1(F,\Tilde{G
})]    
\end{equation*}
in
\begin{equation*}
\ker[H^1(F, G_{\bar{\delta}\sigma})\to H^1(F,\Tilde{G
})].    
\end{equation*}
Meanwhile the set of $\sigma$-conjugacy classes in $H(E)$ which are stably $\sigma$-conjugate to $\delta$ corresponds to the set 
\begin{equation*}
\ker[H^1(F, H_{\delta\sigma})\to H^1(F,\Tilde{H
})]    
\end{equation*}
since $H_{\delta\sigma}=H^\circ_{\delta\sigma}$. This is explained in \cite{kot82}, Pages 805-806. The map $p$ induces a bijection
\begin{equation*}
\ker[H^1(F, H_{\delta\sigma})\to H^1(F,\Tilde{H
})]\xrightarrow{\sim}\ker[H^1(F, G^\circ_{\bar{\delta}\sigma})\to H^1(F,\Tilde{G
})].        
\end{equation*}
hence the claim is proved. Finally, we have to show that for all $\sigma$-conjugacy classes $\delta'$ in the fiber over the class of $\bar{\delta}'$,  
\begin{equation*}
TO^H_{\delta'\sigma}(\phi_1)=TO^G_{\bar{\delta'}\sigma}(\bar{\phi})    
\end{equation*}
From the calculations in (\ref{4.3.1}), we have 
indeed:
\begin{equation*}
\begin{split}
TO^H_{\delta'\sigma}(\phi_1) 
&=\int_{(Z(E)H_{\delta'\sigma})\backslash H(E)}\phi_1(h^{-1}\delta'\sigma(h)) \,d\bar{h}\\
&= \int_{G^\circ_{\bar{\delta'}\sigma}\backslash G(E)}\bar{\phi}(g^{-1}\bar{\delta'}\sigma(g))\,d\bar{g}\\
&= TO^G_{\bar{\delta'}\sigma}(\bar{\phi}).
\end{split}    
\end{equation*}
Where the second equality comes from the fact $Z(E)H_{\delta'\sigma}=p^{-1}(G^\circ_{\bar{\delta'}\sigma})$ and $H(E)/Z(E)=G(E)$ since $Z$ is a product of induced tori, by Lemma \ref{lem. surjection of centralizers} and the following remark. Therefore we have $SO^H_{\delta\sigma}(\phi_1) =SO^G_{\bar{\delta}\sigma}(\bar{\phi})$, and similarly $SO^H_{\gamma}(f_1) =SO^G_{\bar{\gamma}}(\bar{f})$ (using the analogous $H_{\gamma'}(F)=p^{-1}(G^\circ_{\bar{\gamma'}})$), thus (iv) is proved.
\end{proof}

\begin{lem}\label{lem 4.5}
The map $\phi\mapsto\bar{\phi}$ determines a surjective homomorphism $\mathcal{Z}(H(E), \Tilde{\rho}_H)\to\mathcal{Z}(G(E), \Tilde{\rho})$. It is compatible with the base change homomorphism in the sense that
\begin{equation*}
b(\bar{\phi})=\overline{b(\phi)}.    
\end{equation*}
\end{lem}
\begin{proof}
The proof is almost the same as in \cite{haines2009base} Lemma 5.3.2, but instead of the Bernstein isomorphism $B$ there, we use $\beta:\C[\mathfrak{X}_{\mathfrak{s}^H_E}]\xrightarrow{\sim}\mathcal{Z}(H(E), \Tilde{\rho}_H)$ (also $\beta:\C[\mathfrak{X}_{\mathfrak{s}_E}]\xrightarrow{\sim}\mathcal{Z}(G(E), \Tilde{\rho})$). Here $\mathfrak{s}^H_E:=\mathfrak{s}_{\chi_H N}$. Recall that $\xi:T(F)\to\C^\times$ is a smooth character extending some $W(F)$-conjugate of the depth zero character $\chi:T(F)_1\to\C^\times$. We can extend $\xi$ to a character $\xi_H$ on $T_H(F)$, which extends $\chi_H$.
Therefore the morphism of algebraic varieties
\begin{equation*}
\begin{split}
p^*:\mathfrak{X}_{\mathfrak{s}_E}&\to\mathfrak{X}_{\mathfrak{s}^H_E}\\
(T(E), \xi_N)_{G(E)}&\mapsto(T_H(E), \xi_H\circ N)_{H(E)},
\end{split}
\end{equation*}
gives rise to the homomorphism $p:\C[\mathfrak{X}_{\mathfrak{s}^H_E}]\to\C[\mathfrak{X}_{\mathfrak{s}_E}]$ that we need as in the proof in loc. cit.
\end{proof}

\subsection{The case where the central character is not unitary}\label{sec 4.4}
In applying Lemma \ref{lem 4.4}(i), we need the following lemma which will enable us to assume that $\lambda$ is a unitary character on $Z$. It is not necessarily the case that $\lambda$ is unitary in our reduction steps. Nonetheless, we can always twist it by a global character as follows.
\begin{lem}\label{lem 4.6}
If $(\pi,V)$ is any smooth $\omega$-representation of $G(F)$, then there exists a
unique positive real-valued character $\chi$ on $G(F)$ such that the restriction of $\pi\otimes\chi$ to $Z(F)$
is unitary, and the restriction of $\chi$ to $Z(F)_1$ is trivial.  
\end{lem}
\begin{proof}
This is Lemma 5.2.5 in \cite{casselman1995introduction}. We notice that by construction in the proof, $\chi$ is trivial on the Iwahori subgroup $I\subset G$.   
\end{proof}
\begin{lem}\label{lem 4.7}
Assume $\phi\in\mathcal{Z}(H(E),\tilde{\rho})$ and let $f:=b\phi\in\mathcal{Z}(H,\rho)$. If $\phi_{\lambda N}$ and $f_\lambda$ are associated for all unitary central characters $\lambda:Z(F)\to\C^\times$ and all $\phi\in\mathcal{Z}(H(E),\tilde{\rho})$, then $\phi$ and $f$ are associated.    
\end{lem}
\begin{proof}
By Lemma \ref{lem 4.4} (i) and (iii), the functions $\phi$ and $f$ are associated if $\phi_{\lambda N}$ and $f_\lambda$ are associated for every character $\lambda$ such that $\lambda|_{Z(F)_1}=\chi^{-1}|_{Z(F)_1}$. We claim that it is enough to prove the results for such \textbf{unitary} characters $\lambda$. Indeed, by the previous lemma, we can find positive characters $\eta$ on $G$, such that the set $\{\lambda\eta\}$ for unitary $\lambda$ contains all the characters we need in Lemma \ref{lem 4.4} (i) and (iii). 

 We consider the function $f_{\lambda\eta}$, where we understand that we actually mean $\lambda\eta|_Z$ in the definition. 
Now we have 
\begin{equation*}
\begin{split}
f_{\lambda\eta}(g)&=\int_{Z(F)}f(gz)\lambda^{-1}(z)\eta^{-1}(z)\,dz\\
&=\eta(g)\int_{Z(F)}f(gz)\eta^{-1}(gz)\lambda^{-1}(z)\,dz\\
&=\eta(g)(\eta^{-1}f)_\lambda(g).
\end{split}    
\end{equation*}
Therefore, we have
\begin{equation*}\tag{4.4.1}\label{4.4.1}
f_{\lambda\eta}=\eta\cdot(\eta^{-1}f)_\lambda,
\end{equation*}
Similarly, we have over $H(E)$,
\begin{equation*}\tag{4.4.2}\label{4.4.2}
\phi_{\lambda\eta N}=(\eta N)\cdot ((\eta^{-1} N)\phi)_{\lambda N}    
\end{equation*}
Now we assume $\phi\in\mathcal{Z}(H(E),\tilde{\rho})$ and $f=b\phi\in\mathcal{Z}(H,\rho)$. By the lemma below, $\eta^{-1}f$ and $(\eta^{-1}N)\phi$ are still in the center of respective Hecke algebras and we have $b((\eta^{-1}N)\phi)=\eta^{-1}\cdot b\phi=\eta^{-1}f$. Therefore, by assumption, $(\eta^{-1}f)_\lambda$ and $((\eta^{-1} N)\phi)_{\lambda N}$ are associated. Then by a straightforward calculation, we see that the left hand sides of (\ref{4.4.1}) and (\ref{4.4.2}) are associated, as desired.
\end{proof}

\begin{lem}\label{lem 4.8}
Let $\phi\in\mathcal{Z}(H(E),\rho)$ and $\eta$ be a character on $H$ such that $\eta|_I=\text{triv}$, where $I\subset H$ is the Iwahori subgroup where $\rho$ is defined on. Then $(\eta N)\cdot \phi\in\mathcal{Z}(H(E),\rho)$ and $b((\eta N)\cdot \phi)=\eta\cdot b\phi$.   
\end{lem}
\begin{proof}
Let $\psi\in \mathcal{H}(H(E),\rho)$. Since $\eta$ is trivial on $I$, it is clear that $(\eta N)\cdot\phi\in\mathcal{H}(H(E),\rho)$
then we have 
\begin{equation*}
\begin{split}
[(\eta N)\cdot \phi]\ast\psi(h)&=\int_{H(E)}\eta N(g)\phi(g)\psi(g^{-1}h)\,dg\\
&=\eta N(h)\int_{H(E)}\eta N(g^{-1})\phi(hg^{-1})\psi(g)\,dg\\
&= \eta N(h)[\phi\ast((\eta^{-1}N)\cdot\psi)](h).
\end{split}    
\end{equation*}
Therefore we have 
\begin{equation*}
\begin{split}
[(\eta N)\cdot \phi]\ast\psi&= \eta N\cdot[\phi\ast((\eta^{-1}N)\psi)]\\
&=\eta N\cdot[((\eta^{-1}N)\psi)\ast\phi]\,(\phi\text{ is central})\\
&=\psi\ast[(\eta N)\cdot \phi]\,(\text{use the calculation above reversely}).
\end{split}
\end{equation*}
The centrality is shown. Using the identities (\ref{4.2.1}) and (\ref{3.3.3}), we have 
\begin{equation*}
\begin{split}
ch_{\xi^{-1}}(b((\eta N)\cdot \phi))&=ch_{\xi^{-1}_N}((\eta N)\cdot\phi)\\
&=ch_{(\xi\eta)^{-1}_N}(\phi)\\
&=ch_{(\xi\eta)^{-1}}(b\phi)\\
&=ch_{\xi^{-1}}(\eta\cdot b\phi)
\end{split}    
\end{equation*}
therefore $b((\eta N)\cdot \phi)=\eta\cdot b\phi$, as desired.
\end{proof}

\subsection{The case where the center is not an induced torus}\label{sec 4.5}
We assume $\gamma\in G(F)$ is a regular elliptic semisimple element, and $G_{\text{der}}=G_{\text{sc}}$. We will show that there is an exact sequence
\begin{equation*}
1\longrightarrow G\longrightarrow G'\longrightarrow Q\longrightarrow 1,    
\end{equation*}
where $G'$ is an unramified group over $F$ with the properties that $G'_{\text{der}}=G'_{\text{sc}}$ and $Z(G')$ is an induced torus. We follow the construction in \cite{Clozel90}, 6.1(b). Indeed, since $G$ is unramified, $Z(G)$ is contained in a maximal unramified $F$-torus $T\subset G$. The group of characters $X^*(T)$ is a finitely generated $\Z[\Gal(F'/F)]$-module, where $F'$ is an unramified extension of $F$ splitting $G$. Let $P\twoheadrightarrow X^*(T)$ be a free $\Z[\Gal(F'/F)]$-module. Then dually we get an embedding $T\hookrightarrow Z'$ where $Z'$ is an induced torus. We have the following exact sequence
\begin{equation*}
1\longrightarrow Z(G_{\text{der}})\longrightarrow Z(G)\times G_{\text{der}}\longrightarrow G\longrightarrow 1,    
\end{equation*}
see \cite{Milne_2017}, Example 19.25. We take $G'$ to be $(G_\text{der}\times Z')/Z(G_\text{der})$. It is easy to check that $Z(G')=Z'$ and  $G'_\text{der}=G_\text{der}$, therefore $G'$ satisfies the 
conditions we need. 

The natural embedding $G\hookrightarrow G'$ induces a natural homomorphism
\begin{equation*}
j:\mathcal{Z}(G,\rho)\hookrightarrow\mathcal{Z}(G',\rho')    
\end{equation*}
by the following commutative diagram
\begin{equation*}
\begin{tikzcd}
\C[\mathfrak{X}_\mathfrak{s}] \arrow[r, "\sim","\beta"', rightarrow] \arrow[d,hookrightarrow]
& \mathcal{Z}(G,\rho)\arrow[d,hookrightarrow, "j"'] \\
\C[\mathfrak{X}_{\mathfrak{s}'}] \arrow[r, "\sim","\beta"', rightarrow] 
&  \mathcal{Z}(G',\rho')
\end{tikzcd}
\end{equation*}
where the left vertical map is induced by the surjective map
\begin{equation*}
\begin{split}
\mathfrak{X}_{\mathfrak{s}'}&\twoheadrightarrow \mathfrak{X}_{\mathfrak{s}}\\
(T',\xi')_{G'}&\mapsto (T:=T'\cap G, \xi'|_T)_G
\end{split}    
\end{equation*}
where $T'\subset G'$ is a maximal $F$-torus and $\xi':T'(F)\to\C^\times$ is a smooth character extending some $W'(F)$-conjugate of $\chi'$, where $W'$ is the Weyl group of $G'$ and $\chi':T'(F)_1\to\C^\times$ is an extension of $\chi$. 

The injective morphism $j$ is uniquely determined by the Bernstein isomorphisms and it is the restriction of the map between Hecke algebras
\begin{equation*}
\begin{split}
\mathcal{H}(G,\rho)&\rightarrow \mathcal{H}(G',\rho') \\  
1_{n_w}&\mapsto 1'_{n_w},
\end{split}
\end{equation*}
where $n_w$ is the extended affine Weyl group $\tilde{W}_H=X_*(A)\rtimes W^\circ_\chi$ and $1_{n_w}$ (resp. $1'_{n_w}$) is the function in $\mathcal{H}(G,\rho)$ (resp., $\mathcal{H}(G',\rho')$) supported on $I_r n_w I_r$ (resp., $I'_r n_w I'_r$), whose value at $n_w$ is $1$. Such functions $1_{n_w}$ form a basis of the Hecke algebra $\mathcal{H}(G,\rho)$. Here we are using the notations from \S7.3, \cite{haines2012base}. Li (\cite{li2026base}, Lemma 7.2.3) shows that the morphism $j$ on the centers is indeed the restriction of the linear transformation on the whole Hecke algebras, therefore, it suffices to prove the following: 

\begin{lem}\label{lem 4.9}
There exists a constant $C_T$ such that, for every $\delta\in G(E)$ with elliptic regular norm in $T$, and for every $w\in\Tilde{W}_E$, we have
\begin{equation*}\tag{4.5.1}\label{4.5.1}
SO^{G(E)}_{\delta\sigma}(1_{n_w})= C_TSO^{G'(E)}_{\delta\sigma}(1'_{n_w}) 
\end{equation*}
\end{lem}
\begin{proof}
This is proved in \cite{haines2012base}, Lemma 7.3.1. The proof also works for positive characteristic.    
\end{proof}

From (\ref{4.5.1}), we have the following equalities: 
\begin{equation*}\tag{4.5.2}\label{4.5.2}
\begin{split}
SO^G_\gamma(f,dt,dg)&=c_T SO^{G'}_\gamma(jf,dt',dg')\\
SO^{G(E)}_{\delta\sigma}(\phi,dt,dg_E)&=C_T SO^{G'(E)}_{\delta\sigma}(j\phi,dt',dg_E'),
\end{split}    
\end{equation*}
for all functions $f\in\mathcal{Z}(G,\rho)$ and $\phi\in\mathcal{Z}(G(E),\rho_E)$ and for all ($\sigma$-)regular ($\sigma$-)elliptic elements $\gamma$ and $\delta$ in $G$ whose ($\sigma$-)centralizer is the elliptic 
$F$-torus $T$. The proof of Lemma 7.3.1 in \textit{loc. cit.} shows that the constants $c_T$ and $C_T$ making (\ref{4.5.2}) hold depend only on the torus $T$ and choices of measures (so not on the functions $f$ and $\phi$, and the choice of $\gamma=\mathcal{N}(\delta)$). Therefore, to force $c_T= C_T$, we use the fact that $1_{I_r}$ and $1_{I}$ are associated (\cite{kottwitz1986base}, \S1\footnote{Kottwitz\cite{kottwitz1986base} showed that the functions $1_{K_E}$ and $1_K$ are associated for any open bounded subgroup $K_E$ (resp., $K$) of $G(E)$ (resp., $G(F)$) satisfying certain assumptions. In particular, $K$ (resp. $K_E$) satisfies those assumptions when it is the $\mathcal O_F$-points (resp. $\mathcal{O}_E$-points) of a smooth and connected affine group scheme over $\mathcal O_F$ (resp. $\mathcal{O}_E$).}) and the stable orbital integrals of $1_I$ do not vanish identically on any torus $T$ in $G$.

\subsection{Summary of reduction steps}\label{sec 4.6}
By Lemma \ref{lem vanishing non-norm}, we may assume that $\gamma$ is a norm, and we write $\gamma=\mathcal{N}\delta$. We assume $\gamma$ is \textit{regular semisimple} from the beginning. Notice that it is different from the assumptions in \cite{haines2009base, haines2012base, Clozel90} of just being semisimple since we do not have  \textit{(twisted) Shalika germs} for local function fields and the homogeneity that were used in the proof of Proposition 7.2 in \cite{Clozel90} to reduce to regular semisimple elements. The reduction steps (1) and (2) are the same as in \cite{haines2009base}, 5.4. However, we need to pass it to the global setup afterwards.
\begin{enumerate}[(1)]
    \item \textit{We may assume that} $G_\text{der}=G_\text{sc}$. Indeed, for given $G$, we take $H$ to be a $z$-extension as in \ref{sec 4.3}, and let $\phi\in\mathcal{Z}(H(E), \tilde{\rho}_H)$ with $b\phi\in\mathcal{Z}(H, {\rho}_H)$. Assume that $(\phi, b\phi)$ are associated, then by Lemma \ref{lem 4.4}(i) and (iv), $(\overline{\phi}, \overline{b\phi})$ are associated. By Lemma \ref{lem 4.5}, $(\bar{\phi}, b\bar{\phi})$ are associated and $\bar{\phi}$ ranges over all functions in $\mathcal{Z}(G(E),\tilde{\rho})$, therefore the base change fundamental lemma is proved for $G$.
    \item \textit{We may assume that $\gamma$ is elliptic.} This is explained in \cite{haines2009base}, \S4 and \cite{haines2012base}, \S6.

    From here we pass to the global setup momentarily. We may assume that $G$ is split over an unramified extension $K/F$ such that $E\subset K$. We may also assume $G$ and $\gamma$ satisfies the conditions (1) and (2). Choose a degree $[K:F]$ cyclic extension of function fields $\underline{K}/\underline{F}$ and a finite place $v_0$ of $\underline{F}$ such that $\underline{K}_{v_0}$ is a field and $\underline{K}_{v_0}/\underline{F}_{v_0}\cong K/F$. Then there is a degree $r=[E:F]$ cyclic extension $\underline{E}/\underline{F}$ with $\underline{E}\subset\underline{K}$ and $\underline{E}_{v_0}/\underline{F}_{v_0}\cong E/F$.

    There is a quasi-split group $\underline{G}$ over $\underline{F}$ with the property that $\underline{G}\times_{\underline{F}}\underline{F}_{v_0}\cong G$. Let $\theta$ denote the $\underline{F}$-linear automorphism of $\Tilde{\underline{G}}$ from $\Gal(E/F)\cong \Gal(\underline{E}/\underline{F}).$. We may repeat the process of finding a $G\hookrightarrow G'$ at the beginning of \ref{sec 4.4}, but globally: a global group $\underline{G'}$ such that $\underline{G}'_{\text{der}}=\underline{G}_{\text{der}}$ and $Z(\underline{G}')$ is an induced torus. Therefore after base change to $v_0$, we still have the desired properties for $G$ and $G':=\underline{G'}_{v_0}$ as in \ref{sec 4.4}. Therefore we have the following reduction step:
    \item \textit{We may assume that} $G_\text{der}=G_\text{sc}$ \textit{and $Z(G)$ is an induced torus which comes from a global induced torus}. Indeed we work under assumptions (1) and (2), and Lemma \ref{lem 4.9} and the discussion there show that we may assume $Z(G)$ is indeed an globally induced torus.
    \item \textit{We may assume that $\gamma$ belongs to a specific dense subset in the set of regular elliptic elements.} This is explained in \cite{Clozel90} Lemma 6.7. Recall that we call a regular semisimple element $\gamma$ \textit{strongly regular semisimple} if $G_\gamma$ is a (connected) torus. Later in Lemma \ref{lem equality of integrals between groups}, we will see that $\gamma\in G(F)$ will be assumed to be strongly regular semisimple, whose image in the adjoint group is still strongly regular. By Lemma 8.7, we will indeed see that such $\gamma$ will form a dense subset of the set of regular elliptic elements in $G(F)$. Notice that we don't need to worry about the $a(\delta')$ terms as in the reduction steps of \cite{haines2009base} since our group $G$ has simply connected derived group.
\end{enumerate}
\begin{con}\label{conclusion of reduction steps}
We may assume that $G_{\mathrm{der}}=G_{\rm{sc}}$ with $Z(G)$ an globally induced torus, and $\gamma$ is a strongly regular elliptic element itself and in the adjoint group, and is a norm.    
\end{con}

\begin{rem}\label{rem of haines12 reduction steps}
The reduction steps in \cite{haines2012base} have flaws in that Lemma 7.2.3 in \textit{loc. cit.} is not valid for depth-zero Hecke algebras. Therefore, we cannot reduce to the case that $G$ is adjoint. The reduction steps here present a solution in order to avoid reducing to adjoint groups.  
\end{rem}
We hope to show the associations of functions $\phi\in\mathcal{Z}(H(E),\tilde{\rho})$ and $b\phi$ at all such elements $\gamma$ over the group $G$. The strategy of the proof in the following sections is that, using Lemma \ref{lem 4.4} (i) and Lemma \ref{lem 4.7}, we will show that $\phi_{\lambda N}$ and $(b\phi)_\lambda$ are associated for all unitary central characters and all $\phi\in\mathcal{Z}(H(E),\tilde{\rho})$, at all strongly regular elliptic semisimple elements in $H$ whose images in the adjoint group is still strongly regular. 

Using the existence of \textit{local data} adapted to the case with unitary central characters (\S \ref{sec 7}), we are able to show the desired association. In order to produce the adapted local data, one needs to use the global (twisted) trace formulas (\S \ref{sec 5}) and one needs to \textit{stabilize} the trace formulas (\S \ref{sec 6}) for them to be useful. 

\section{The Simple Trace Formula}\label{sec 5}

\subsection{Setup}\label{sec 5.1}
This section is mostly independent of other sections. We do not make any assumptions on the characteristics of the global field $F$.
Our main references for the simple (twisted) trace formula are \cite{arthur1989simple} \S 1.2 and \cite{DKV} A.1.

Assume that $E/F$ is a cyclic extension of global fields of degree $r$, we write $\A$ (resp. $\A_E$) for the adeles of $F$ (resp. $E$). We denote $\Gamma:=\Gal(E/F)$ and $\theta\in\Gamma$ be a generator. Let $G$ be a connected reductive group defined over $F$, and let $[G]_E:=G(E)\backslash G(\A_E)$ denote the usual \textit{adelic quotient} over $E$. We consider the left action of $G(\A_E)$ on $L^2(G(E)\backslash G(\A_E))$ given by
\begin{equation*}\tag{5.1.1}\label{5.1.1}
(R(g)f)(x):=f(xg)    
\end{equation*}
for any $g\in G(\A_E), f\in L^2(G(E)\backslash G(\A_E))$ and $x\in G(E)\backslash G(\A_E)$.

Let $Z=Z(G)$ be the center of $G$. By the reduction steps, we can assume that $Z$ is an induced torus over $F$, therefore, we have $G(R)/Z(R)=G_{\rm ad}(R)$ for any $F$-algebra $R$, where $G_{\rm ad}:=G/Z$. 
Fix a \textbf{unitary} character $\lambda: Z(\A_F)/Z(F)\to\C^\times$ and define $\lambda N: Z(\A_E)\to\C^\times$ in the usual way. 

We consider the space $C^\infty_c(Z(\A_E)\backslash G(\A_E), \lambda^{-1})$ of locally constant functions whose supports are compact modulo $Z(\A_E)$ and transformed by $\lambda ^{-1}N$ under the action of the center $Z$. We will denote it by $C^\infty_c(G_{\rm ad}(\A_E), \lambda^{-1})$ with $Z(\A_E)$ understood in the definition. Similarly, we consider the space $L^2(G_{\rm ad}(E)\backslash G_{\rm ad}(\A_E), \lambda)=L^2(G_{\rm ad}(E)\backslash G_{\rm ad}(\A_E), \lambda)$ of functions on $G(E)\backslash G(\A_E)$ that are square-integrable on $G_{\rm ad}(E)\backslash G_{\rm ad}(\A_E)$ which transform by $\lambda$ under $Z$. Notice that we need $\lambda$ to be unitary in order to define the integrability over the quotient.  

Now for any $\phi\in C^\infty_c(G_{\rm ad}(\A_E), \lambda^{-1})$, the above action (\ref{5.1.1}) induces an action of $\phi$ on \[L^2(G_{\rm ad}(E)\backslash G_{\rm ad}(\A_E), \lambda),\] given by
\begin{equation*}\tag{5.1.2}\label{5.1.2}
(R(\phi)f)(x):=\int_{G_{\rm ad}(\A_E)}\phi(g)R(g)f(x)\,dg=\int_{G_{\rm ad}(\A_E)}\phi(g)f(xg)\,dg.
\end{equation*}
for any $f\in L^2(G_{\rm ad}(E)\backslash G_{\rm ad}(\A_E), \lambda)$.

\begin{defn}\label{defn of cuspidal subspace}
We define the \textit{cuspidal subspace} 
\[L^2_{\mathrm{cusp}}(G_{\rm ad}(E)\backslash G_{\rm ad}(\A_E), \lambda)\subset L^2(G_{\rm ad}(E)\backslash G_{\rm ad}(\A_E), \lambda)\] to be the space of functions $\phi\in L^2(G_{\rm ad}(E)\backslash G_{\rm ad}(\A_E), \lambda)$ such that, for every parabolic subgroup $P\subset G$ with unipotent radical $N$, one has 
\begin{equation*}
\int_{[N]}\phi (ng)\,dn=0   
\end{equation*}
for all $g\in G(\A_E)$. We call such $\phi$ \textit{cuspidal functions}.  
\end{defn}

We also have the operator
\begin{align*}
I_\theta: L^2(G_{\rm ad}(E)\backslash G_{\rm ad}(\A_E), \lambda)&\rightarrow L^2(G_{\rm ad}(E)\backslash G_{\rm ad}(\A_E), \lambda),\\
f&\mapsto(x\mapsto f(\theta^{-1}(x))),
\end{align*}
that preserves $L^2_{\text{cusp}}(G_{\rm ad}(E)\backslash G_{\rm ad}(\A_E), \lambda)$.
\begin{prop}\label{prop kernel}
For each $\phi\in C^\infty_c(G_{\rm ad}(\A_E), \lambda^{-1})$, the composite
\[
R(\phi)\circ I_\theta: L^2(G_{\rm ad}(E)\backslash G_{\rm ad}(\A_E), \lambda)\to L^2(G_{\rm ad}(E)\backslash G_{\rm ad}(\A_E), \lambda)
\]
has kernel function 
\[
K_\phi(g,\theta(h)):=\sum_{\delta\in G_{\rm ad}(E)}\phi(g^{-1}\delta\theta(h)),
\]
\end{prop}
In other words, we have 
\begin{equation*}
\begin{split}
(R(\phi)\circ I_\theta(f))(g)&=\int_{G_{\rm ad}(E)\backslash G_{\rm ad}(\A_E)}K_\phi(g,\theta(h)) f(h)\,dh\\
&=\int_{G_{\rm ad}(E)\backslash G_{\rm ad}(\A_E)} \sum_{\delta\in G_{\rm ad}(E)}\phi(g^{-1}\delta\theta(h))f(h)\,dh
\end{split}    
\end{equation*}
For any $f\in L^2(G_{\rm ad}(E)\backslash G_{\rm ad}(\A_E),\lambda)$ and $g\in G_{\rm ad}(\A_E)$.

\begin{proof}
Indeed, since 
\begin{equation*}
\begin{split}
(R(\phi)\circ I_\theta(f))(g)&=\int_{G_{\rm ad}(\A_E)}\phi(h)I_\theta(f)(gh)\,d\bar{h}\\
&=\int_{G_{\rm ad}(\A_E)}\phi(h)f(\theta^{-1}(g)\theta^{-1}(h))\,d\bar{h}\\
&=\int_{G_{\rm ad}(\A_E)}\phi(g^{-1}\theta(h))f(h)\,d\bar{h}\\
&=\int_{G_{\rm ad}(E)\backslash G_{\rm ad}(\A_E)}\sum_{\delta\in  G_{\rm ad}(E)}\phi(g^{-1}\delta\theta(h))f(h)\,d\bar{h}
\end{split}    
\end{equation*}
where the third line follows from change of variables, and the last line follows from a well-known technique called \textbf{unfolding} or \textbf{integration in stages} and the fact that $f$ is invariant under the action of elements in $G_{\rm ad}(E)$. We also use the fact that the set $G_{\rm ad}(E)$ is $\theta$-stable.
For completeness we include the unfolding lemma below. For a proof, see \cite{getz2024introduction} Theorem 3.2.2 and Lemma 9.2.4.   
\end{proof}

\begin{lem}\label{lem unfolding}
Suppose that $G$ is a Hausdorff, locally compact, second countable topological group with right Haar measure $dg$. If $f\in L^1(G)$ and $\Gamma\leq G$ is a discrete subgroup such that the modular character of $G$ is trivial on $\Gamma$, then we have
\begin{equation*}
\int_{G/\Gamma}\sum_{\gamma\in\Gamma}f(\gamma g)\,dg=\int_Gf(g)\,dg.    
\end{equation*}
\end{lem}
\begin{lem}\label{lem trace class}
The linear operator $R_{\mathrm{cusp}}(\phi)\circ I_\theta$ is of trace class on $L^2_{\mathrm{cusp}}(G_{\rm ad}(E)\backslash G_{\rm ad}(\A_E), \lambda)$.  
\end{lem}
\begin{proof}
Since $R_{\text{cusp}}(\phi):=R(\phi)|_{L^2_{\text{cusp}}(G_{\rm ad}(E)\backslash G_{\rm ad}(\A_E), \lambda)}$ is of trace class by Theorem 9.1.1 of \cite{getz2024introduction}, the same is true of $R_{\text{cusp}}(\phi)\circ I_\theta$, since the set of trace class operators forms a two-sided ideal in the algebra of bounded linear operators on $L^2_{\text{cusp}}(G_{\rm ad}(E)\backslash G_{\rm ad}(\A_E), \lambda)$ (see \cite{deitmar2014principles}, Lemma 5.3.4).
\end{proof}

\subsection{Proof of the Simple Trace Formula}\label{sec 5.2}
Let us choose two places $v_1, v_2$ of $F$ which split completely in $E$. We put the following assumptions on the \textit{test function} $\phi\in C^\infty_c(G_{\rm ad}(\A_E), \lambda^{-1})$:

\begin{enumerate}[(1)]
    \item The adelic function $\phi$ is a (pure) tensor product of local functions $\phi=\prod\phi_v$ over \textbf{places $v$ of $F$}, where $\phi_v\in C^\infty_c(G(E_v)/Z(E_v),\lambda^{-1})$. For almost all places $v$ of $F$, $\phi_v$ is invariant under $G(\mathcal{O}_{E_v})$ and supported on $Z(E_v)G(\mathcal{O}_{E_v})$, and satisfies
    \begin{equation*}
    \int_{Z(E_v)\backslash Z(E_v)G(\mathcal{O}_{E_v})}\phi(g)\,dg=1.    
    \end{equation*}
    \item On $G(E_{v_1})\cong G(E_{w_1})\times G(E_{w_2})\times\dots\times G(E_{w_r})$, we have $\phi_{v_1}=(\phi^1_{w_1},\dots,\phi^1_{w_r})$ where each $\phi^1_{w_i}$ is a matrix coefficient of the same supercuspidal representation $\pi$ of $G(E_{w_i})\cong G(F_{v_1})$. This is possible since by the work of A. Kret\cite{Kret}, we know that supercuspidal representations exist for any reductive group $G$ over an non-archimedean local field $F$. Twisted by a character on the whole group by Lemma \ref{lem 4.6}, we may assume that the central character is unitary, and the representation is still supercuspidal since the matrix coefficients are still compactly supported modulo the center.

    \item Let $\phi_{v_2}=(\phi^2_{u_1},\dots,\phi^2_{u_r})$ be the analogous decomposition at $v_2$. Let $\Omega_i=\text{Supp}(\phi^2_{u_i})$, then $\Omega_1\Omega_2\dots\Omega_r$ is contained in the set of elements of $G(F_{v_2})$ with strongly regular elliptic image in $G_{\rm ad}(F_{v_2})$.
\end{enumerate}

We first show that the assumption (2) implies that the image of $R(\phi)$ is in the space of cusp forms.
\begin{lem}\label{lem cuspidal image}
Let $v$ be a finite place of $E$, let $f^v\in C^\infty_c(G(\A^v_E))$, and let $f_v\in C^\infty_c(G(E_v))$ be supercuspidal. Let $f(g)=f^v(g^v)f_v(g_v)$, so that $f\in C^\infty_c(G(\A_E))$. Then $R(f)$ has cuspidal image. 
\end{lem}
Recall that we say $f_v\in C^\infty_c(G(E_v))$ is \textit{supercuspidal} if 
\begin{equation*}
\int_{N(E_v)} f_v(gnh)\,dn=0,  
\end{equation*}
for all proper parabolic subgroup $P< G_{E_v}$ with unipotent radical $N$ and for all $g,h\in G(E_v)$. As the name suggests, matrix coefficients of supercuspidal representations are indeed supercuspidal, see \cite{getz2024introduction}, Lemma 16.4.2.
\begin{proof}
This is standard, for example see Lemma 16.4.1 in \cite{getz2024introduction}.
\end{proof}

And we have the main theorem of this section. We define the \textit{adelic twisted orbital integral} of $\phi\in C^\infty_c(G_{\rm ad}(\A_E), \lambda^{-1})$ to be 
\begin{equation*}
TO^{G_{\rm ad}(\A_E)}_{\delta\theta}(\phi)=\int_{G_{\rm ad,\delta\theta}(\A_E)\backslash G_{\rm ad}(\A_E)}\phi(g^{-1}\delta\theta(g))\,dg.
\end{equation*}
We notice that the above twisted orbital integral converges since $\phi$ is compactly supported on $G_{\rm ad}(\A_E)$.
\begin{thm}\label{Thm twisted trace formula}
Under the assumptions (1)(2)(3) above, the operator $R(\phi)$ sends $L^2$ automorphic forms into cusp forms, and 
\begin{equation*}\tag{5.2.1}\label{5.2.1}
\mathrm{tr}\,(R_{\mathrm{cusp}}(\phi)\circ I_\theta)=\sum_{\delta}\tau(G_{\rm ad, \delta\theta})TO^{G_{\rm ad}(\A_E)}_{\delta\theta}(\phi)
\end{equation*}
where $\delta$ runs over the $\theta$-conjugacy classes of elements of $G_{\rm ad}(E)$ with strongly elliptic regular norms, and the group $G_{\rm ad,\delta\theta}$ is the $\theta$-centralizer of $\delta$. Moreover $\tau(G_{\rm ad,\delta\theta})=\text{vol}\,(G_{\rm ad,\delta\theta}(E)\backslash G_{\rm ad,\delta\theta}(\A_E))$.
\end{thm}
\begin{proof}
From the lemma above we know that
\[
\tr(R_\text{cusp}(\phi)\circ I_\theta)=\tr(R(\phi)\circ I_\theta)
\]
and we obtain the latter trace by integrating along the diagonal $K_{\phi}(g,\theta(g))$ of the kernel associated to $R(\phi)\circ I_\theta$, whence
\begin{equation*}
\begin{split}
 \tr\,(R_{\text{cusp}}(\phi)\circ I_\theta)&=\int_{G_{\rm ad}(E)\backslash G_{\rm ad}(\A_E)}K_\phi(g,\theta(g))\,dg\\
 &=\int_{G_{\rm ad}(E)\backslash G_{\rm ad}(\A_E)}\sum_{\delta\in G_{\rm ad}(E)}\phi(g^{-1}\delta\theta(g))dg. 
\end{split}
\end{equation*}
\begin{lem}\label{lem 5.7}
Only those $\delta$ with $N\delta$ strongly regular elliptic in $G_{\rm ad}(E)$ will appear in the summation above.
\end{lem}
\begin{proof}
Same as in \cite{arthur1989simple}, P15.
\end{proof}

We need to swap the order of the integration and the summation. 
\begin{lem}\label{lem 5.8}
The function 
\begin{equation*}
F(g)=\sum_{\substack{\delta\in G_{\rm ad}(E),\\N\delta \mathrm{\,strongly\, elliptic \,regular}}}|\phi(g^{-1}\delta\theta(g))|
\end{equation*}
is compactly supported on $G(E)Z(\A_E)\backslash G(\A_E)$.
\end{lem}
\begin{proof}
Using Henniart (\cite{MSMF_1983_2_11-12__1_0}, Appendice 2), which is valid in positive characteristic, the proof of Lemma 2.6 in \cite{arthur1989simple} adapts.
\end{proof}
We regroup the summation by $\theta$-conjugacy by $G_{\rm ad}(E)$:
\[
\tr\,(R_{\text{cusp}}(\phi)\circ I_\theta)=\int_{G_{\rm ad}(E)\backslash G_{\rm ad}(\A_E)}\sum_{\substack{\delta\in G_{\rm ad}(E)\\N\delta \text{ strongly ell. reg.}\\\text{up to }G_{\rm ad}(E)\text{-}\theta\text{-}\text{conjugacy}}}\sum_{\substack{\delta'\in G_{\rm ad}(E)\\\delta' \sim \delta\text{ by } G_{\rm ad}(E)\text{-}\theta\text{-conjugacy}}}\phi(g^{-1}\delta'\theta(g))dg,
\]

which implies the final result of the theorem by the following manipulations of integration "in stages":

\begin{equation*}
\begin{split}
\tr\,(R_{\text{cusp}}(\phi)\circ I_\theta)&=\sum_{\delta}\int_{G_{\rm ad}(E)\backslash G_{\rm ad}(\A_E)}\sum_{\substack{\delta'\in G_{\rm ad}(E)\\\delta' \sim \delta\text{ by } G_{\rm ad}(E)\text{-}\theta\text{-conjugacy}}}\phi(g^{-1}\delta'\theta(g))dg\\
&=\sum_{\delta}\text{vol}(G_{\rm ad, \delta\theta}(E))^{-1}\int_{G_{\rm ad}(E)\backslash G_{\rm ad}(\A_E)}\int_{G_{\rm ad}(E)}\phi(g^{-1}v^{-1}\delta\theta(v)\theta(g))\,dv\,dg\\
&=\sum_{\delta}\text{vol}(G_{\rm ad, \delta\theta}(E))^{-1}\int_{G_{\rm ad}(\A_E)}\phi(g^{-1}\delta\theta(g))dg\\
&=\sum_{\delta}\text{vol}(G_{\rm ad, \delta\theta}(E))^{-1}\,\text{vol}(G_{\rm ad, \delta\theta}(\A_E))\int_{G_{\rm ad, \delta\theta}(\A_E)\backslash G_{\rm ad}(\A_E)}\phi(g^{-1}\delta\theta(g))dg\\
&=\sum_{\substack{\delta\in G_{\rm ad}(E)\\N\delta \text{ strongly ell. reg.}\\\text{up to }G_{\rm ad}(E)\text{-}\theta\text{-}\text{conjugacy}}}\text{vol}(G_{\rm ad, \delta\theta}(E)\backslash G_{\rm ad, \delta\theta}(\A_E))\,TO^{G_{\rm ad}(\A_E)}_{\delta\theta}(\phi)
\end{split}    
\end{equation*}
Thus the theorem is proved.
\end{proof}
In the special case of $E=F$, where $r=1$ and $\theta=\text{id}$, we denote $r(f)$ to be the representation of $f\in C^\infty_c(G_{\rm ad}(\A), \lambda^{-1})$ on the space $L^2(G_{\rm ad}(\A), \lambda)$, and choose $f$ similarly as in the twisted case, by the same process, we will have 
\begin{equation*}\tag{5.2.2}\label{5.2.2}
\tr\,(r_{\text{cusp}}(f))=\sum_{\gamma}\tau(G_{\rm ad, \gamma})O^{G_{\rm ad}(\A)}_\gamma(f),    
\end{equation*}
where the sum $\gamma$ is over the set of conjugacy classes of regular elliptic elements in $G_{\rm ad}(F)$ and $\tau(G_{\rm ad, \gamma}):=\text{vol}(G_{\rm ad, \gamma}(F)\backslash G_{\rm ad, \gamma}(\A))$. Similarly the \textit{adelic orbital integral} is defined to be
\begin{equation*}
O^{G_{\rm ad}(\A)}_\gamma(f)=\int_{G_{\rm ad, \gamma}(\A)\backslash G_{\rm ad}(\A)}f(g^{-1}\gamma g)\,dg
\end{equation*}

In the next section, we will perform the so-called \textit{stabilizations} of the trace formulas (\ref{5.2.1}) and (\ref{5.2.2}) in order to compare them effectively. 
    
However, before doing that, we will want to examine further the relation between various objects on the group and adjoint group in this setting.

\subsection{Passing Between the Adjoint Group and the Group}\label{sec 5.3}
The trace formulas are given in terms of (twisted) orbital integrals on the adjoint group $G_{\rm ad}$, however, since our functions are defined on $G$, later we would like to make statements about integrals on $G$. Therefore, it seems necessary to relate various integrals on the group and the adjoint group. In this subsection, we assume that $G$ is a connected reductive group defined over a local or global field $F$ with $G_{\text{der}}=G_\text{sc}$. We assume that $E/F$ is a cyclic extension of degree $r$, and let $\theta\in\Gal(E/F)$ be a generator. We use the same $\theta$ to denote the automorphism on the group. We assume that $Z(G)$ is an induced torus over $F$ (in particular, it is connected). We also assume the Haar measure $dg$ on $G(F)$ is normalized so that $\text{vol}(Z(E)/Z(F))=1$. Denote $G_{\rm ad}:=G/Z$, therefore $G_{\rm ad}(E)=G(E)/Z(E)$ and $G_{\rm ad}(F)=G(F)/Z(F)$.

Assume $\lambda:Z(F)\to\C^\times$ is a smooth character trivial on the maximal compact subgroup $Z(F)_1\subset Z(F)$. Define $\lambda N: Z(E)\to\C^\times$ as before, and we still denote it by $\lambda$ by abuse of notations. Let $\phi\in C^\infty_c(G_{\rm ad}(E),\lambda^{-1})$ and $f\in C^\infty_c(G_{\rm ad}(F),\lambda^{-1})$ with similar definitions as in 8.1. Let $\gamma\in G(F)$ be a semisimple element. We first prove a technical result:
\begin{lem}\label{lem. surjection of centralizers}
The maps $G_\gamma(F)\to G^\circ_{\mathrm{ad},{\bar{\gamma}}}(F)$ and $G_{\delta\theta}(F)\to G^\circ_{\mathrm{ad},\bar{\delta}\theta}(F)$ are surjective.   
\end{lem}
\begin{proof}
We first notice that it suffices to prove that $G_\gamma\to G^\circ_{\mathrm{ad},\bar{\gamma}}$ and   $G_{\delta\theta}\to G^\circ_{\mathrm{ad},\bar{\delta}\theta}$ are surjective, as the results over $F$-rational point will follow from the long exact sequences induced by
\begin{equation*}
    1\longrightarrow Z\longrightarrow G_\gamma\longrightarrow G^\circ_{\mathrm{ad},\bar{\gamma}}\longrightarrow 1
\end{equation*}
and 
\begin{equation*}
    1\longrightarrow Z\longrightarrow G_{\delta\theta}\longrightarrow G^\circ_{\mathrm{ad},\bar{\delta}\theta}\longrightarrow 1.
\end{equation*}
We remark that these groups are defined over $F$ in the second exact sequence. To show that $G_\gamma\to G^\circ_{\mathrm{ad},\bar{\gamma}}$ is surjective, we notice that $G_\gamma$ is generated by the maximal torus $T$ containing $\gamma$ and root groups $U_\alpha$ for roots $\alpha$ relative to $T$ such that $\alpha(\gamma)=1$. We see that $\alpha(\gamma)=\alpha(\bar{\gamma})$ since $\alpha(z)=1$ for any root $\alpha$ and $z\in Z$, and $T$ surjects onto the maximal torus $\bar{T}$ in $G_\mathrm{ad}$ containing $\gamma$, therefore $G_\gamma\to G^\circ_{\mathrm{ad},\bar{\gamma}}$ is surjective.

For the twisted centralizers, we have the following commutative diagram:
\begin{equation*}
\begin{tikzcd}
(G_{\delta\theta})_E \arrow[r] \arrow[d, "\simeq"]
& (G^\circ_{\mathrm{ad},\bar{\delta}\theta})_E\arrow[d,"\simeq"] \\
(G_{N\delta})_E \arrow[r, twoheadrightarrow]
&  (G^\circ_{\mathrm{ad},N\bar{\delta}})_E
\end{tikzcd}
\end{equation*}
therefore, $(G_{\delta\theta})_E\to(G^\circ_{\mathrm{ad},\bar{\delta}\theta})_E$ is a surjection. Since $(G^\circ_{\mathrm{ad},\bar{\delta}\theta})_E\to G^\circ_{\mathrm{ad},\bar{\delta}\theta}$ is also an $F$-surjection , we have $G_{\delta\theta}\to G^\circ_{\mathrm{ad},\bar{\delta}\theta}$ is surjective, as desired.
\end{proof}
\begin{rem}
In the proof above, it suffices to just assume that $Z\subset Z(G)^\circ$.    
\end{rem}

\begin{lem}\label{lem ser in the adjoint group}
\begin{enumerate}[(a)]
    \item $\gamma$ is regular in $G(F)$ if and only if $\bar{\gamma}$ is regular in $G_{\rm ad}(F)$;
    \item $\gamma$ is elliptic in $G(F)$ if and only if $\bar{\gamma}$ is elliptic in $G_{\rm ad}(F)$;
    \item If $\bar{\gamma}$ is strongly regular in $G_{\rm ad}(F)$, then so is $\gamma$ in $G(F)$.
\end{enumerate}
\end{lem}
\begin{proof}
From the isomorphisms $G_\gamma/Z\cong G^{\circ}_{\rm ad, \bar{\gamma}}$ and $G_\gamma(F)/Z(F)\cong G^{\circ}_{\rm ad, \bar{\gamma}}(F)$, we know that (a)(b) are immediate. (c) also follows since $G_\gamma=G^\circ_\gamma$ for every semisimple $\gamma\in G(F)$. However we don't even need to assume $G_{\text{der}}=G_\text{sc}$, since $G_\gamma/Z\cong G^{\circ}_{\rm ad, \bar{\gamma}}=G_{\rm ad, \bar{\gamma}}$, and since $Z$ is connected, $G_\gamma$ is also connected.
\end{proof}
As an easy corollary, let $\delta\in G(E)$ be a $\theta$-semisimple element. We have the identity $N\bar\delta=\overline{N\delta}$. Therefore, we get the same statements by replacing regular (resp., elliptic, strongly regular) with $\theta$-regular (resp., $\theta$-elliptic, $\theta$-strongly regular) and replace $\gamma$ (resp. $F$) with $\delta$ (resp. $E$).

\begin{rem}\label{rem counterexample of strongly regular in the adjoint}
We remark that it is not necessary for $\bar{\gamma}$ to be strongly regular even if $\gamma$ is strongly regular. Let $F$ be a field of characteristic not equal to $2$. Consider the matrix \[\gamma=\begin{pmatrix}
    1 & 0\\
    0 & -1
\end{pmatrix}\in \GL_2(F).
\] It is strongly regular semisimple in $\GL_2(F)$. However, one can calculate that  $(\mathrm{PGL}_2)_\gamma(F)=T\rtimes (\Z/2\Z)$, where $T$ is the one-dimensional maximal split torus consists of diagonal matrices in $\mathrm{PGL}_2$  and the nontrivial element of $\Z/2\Z$ is represented by the matrix 
\[
\begin{pmatrix}
    0 & -1\\
    1 &0
\end{pmatrix}.
\]
\end{rem}
\begin{lem}\label{lem density of ser in the ad}
Assume $F$ is a local field. The projection map $\pi: G(F)\to G_{\rm ad}(F)$ restricts to a surjection 
\begin{equation*}
    \pi: \{\text{regular elliptic semisimple elements in }G(F)\}\to\{\text{regular elliptic semisimple elements in }G_{\rm ad}(F)\}
\end{equation*}
Moreover, the inverse image of  
\begin{equation*}
\{\text{strongly regular elliptic semisimple elements in }G_{\rm ad}(F)\}
\end{equation*}
is dense in the left hand side (in the analytic topology\footnote{Density in Zariski topology can be similarly shown as in \S 2.5 in \cite{humphreys1995conjugacy}, however density in Zariski topology in general does not imply density in analytic topology.} of $G(F)$).
\end{lem}
\begin{proof}
The surjection is from Lemma \ref{lem ser in the adjoint group}(a)(b). Since we have the surjection statement. For the density statement, we fix a maximal $F$-torus $T$ and it is clear that we only to prove the density statement for $T$, which follows from the following lemma. 
\end{proof}

\begin{lem}
Let $T$ be a maximal $F$-torus splits over a finite extension. Then the inverse image of
\begin{equation*}
U=\{\text{strongly regular elliptic elements in }T_{\rm ad}(F)\}    
\end{equation*}
is dense in the set of regular elliptic elements in $T(F)$, in the analytic topology on $T(F)$. 
\end{lem}
\begin{proof}
Let $X=\{\text{regular elliptic elements in } T(F)\}$ and let $Y=\pi^{-1}(U)\subset X$. We claim that $X\backslash Y$ has no interior point in $X$, therefore the lemma is proved.

Let $W=N_G(T)/T$ absolute Weyl group and let $\Phi$ be the roots relative to $T$. Both of them are finite.
The condition of  $y\in Y$ can be described as: 
\begin{equation*}
\begin{split}
\alpha(y)&\neq 1 \text{ for all  } \alpha\in\Phi \\ 
w(y)&\neq y \text{ for all }1\neq w\in W\\
\alpha(w(y)y^{-1})&\neq 1\text{ for some }\alpha\in\Phi \text{ and for all }1\neq w\in W
\end{split}    
\end{equation*}
Therefore, we see that the vanishing locus $X\backslash Y$ is given by a finite number of locally analytic equations. Therefore, locally they are given in terms of a finite number of convergent power series over the non-archimedean field. Therefore they vanish on a set with no interior point, as desired.
\end{proof}

Now we can relate orbital integrals on both sides. Assume $\delta\in G(E)$ is such that $\bar\delta$ is $\theta$-strongly regular, $\theta$-elliptic and $\theta$-semisimple in $G_{\rm ad}(E)$, and $\gamma\in G(F)$ is such that $\bar\gamma$ is strongly regular semisimple in $G_{\rm ad}(F)$.
\begin{lem}\label{lem equality of integrals between groups}
We have 
\begin{equation*}\tag{5.3.1}\label{5.3.1}
    TO^{G(E)}_{\delta\theta}(\phi)=TO^{G_{\rm ad}(E)}_{\bar\delta\theta}(\phi)
\end{equation*}
and
\begin{equation*}\tag{5.3.2}\label{9.3.2}
    O^{G(F)}_\gamma(f)=O^{G_{\rm ad}(F)}_{\bar\gamma}(f)
\end{equation*}
\end{lem}
\begin{proof}
We will just show the identity (\ref{5.3.1}), as (\ref{9.3.2}) follows from the same calculations. Recall that $\phi\in C^\infty_c(G_{\rm ad}(E),\lambda^{-1})$, therefore, the calculation is already done in (\ref{4.3.1}) given our normalization that $\text{vol}(Z(E)/Z(F))=1$. 
\end{proof}

\begin{lem}\label{lem 9.3.4}
Similarly, we have   
\begin{equation*}\tag{5.3.3}\label{5.3.3}
    SO^{G(E)}_{\delta\theta}(\phi)=SO^{G_{\rm ad}(E)}_{\bar\delta\theta}(\phi)
\end{equation*}
and
\begin{equation*}\tag{5.3.4}\label{5.3.4}
    SO^{G(F)}_\gamma(f)=SO^{G_{\rm ad}(F)}_{\bar\gamma}(f)
\end{equation*}
\end{lem}
\begin{proof}
The arguments are exactly like those in the proof of Lemma \ref{lem 4.4} (iv), except it is easier in this case: since $\delta$ and $\bar{\delta}$ are $\theta$-strongly regular, we don't have the $a(\delta')$ terms in the stable orbital integrals. They reduce to 
\begin{equation*}
\begin{split}
SO^{G(E)}_{\delta\theta}(\phi)&=\sum_{\delta'}e(\delta')TO^{G(E)}_{\delta'\theta}(\phi)\\
SO^{G_{\rm ad}(E)}_{\bar\delta\theta}(\phi)&=\sum_{\bar\delta'}e(\bar\delta')TO^{G_{\rm ad}(E)}_{\bar\delta'\theta}(\phi)
\end{split}
\end{equation*}
respectly, where $\delta'$ (resp. $\bar\delta'$) ranges over $\theta$-conjugacy classes in $G(E)$ (resp. $G_{\rm ad}(E)$) which are stably $\theta$-conjugate to $\delta$ (resp. $\bar\delta$). The orbital integrals are equal by the previous lemma, and $e(\delta')=e(\bar\delta')$ again from \cite{kottwitz1983sign}. Finally, as in Lemma \ref{lem 4.4} (iv), the surjective map from the set of $\theta$-conjugacy classes in $G(E)$  which are stably $\theta$-conjugate to $\delta$\ to the set of $\theta$-conjugacy classes in $G_{\rm ad}(E)$  which are stably $\theta$-conjugate to $\bar\delta$ has single fiber 
\begin{equation*}
   \ker[H^1(F, G^{\circ}_{\rm ad, \bar\delta'\theta})\to H^1(F, G_{\rm ad, \bar\delta'\theta})] 
\end{equation*}
over $\bar\delta'$ since $G^{\circ}_{\rm ad, \bar\delta'\theta}=G_{\rm ad , \bar\delta'\theta}$ in this case, the lemma is proved.
\end{proof}

The upshot of these is that, after stabilization in the next section, we may extract conditional identities of (stable) orbital integrals on $G_{\rm ad}$, but with the relations we have proved in this subsection, we can regard them as identities on $G$, which are eventually what we desire.

\section{Stabilization of the Twisted Trace Formula}\label{sec 6}

We follow \cite{Clozel90},  \S6.2 closely with some changes along the way.

We fix the notations that will be used in this section. Let $E/F$ be a unramified cyclic extension of global fields of degree $r$, and let $\theta\in\Gal(E/F)$ be a generator. If $G$ is an $F$-group, then $\Tilde{G}:=\Res_{E/F}G_E$ is the $F$-group obtained by the restriction of scalars. We use the same notation $\theta$ to denote the $F$-linear automorphism of $\Tilde{G}$ over $F$.  We denote by $\A$ the adeles of $F$, and we denote by $\A^s$ the ring of adeles of $F^s$. 

We make the assumption that $G$ is a quasi-split, connected reductive group over $F$, such that $G_{\text{der}}$ is simply connected. By the reduction steps, we may also assume that $Z(G)$ is an induced torus over $F$ and let $H:=G/Z(G)$ to denote the adjoint group in this section, which we denote by $G_{\rm ad}$ before, to avoid the overflow of the subscripts. We will assume that $H$ splits over an unramified extension $K/F$ such that $E\subset K$. We may also assume that $\Tilde{H}$ satisfies the \textit{Hasse principle}, that is:
\begin{equation*}
\ker^1(F,\Tilde{H}):=\ker[H^1(F,\Tilde{H})\to\prod_vH^1(F_v,\Tilde{H})]=1.    
\end{equation*}

We denote by $\mathcal{N}$ the abstract norm mapping sending regular semisimple elements of $\Tilde{H}(F)$ to stable conjugacy classes of regular semisimple elements in $H(F)$, as well as its local versions.

\subsection{The Construction of the Obstruction}\label{sec 6.1}
Let $\gamma\in H(F)$ be a strongly regular semisimple element, then its centralizer $T$ (resp., $\Tilde{T}$) is a maximal torus of $H$ (resp., $\Tilde{H}$) over $F$. Assume there exists $\delta\in\Tilde{H}(\A)$ is such that $\mathcal{N}(\delta_v)$ is equal to the stable conjugacy class of $\gamma\in H(F_v)$ for every place $v$ of $F$.

Let $N$ denote the map $x\mapsto x\theta(x)\dots\theta^{r-1}(x)$ from $\Tilde{H}$ to itself. The map $N:\Tilde{T}(F^s)\to T(F^s)$ is surjective, so we can choose $t\in\Tilde{T}(F^s)$ such that $Nt=\gamma$. Since $\mathcal{N}(\delta_v)=\gamma$, for any place $v$ of $F$, there is an element $g_v\in \Tilde{H}(F_v^s)$ such that $g^{-1}_v\delta_v\theta(g_v)=t$. We may even assume that $g=(g_v)\in\Tilde{H}(\A^s)$ by the following lemma.
\begin{lem}\label{lem 6.1}
Let $G$ be a (possibly disconnected) reductive group over a global field $F$, and let $t\in G(F)$ be a semisimple element. Then outside of finitely many places $v$ of $F$, for every $\delta\in G(\mathcal{O}_v)$ such that $\delta$ is conjugate to $t$ under $G^\circ(F_v^s)$, there exists $y\in G^\circ(\mathcal{O}^{\text{un}}_v)$ such that $yt y^{-1}=\delta$.
\end{lem}
Here $\mathcal{O}_v$ is the valuation ring of the completion $F_v$ of $F$ at the place $v$, and $\mathcal{O}^{\text{un}}_v$ is the valuation ring of the maximal unramified extension $F^{\text{un}}_v$ of $F_v$.
\begin{proof}
We adapt the proof of Lemma 5 in \cite{KotRog00}. Let $X$ denote the conjugacy class of $t$ under $G^\circ$, and let $i:X\hookrightarrow G$ denote the inclusion, it is a closed immersion since $t$ is semisimple. Let $f$ denote the morphism $g\mapsto gt g^{-1}$ from $G^\circ$ to $X$. Then $f$ is smooth and surjective. There exists an ideal $I\subset\mathcal{O}_F$, which is the product of finitely many prime ideals, such that $G, G^\circ, X,t, i$ and $f$ come from objects over $\mathcal{O}_F[\frac{1}{I}]$. We may assume that $f$ is smooth and surjective and $i$ is a closed immersion over $\mathcal{O}_F[\frac{1}{I}]$ by replacing $I$ with a suitable multiple. 

Now we choose a place $v$ of $F$ outside of the finitely many places of $I$. By hypothesis $\delta\in G(\mathcal{O}_v)\cap X(F_v)$, which is equal to $X(\mathcal{O}_v)$, since $i$ is a closed immersion. The fiber $Y_\delta$ of $f:G^\circ\to X$ over $\delta\in X(\mathcal{O}_v)$ is a smooth scheme of finite type over $\mathcal{O}_v$. The structural morphism $Y_\delta\to\Spec(\mathcal{O}_{v})$ is surjective since it is the composite of the following surjective morphisms:
\begin{equation*}
Y_\delta\longrightarrow\{\text{conjugacy class of } \delta \text{ under } G^\circ\}\longrightarrow\Spec(\mathcal{O}_v)    
\end{equation*}
hence the special fiber $Y_\delta$ is non-empty, and it has a point in some finite extension of the residue field of $\mathcal{O}_v$. Therefore by the smoothness of $Y_\delta$, it has a point in the valuation ring of some finite unramified extension of $F_v$. Hence $Y_\delta(\mathcal{O}^{\text{un}}_v)$ is non-empty, which means that there exists $y\in G^\circ(\mathcal{O}^{\text{un}}_v)$ such that $y\gamma y^{-1}=\delta$.
\end{proof}

\begin{rem}\label{rem 6.2}
Given the lemma above, we set $G=\Tilde{H}\rtimes \langle\theta\rangle$, hence $G^\circ=\Tilde{H}\rtimes 1$. We can regard $t=(t,\theta)\in G(F)$ and similarly $\delta\in G^\circ(\A)$. Outside of finitely many places, we have $\delta_v\in G^\circ(\mathcal{O}_v)$. In this way, we can translate the fact that $\delta_v$ and $t$ are $\theta$-conjugate under $\Tilde{H}(F^s_v)$ to that they are conjugate under $G^\circ(F^s_v)$, therefore we can apply the lemma to get $g\in G^\circ(\A^s)=\Tilde{H}(\A^s)$ such that 
\begin{equation*}\tag{6.1.1}\label{6.1.1}
g^{-1}\delta\theta(g)=t.    
\end{equation*}
\end{rem}

Consider the map $\tau\mapsto t_\tau=g^{-1}\tau(g)\in\Tilde{H}(\A^s)$ for $\tau\in\Gal(F^s/F)$. We claim that $t_\tau$ defines a 1-cocycle of $\Gal(F^s/F)$ in $T(\A^s)\tilde{T}(F^s)$. Indeed, apply $\tau$ to (\ref{6.1.1}) we get
\begin{equation*}\tag{6.1.2}\label{6.1.2}
\tau(g^{-1})\delta\theta(\tau(g))=\tau(t).
\end{equation*}
from (\ref{6.1.1}) and (\ref{6.1.2}) we get 
\begin{equation*}\tag{6.1.3}\label{6.1.3}
t_\tau\tau(t)\theta(t_\tau)^{-1}=t   
\end{equation*}
Apply the map $N$ on both sides, we get
\begin{equation*}
t_\tau N(\tau(t))t_\tau^{-1}=Nt;
\end{equation*}
we have $Nt=\gamma$, and $\tau$ commutes with $N$ since $\tau$ commutes with $\theta$, hence we have $t_\tau\gamma t^{-1}_\tau=\gamma$. Therefore we have $t_\tau\in\Tilde{T}(\A^s)$. Then from (\ref{6.1.3}) we know that $t_\tau\theta(t_\tau)^{-1}=t\tau(t)^{-1}\in\Tilde{T}(F^s)$, and we know that $u=t\tau(t^{-1})$ satisfies $Nu=1$. Since $\theta$ acts on $\Tilde{T}(F^s)=T(F^s)\times\dots\times T(F^s)$ by cyclic permutation, this implies that $u=t_\tau\theta(t_\tau)^{-1}=t\tau(t)^{-1}=v\theta(v)^{-1}$ for some $v\in\Tilde{T}(F^s)$. Then $t_\tau v^{-1}$ is fixed by $\theta$, in other words, it belongs to $T(\A^s)$. This shows that $t_\tau\in T(\A^s)\Tilde{T}(F^s)$. It is a 1-cocycle since by definition it is a coboundary in $\Tilde{H}(\A^s
)$.
\begin{defn}\label{defn obs}
Let $t_\tau:=g^{-1}\tau(g)$ for $\tau\in\Gal(F^s/F)$. We take $x_\tau$ to be the image of $t_\tau$ in $T(\A^s)\Tilde{T}(F^s)/\Tilde{T}(F^s)=T(\A^s)/T(F^s)$, and define $\rm{obs}(\delta)$ to be the class of $(x_\tau)$ in $H^1(F,T(\A^s)/T(F^s))$. This definition does not depend on the choice of $g$ or $t$.
   
\end{defn}

We have the following important property of $\text{obs}(\delta)$:

\begin{lem}\label{lem vanishing of obs}
$\mathrm{obs}(\delta)\in H^1(F,T(\A^s)/T(F^s))$ is trivial if and only if $\delta$ is $\theta$-conjugate under $\Tilde{H}(\A)$ to an element of $\Tilde{H}(F)$.
\end{lem}
\begin{proof}
The proof is the same as that of Lemma 6.2 in \cite{Clozel90}.
\end{proof}

\subsection{Pre-Stabilization}\label{sec 6.2}
Now we assume that $T$ is an $F$-torus of $H$ of the form $T=H_\gamma$ for a strongly regular semisimple element $\gamma\in H(F)$ (recall that $\gamma$ is strongly regular means that $T=H_\gamma$ is a maximal torus, in particular, connected). We set $D=H/H_\text{der}$ and $\Tilde{D}=\Tilde{H}/\Tilde{H}_\text{der}$, We have the following commutative diagram
\begin{equation*}
\begin{tikzcd}
T \arrow[r, hookrightarrow] \arrow[d, twoheadrightarrow]
& \Tilde{T}\arrow[d,twoheadrightarrow] \\
D \arrow[r, hookrightarrow]
&  \Tilde{D}
\end{tikzcd}
\end{equation*}
Dually, it becomes
\begin{equation*}
\begin{tikzcd}
\hat{\Tilde{D}} \arrow[r, hookrightarrow] \arrow[d, twoheadrightarrow]
& \hat{\Tilde{T}}\arrow[d,twoheadrightarrow] \\
\Hat{D} \arrow[r, hookrightarrow]
&  \Hat{T}
\end{tikzcd}
\end{equation*}
Those maps are all $\Gamma$-equivariant. 

We define the finite abelian groups
\begin{align*}\tag{6.2.1}\label{6.2.1}
A(T/F)& = \pi_0(\hat{T}^{\Gamma})/\text{Im}\,\pi_0(\Hat{\Tilde{D}}^{\Gamma}),\\
A(T/F_{v_1})& = \pi_0(\hat{T}^{\Gamma_1})/\text{Im}\,\pi_0(\Hat{\Tilde{D}}^{\Gamma_1}).
\end{align*}
Where $\Gamma_1=\Gal(F^s_{v_1}/F_{v_1})$ is the decomposition group at the place $v_1$ of $F$.

The finite Abelian groups $\pi_0(\hat{T}^{\Gamma})$ and $H^1(F,T(\A^s)/T(F^s))$ are canonically dual by \cite{KotCusp}. Let $<,>$ be the pairing between them with value in $\mathbb{C}^{\times}$. We claim that $\text{obs}(\delta)$ has trivial image under the composition
\[
H^1(F,T(\A^s)/T(F^s))\xrightarrow{\sim}\pi_0(\hat{T}^{\Gamma})^*\to\pi_0(\Hat{\Tilde{D}}^{\Gamma})^*
\]
Indeed, since $\pi_0(\Hat{\Tilde{D}}^{\Gamma})^*$ is canonically dual to $H^1(F,\Tilde{D}(\A^s)/\Tilde{D}(F^s))$, so we only need to find the image of $\text{obs}(\delta)$ under the map 
\[
H^1(F,T(\A^s)/T(F^s))\to H^1(F,\Tilde{D}(\A^s)/\Tilde{D}(F^s)), 
\]
since $\text{obs}(\delta)$ was constructed to be the class of $t_\tau=g^{-1}\tau(g)$ for some $g\in\Tilde{H}(\A^s)$, it splits when we project $t_\tau$ to $\Tilde{D}(\A^s)$. Hence for $\kappa\in A(T/F)$, we can further define the pairing $\langle\text{obs}(\delta),\kappa\rangle\in\mathbb{C}$. We have 
\begin{equation*}
\text{obs}(\delta)=1\Leftrightarrow\langle\text{obs}(\delta),\kappa\rangle=1 \text{ for all } \kappa\in A(T/F).
\end{equation*}

We now begin the stabilization of the right hand side of the trace formula (\ref{5.2.1})
\begin{equation*}\tag{6.2.2}\label{6.2.2}
\sum_{\delta}\tau(H_{\delta\theta})TO^{H(\A_E)}_{\delta\theta}(\phi)
\end{equation*}
where we recall that the sum is over $\delta\in H(E)$ up to $\theta$-conjugacy under $H(E)$ with $N\delta$ regular elliptic. We also recall that $\phi\in C^\infty_c(H(\A_E), \lambda^{-1})$. 

The global twisted orbital integral $TO^{H(\A_E)}_{\delta\theta}$ of $\phi$ only depends on the $H(\A_E)$-$\theta$-conjugacy class of $\delta$. We have the following lemma which allows us to regroup the elements in the summation:
\begin{lem}\label{lem 6.5}
The number of the terms in the sum that are indexed by the $H(\A_E)$-$\theta$-conjugates of $\delta$ is given by
\[
|\ker(\ker^1(F,H_{\delta\theta})\to\ker^1(F,\Tilde{H}))|=|\ker^1(F,H_{\delta\theta})|
\]
where we have used the Hasse principle for $\Tilde{H}$.
\end{lem}
\begin{proof}
We just need to look at the following diagram:
\begin{equation*}
\begin{tikzcd}
H^1(F, H_{\delta\theta})\arrow[r, "f", rightarrow] \arrow[d, "g"]
& \prod_v H^1(F_v, H_{\delta\theta})\arrow[d, "g'"] \\
H^1(F,\Tilde{H}) \arrow[r, "f'", rightarrow] 
&  \prod_v H^1(F_v,\Tilde{H})
\end{tikzcd}
\end{equation*}
Assume $\delta'$ is $\theta$-conjugate to $\delta$ under $H(\A_E)$ but not under $H(E)$, then they are $\theta$-conjugate under $H(E^s)$ by Lemma \ref{lem 6.6}, therefore corresponds to a cocycle $\tau\in \ker(g)$, which image $f(\tau)$ in $\prod_v H^1(F_v, H_{\delta\theta})$ is trivial by assumption, that is to say, $\tau\in\ker(\ker^1(F, H_{\delta\theta})\to\ker^1(F,\Tilde{H}))$. Conversely, if $\tau\in\ker(H^1(F, H_{\delta\theta})\to H^1(F,\Tilde{H}))$, then we can associated a $\delta'\in H(E)$ that is stably $\theta$-conjugate to $\delta$. The condition $f(\tau)=0$ guarantees that $\delta'$ and $\delta$ are $\theta$-conjugate under $H(\A_E)$.
\end{proof}

\begin{lem}\label{lem 6.6}
Let $K$ be a global field and $H$ be a connected reductive group over $K$. Let $\theta$ be an $K$-linear automorphism of $H$. Let $a,b\in H(K)$ such that they are strongly $\theta$-regular semisimple and they are $\theta$-conjugate under $H(\A_K)$, where $\A_K$ is the ring of adeles of $K$, then they are $\theta$-conjugated under $H(K^s)$, where $K^s$ is a fixed separable closure of $K$.   
\end{lem}
\begin{proof}
Indeed, the conditions imply that the algebraic variety $X:=\{h\in H: h^{-1}a\theta(h)=b\}$ has a point in $K^s_v$, where $v$ is a place of $K$ and $K^s_v$ is a separable closure of the completion $K_v$. Since $X$ is defined by algebraic equations over $K$ as $a, b$ are rational over $K$, this implies that $X(\bar{K})\neq\emptyset$, where $\bar{K}$ is an algebraic closure containing $K^s$. Since $X$ is isomorphic to $H_{a\theta}$ as schemes and the latter is geometrically reduced, since $H_{a\theta}$ is a connected torus, we have $X(K^s)\neq\emptyset$, as desired.   
\end{proof}

Therefore we can write (\ref{6.2.2}) as 
\begin{equation*}\tag{6.2.3}\label{6.2.3}
\sum_{\delta}\tau(H_{\delta\theta})|\ker^1(F,H_{\delta\theta})|\,TO^{H(\A_E)}_{\delta\theta}(\phi)
\end{equation*}
where the sum is over $\delta\in H(E)$ up to $\theta$-conjugacy under $H(\A_E)$ with $N\delta$ strongly elliptic regular.

Now we consider a strongly regular elliptic element $\gamma\in H(F)$ up to stable conjugacy, in other words, conjugacy under $H(F^s)$ since $\gamma$ is strongly regular. Assume $\delta\in H(\A_E)$ is such that its norm $\mathcal{N}\delta$ is equal to the stable class of $\gamma$ at every place. The value of
\begin{equation*}
\frac{1}{|A(T/F)|}\sum_{\kappa\in A(T/F)} \langle\text{obs}(\delta),\kappa\rangle
\end{equation*}
is $1$ when $\delta$ is $\theta$-conjugate to an element of $H(E)$ under $H(\A_E)$, by Lemma \ref{lem vanishing of obs}, and $0$ otherwise since $A(T/F)$ is a finite abelian group. Use this, we can write (\ref{6.2.3}) as 
\begin{equation*}\tag{6.2.4}\label{6.2.4}\begin{split}
\sum_{\substack{\gamma\in H(F) \text{ strongly ell. reg.}\\\text{up to }H(F^s)-\text{conj}}} &\tau(T)|\ker^1(F,T)|\\ &\sum_{\substack{\delta\in H(\A_E)\\\mathcal{N}\delta=\gamma\\\text{up to }H(\A_E)-\theta-\text{conjugacy}}}[\frac{1}{|A(T/F)|}\sum_{\kappa\in A(T/F)} \langle\text{obs}(\delta),\kappa\rangle] TO^{H(\A_E)}_{\delta\theta}(\phi) 
\end{split}
\end{equation*}
The equation $\mathcal{N}\delta=\gamma$ holds for every place $v$ of $E$. And we write $T=H_\gamma$ for the centralizer for $\gamma$, which is isomorphic to $H_{\delta\theta}$ if $\delta$ is a global element of norm $\gamma$.

In order to swap the order of the outer summations, we claim that the triple sum has only finitely many nonzero terms on any compact subset.
\begin{lem}\label{lem 6.7}
Fix a compact set $C$ of $H(\A_E)$, then there are only finitely many $\theta$-conjugacy classes $\delta$ under $H(\A_E)$ that meet $C$ and such that $\mathcal{N}\delta$ is the class of a regular semisimple element of $H(F)$.
\end{lem}
\begin{proof}
 Following the same proof as in \cite{KotEll}, we can reduce to case that $N\delta$ is stably conjugate to a fixed regular semisimple element $\gamma$ of $H(F)$. We can find $\delta_0\in H(E)$ such that $N\delta_0=\gamma$. Hence $\delta_0$ and $\delta$ are stably $\theta$-conjugate. We assume that $C$ is contained in $H(\mathcal{O}_{\A^S_E})\times H(\prod_{v\in S}E_v)$, where $S$ is a  finite set of fixed places of $E$. We enlarge $S$ to apply Lemma \ref{lem 6.1}. Now for any $\delta\in C$ such that $\delta_{0,v}$ is stably $\theta$-conjugate to $\delta_v$ for any place $v\notin S$, we view $\delta$ and $\delta_0$ as $(\delta, \theta)$ and $(\delta_0, \theta)$ in the disconnected group $H\ltimes\langle\theta\rangle$, respectively. Hence $(\delta, \theta)$ is conjugate to $(\delta_0,\theta)$ under $H(E_v^s)=(H\ltimes\langle\theta\rangle)^\circ(E^s_v)$ in this sense. By Lemma \ref{lem 6.1}, there exists $g\in H(\mathcal{O}_v^{\text{un}})$ such that $\delta$ and $\delta_0$ are conjugate under $g$. We claim that we can actually assume $g\in H(O_v)$. 

Indeed, consider the cocycle $\tau_g:\theta\mapsto g^{-1}\theta(g)$, here we use $\theta$ to denote the topological generator of $\Gal(F^{\text{un}}/F)$. It is clear that $\tau_g$ is a 1-cocycle in $\mathcal{T}(\mathcal{O}_v^{\text{un}}):=H_\gamma(\mathcal{O}_v^{\text{un}})$, by applying $\theta$ on the equation $g^{-1}\delta_vg=\delta_{0,v}$. Here we use the existence of lft Neron model $\mathcal{T}$ associated to $T=H_\gamma$ over $\mathcal{O}_v^{\text{un}}$. Hence $\tau_g\in H^1(\langle\theta\rangle,\mathcal{T}(\mathcal{O}_v^{\text{un}}))$. It is trivial by (7.6.1) of \cite{KOTTWITZ_1997}, hence there exists $t\in\mathcal{T}(\mathcal{O}_v^{\text{un}})$ such that $g^{-1}\theta(g)=t^{-1}\theta(t)$, in other words, $gt^{-1}\in H(\mathcal{O}^\text{un}_v)\cap H(F_v)=H(O_v)$. Replace $g$ by $gt^{-1}$, the claim is proved.

Now for every $\theta$-conjugacy class $\delta$ that meet $C$ and $\mathcal{N}\delta=\gamma$, we know for $v\notin S$, the $\theta$-conjugacy class of $\delta_v$ contains $\delta_{0,v}$, and for every place $w\in S$, there are only finitely many $\theta$-conjugacy classes inside a fixed stable $\theta$-conjugacy class. Therefore, there are only finitely many such $\theta$-conjugacy classes $\delta$ satisfying those conditions when $C\subset H(\A_E)$.
\end{proof}

Therefore, we can reorder the sum in (\ref{6.2.4}) as
\begin{equation*}\tag{6.2.5}\label{6.2.5}
\sum_{\gamma}\tau(T)\frac{|\ker^1(F,T)|}{|A(T,F)|}\sum_{\kappa}\sum_{\delta} \langle\text{obs}(\delta),\kappa\rangle TO^{H(\A_E)}_{\delta\theta}(\phi).   
\end{equation*}
Where each index is summing over the same set as in (\ref{6.2.4}).
\subsection{Vanishing of Kappa Orbital Integrals}\label{sec 6.3}
Consider the inner sum in (\ref{6.2.5}) 
\begin{equation*}\tag{6.3.1}\label{6.3.1}
\sum_{\delta} \langle\text{obs}(\delta),\kappa\rangle TO^{H(\A_E)}_{\delta\theta}(\phi)    
\end{equation*}
where the summation is over $H(\A_E)$-$\theta$-conjugacy classes of $\delta\in H(\A_E)$ such that $\mathcal{N}\delta=\gamma$ locally everythere, for an elliptic strongly regular $\gamma\in H(F)$ up to conjugacy by $H(F^s)$. 

We assume that there is a $\theta$-elliptic, $\theta$-strongly regular $\delta^0\in H(\A_E)$ such that $\mathcal{N}\delta^0=\gamma$ holds. Otherwise, the sum in (\ref{6.3.1}) would be empty. Therefore, the summation in (\ref{6.3.1}) will be over the global $\theta$-conjugacy classes within the global stable $\theta$-conjugacy class of $\delta^0$, by Proposition \ref{prop of abstract norm map}.

We know that the $\theta$-conjugacy classes within the stable $\theta$-conjugacy class of $\delta^0_v$, the local component of $\delta^0$ at $v$, are in bijection with 
\begin{equation*}
\mathcal{D}_\theta(I_v/F_v):=\ker(H^1(F_v, I_v)\to H^1(F_v,\Tilde{H})),    
\end{equation*}
where $I_v$ is the $\theta$-centralizer of $\delta^0_v$ in $\Tilde{H}_{F_v}$. We notice that $I_v=I^\circ_v$ by assumptions on $\delta^0$.

The next goal is to write (\ref{6.3.1}) as a product of local sums. If $\delta=(\delta_v)$ is in the (global) stable $\theta$-conjugacy class of $\delta_0$, by above we can write $(\delta_v)=(x_v\cdot\delta^0_v)$ where $x=(x_v)\in\bigoplus_v\mathcal{D}_\theta(I_v/F_v)$. Let $\text{inv}(\delta,\delta^0)\in H^1(F, T(\A^s)/T(F^s))$ be the image of $(x_v)$ under the maps
\begin{equation*}
\bigoplus_v\mathcal{D}_\theta(I_v/F_v)\to\bigoplus_v H^1(F_v, T(F_v))\to H^1(F, T(\A^s)/T(F^s)),  
\end{equation*}
where the first map is given by the canonical isomorphisms between $I_v$ and $T(F_v)$, where $T=H_\gamma$. From the definitions of $\text{obs}(\delta)$ and $\text{inv}(\delta,\delta^0)$, we see that 
\begin{equation*}
\text{obs}(\delta)=\text{inv}(\delta,\delta^0)\text{obs}(\delta^0).   
\end{equation*}

Therefore, assuming $\phi=\bigotimes_v\phi_v$ (over places of $F$) is a pure tensor product element in $C^\infty_c(H(\A_E))\cong C^\infty_c(\Tilde{H}(\A_F))$, 
therefore we can write (\ref{6.3.1}), when the sum is nonempty, as 
\begin{equation*}\tag{6.3.2}\label{6.3.2}
\langle\text{obs}(\delta^0),\kappa\rangle\,\prod_v(\sum_{x_v\in\mathcal{D}_{\theta}(I_v/F_v)}\langle x_v,\kappa_v\rangle TO^{H(E_v)}_{x_v\cdot\delta^0_v,\theta}(\phi_v))    
\end{equation*}
which we write as 
\begin{equation*}\tag{6.3.3}\label{6.3.3}
\langle\text{obs}(\delta^0),\kappa\rangle\,\prod_vO^{\kappa_v}_{\delta^0_v,\theta}(\phi_v) 
\end{equation*}
where we denote the \textit{$\kappa_v$-twisted orbital integral}, 
\[
O^{\kappa_v}_{\delta^0_v,\theta}(\phi_v):=\sum_{x_v\in\mathcal{D}_{\theta}(I_v/F_v)}\langle x_v,\kappa\rangle TO^{H(E_v)}_{x_v\cdot\delta^0_v,\theta}(\phi_v),
\]
and $\kappa_v\in\mathcal{D}_\theta(I_v/F_v)^*$ is the local image of $\kappa$.

We have arrived at the vanishing results of this section:
\begin{lem}\label{lem Kottwitz function}
Assume $\phi=\bigotimes_v\phi_v$ is a pure tensor product element in $\phi\in C^\infty_c(H(\A_E), \lambda^{-1})$ over places $v$ of $F$, and at some place $v_1$ of $F$, the following two conditions hold:
\begin{enumerate}[(a)]
    \item $TO^{H(E_v)}_{\delta_{v_1},\theta}(\phi_{v_1})=0$ for all elements $\delta_{v_1}\in\Tilde{H}(F_{v_1})=H(E\,\otimes F_{v_1})$ that are not $\theta$-elliptic, and for $\theta$-elliptic $\theta$-regular elements $\delta_{v_1}$, the twisted orbital integral is constant on stable $\theta$-conjugacy classes of $\delta_{v_1}$. 
    \item There exists a finite Galois extension $L$ of $F$ such that $L_{v_1}$ is a field and $L$ splits $\Tilde{D}$.
\end{enumerate}
Then 
\begin{equation*}\tag{6.3.4}\label{6.3.4}
\sum_{\delta} \langle\mathrm{obs}(\delta),\kappa\rangle TO^{H(\A_E)}_{\delta\theta}(\phi)=0    
\end{equation*}
if $\kappa$ is nontrivial in $A(T/F)$.
\end{lem}
\begin{rem}\label{rem 10.3.1}
If $\phi$ and $H$ indeed satisfy the conditions in the lemma, the sum (\ref{6.2.5}) would reduce to its stable part:  
\begin{equation*} \tag{6.3.5}\label{6.3.5}
\sum_\gamma\tau(T)\frac{|\ker^1(F,T)|}{|A(T,F)|}\sum_\delta TO_{\delta\theta}(\phi),    
\end{equation*}
which means that the summation over $\kappa$ only has $\kappa=1$ left.
\end{rem}
\begin{proof}
As in \cite{Clozel90}, Lemma 6.5, it boils down to the following lemma.
\end{proof}

\begin{lem}\label{lem 6.10}
Suppose there exists a finite Galois extension $L$ of $F$ such that $L_{v_1}$ is a field and $L$ splits $\Tilde{D}.$ Suppose $T$ is a maximal $F$-torus in $H$ which is elliptic at $v_1$, then the canonical map
\begin{equation*}
 A(T/F)\to A(T/F_{v_1})   
\end{equation*}
is injective.
\end{lem}
\begin{proof}
As in the proof of Lemma 8.4.1 of \cite{haines2009base} and Lemma 6.5 in \cite{Clozel90}, we simply notice that $\Hat{\tilde{D}}^\Gamma=\Hat{\tilde{D}}^{\Gamma_1}=\Hat{\tilde{D}}^{\Gamma'}$, where $\Gamma'=\Gal(L_{v_1}/F_{v_1})$.    
\end{proof}

Functions with properties in Lemma \ref{lem Kottwitz function}(a) are called \textit{generalized Kottwitz functions} (see \cite{LBC1999} 3.9 and \cite{haines2009base} 8.2). We will discuss them further when we move to the global setup in 11.2.

\subsection{End of Stabilization}\label{sec 6.4}

From now on, we assume that $\phi$ satisfies the conditions of Lemma \ref{lem Kottwitz function}. Following \cite{Clozel90}, we see that the elliptic regular term of the twisted trace formula (\ref{6.3.5}) is equal to
\begin{equation*}\tag{6.4.1}\label{6.4.1}
c(H)\sum_{\substack{\gamma\in H(F) \text{ strongly ell. reg.}\\\text{up to }H(F^s)-\text{conj}}} \tau(T)|\ker^1(F,T)|\frac{|\pi_0(\hat{D}^{\Gamma})|}{|\pi_0(\hat{T}^{\Gamma})|}\sum_{\substack{\delta\in H(\A_E)\\\mathcal{N}\delta=\gamma\\\text{up to }H(\A_E)-\theta-\text{conjugacy}}} TO_{\delta\theta}(\phi),   
\end{equation*}
with 
\begin{equation*}\tag{6.4.2}\label{6.4.2}
c(H)=\frac{|\text{Im}\,\pi_0(\Hat{\Tilde{D}}^\Gamma)|}   
{|\pi_0(\hat{D}^{\Gamma})|}.  
\end{equation*}

We assume $H$ satisfies the Hasse principle for $H^1$, and that $f=\bigotimes_vf_v\in C^\infty_c(H(\A), \lambda^{-1})$ satisfies the analogs of conditions (a) and (b) in Lemma \ref{lem Kottwitz function}. In other words, we regard $E=F$ and $\theta=\text{id}$. Then the elliptic strongly regular part of the trace formula for $f$, by the same stabilization process, reduces to 
\begin{equation*}\tag{6.4.3}\label{6.4.3}
\sum_{\substack{\gamma\in H(F) \text{ strongly ell. reg.}\\\text{up to }H(F^s)-\text{conj}}} \tau(T)|\ker^1(F,T)|\frac{|\pi_0(\hat{D}^{\Gamma})|}{|\pi_0(\hat{T}^{\Gamma})|}\sum_{\substack{\gamma'\in H(\A)\\\gamma'\text{ stably conjugate to }\gamma\\\text{up to }H(\A)-\text{conjugacy}}} O_{\gamma'}(f). \end{equation*}
Since the last sums in (\ref{6.4.1}) and (\ref{6.4.3}) can be expressed as products of local orbital integrals over places $v$ of $F$, as we discussed before, we see that if $\phi$ and $f$ have matching stable orbital integrals at every place of $F$, the two expressions (\ref{6.4.1}) and (\ref{6.4.3}) agree up to the factor $c(H)$, that is:
\begin{prop}\label{prop trace formula and stable orbital integrals}
Let $\phi=\bigotimes_v\phi_v\in C^\infty_c(H(\A_E), \lambda^{-1})$ and
$f=\bigotimes_vf_v\in C^\infty_c(H(\A), \lambda^{-1})$ be functions satisfying the conditions in \ref{sec 5.2} (1)(2)(3) and Lemma \ref{lem Kottwitz function} (a). Assume $\gamma\in G(F)$ and the image $\bar\gamma\in H(F)=G_{\rm ad}(F)$ are strongly elliptic regular in the respective groups, $\delta\in G(\A_E)$ such that $\mathcal{N}\delta=\gamma$ locally everywhere (therefore we also have $\mathcal N\bar\delta=\bar\gamma$, and $\delta_v$ (resp. $\bar\delta_v$) is $\theta$-strongly regular $\theta$-elliptic in $G(E_v)$ (resp. $H(E_v)$)).  
If we have for every place $v$ of $F$
\begin{equation*}\tag{6.4.4}\label{6.4.4}
SO^{G(E_v)}_{\delta_v\theta}=SO^{G(E_v)}_{\gamma_v},
\end{equation*}
for every such $\delta\in G(\A_E)$ and $\gamma\in G(F)$
then there exists a constant $c>0$ such that
\begin{equation*}\tag{6.4.5}\label{6.4.5}
\mathrm{tr}(R_\mathrm{cusp}(\phi)\circ I_\theta)=c \,\mathrm tr(r_{\mathrm{cusp}}(f)).   
\end{equation*}
\end{prop}
\begin{proof}
By the stabilization process in this section, we know that 
\begin{equation*}
\begin{split}
\tr(R_\text{cusp}(\phi)\circ I_\theta)=(\ref{6.4.1})\\
\tr(r_{\text{cusp}}(f))=(\ref{6.4.3}).
\end{split}
\end{equation*}
And we know that $(\ref{6.4.1})=c(H)\cdot (\ref{6.4.3})$ provided 
\begin{equation*}
SO^{H(E_v)}_{\bar\delta_v\theta}=SO^{H(E_v)}_{\bar\gamma_v}\end{equation*}
for every place $v$ of $F$. By Lemma \ref{lem 9.3.4}, this is equivalent to (\ref{6.4.4}), as desired.
\end{proof}

\section{Proof In the Strongly Regular Elliptic Case}\label{sec 7}

\subsection{Local Data}\label{sec 7.1}
We first give the definition of the local data, which is the necessary bridge between spectral side geometric side of the theory. We assume $G$ is any unramified reductive group defined over local field $F$. Let $E/F$ be an unramified extension of degree $r$, denote $\theta$ to be a generator of $\Gal(E/F)$ as usual. Let $G_r:=G(E)$.

Let $\text{Irr}_{\chi,\lambda}(G)$ (resp. $\text{Irr}^\theta_{\chi_r, \lambda_r}(G_r)$) denote the set of irreducible (resp. irreducible $\theta$-stable) admissible representations in $\mathfrak{R}_\chi(G)$ (resp. $\mathfrak{R}_{\chi_r}(G_r)$) that transformed by $\lambda$ (resp. $\lambda_r$) under the center action. We define $\mathcal{H}(G_r,\rho_r, \lambda^{-1}_r)$ to the algebra of compactly supported functions on $G_r/Z(G_r)$ transformed by $\rho_r$ (resp. $\lambda^{-1}$) under the subgroups $I_r$ (resp. $Z(G_r)$), and similarly for $\mathcal{H}(G,\rho,\lambda^{-1})$. 

We define the \textit{local data adapted to} $\text{Irr}_{\chi,\lambda}(G)$ to consist of data (a), (b) and (c), subject to conditions (1) and (2) below:
\begin{enumerate}[(a)]
    \item An index set $\mathcal{I}_{\lambda}$, possibly infinite;
    \item A collection of complex numbers $a_i(\pi)$ for $i\in\mathcal{I}_{\lambda}$ and $\pi\in\text{Irr}_{\chi,\lambda}(G)$;
    \item A collection of complex numbers $b_i(\Pi)$ for $i\in\mathcal{I}_{\lambda}$ and $\Pi\in\text{Irr}^\theta_{\chi_r,\lambda_r}(G_r)$.
\end{enumerate}
\begin{enumerate}[(1)]

    \item For an fixed $i$, the constants $a_i(\pi)$ and $b_i(\Pi)$ are zero for all but finitely many $\pi$ and $\Pi$.
    \item For $\phi\in\mathcal{H}(G_r,\rho_r, \lambda^{-1}_r)$ and $f\in\mathcal{H}(G,\rho,\lambda^{-1})$, the following statements are equivalent:
    \begin{enumerate}[(A)]
        \item For all $i\in\mathcal{I}_{\lambda}$, we have $\sum_\pi a_i(\pi)\langle\text{tr}\,\pi,f\rangle=\sum_\Pi b_i(\Pi)\langle\text{tr}\,\Pi I_\theta,\phi\rangle$;
        \item For all strongly regular elliptic semisimple norms $\gamma=\mathcal{N}(\delta)$, where $\gamma\in G$ and $\delta\in G_r$, we have
        \begin{equation*}
         SO_\gamma(f)=SO_{\delta\theta}(\phi).   
        \end{equation*}
    \end{enumerate}
\end{enumerate}
\begin{rem}\label{rem trace should be mod center}
Since $G$ might not have a compact center and the functions we consider here are only compactly supported mod centers, the trace $\langle\mathrm{tr}\,\pi,f\rangle$ (resp. $\langle\mathrm{tr}\,\Pi I_\theta,\phi\rangle$) should be understood to be over $G/Z$ (resp. $G_r/Z_r$) instead of over $G$ (resp. $G_r$), as the action (\ref{5.1.2}) is over $G/Z$ (resp. $G_r/Z_r$).
\end{rem}

\subsection{Global Setup}\label{sec 7.2}
From now on, we embed the local situation into a suitable global setup, in order to apply the stabilizations of (twisted) simple trace formula to prove the existence of the local data.

We assume $G_\text{der}=G_\text{sc}$. We may assume that $G$ is split over an unramified extension $K/F$ such that $E\subset K$. We choose a degree $[K:F]$ cyclic extension of global function fields $\underline{K}/\underline{F}$ and a finite place $v_0$ of $\underline{F}$ such that $\underline{K}_{v_0}$ is a field and $\underline{K}_{v_0}/\underline{F}_{v_0}\cong K/F$. Then there is a degree $r=[E:F]$ cyclic extension $\underline{E}/\underline{F}$ with $\underline{E}\subset\underline{K}$ and $\underline{E}_{v_0}/\underline{F}_{v_0}\cong E/F$.

The Tchebotarev density theorem is still valid in positive characteristics (for example, see \cite{jarden1982cebotarev}), therefore we can find an inert place $v_1$ of $\underline{F}$, $v_1\neq v_0$ with $\Gal(\underline{K}_{v_1}/\underline{F}_{v_1})=\Gal(K/F)$. In addition, we fix two more auxiliary finite places $v_2$ and $v_3$ of $\underline{F}$ where $\underline{E}/\underline{F}$ splits completely at these places. We can assume that $v_0\notin\{v_1,v_2,v_3\}$.

There is a quasi-split group $\underline{G}$ over $\underline{F}$ with the property that $\underline{G}\times_{\underline{F}}\underline{F}_{v_0}\cong G$. We set $\Tilde{\underline{G}}=\Res_{\underline{E}/\underline{F}}\underline{G}_{\underline{E}}$. By the reduction steps, we can assume that $Z(\underline{G})$ is an induced torus over $F$, and denote $\underline{H}=\underline{G}/Z(\underline{G})$.
Let $\theta$ denote the $\underline{F}$-linear automorphism of $\Tilde{\underline{G}}$ from $\theta=\Gal(E/F)\cong \Gal(\underline{E}/\underline{F})$.

We may assume that the groups $\underline{G}$ and $\Tilde{\underline{G}}$ have simply connected derived subgroups and split over $\underline{K}$. In order to apply the stabilization results in \S \ref{sec 6}, they have to satisfy the Hasse principle for $H^1$ on $\underline{H}$ and $\Tilde{\underline{H}}$, namely $\ker^1(\underline{F},\underline{H})=\ker^1(\underline{F},\Tilde{\underline{H}})=1$. Since $v_i$ splits completely for $i=2,3$, we have identifications
\begin{equation*}
\underline{G}(\underline{E}_{v_i})=  \underline{G}(\underline{F}_{v_i})\times\dots\times  \underline{G}(\underline{F}_{v_i})
\end{equation*}
with $r$ factors, and $\Gal(\underline{E}_{v_i}/\underline{F}_{v_i})$ acts by cyclic permutations. Same decompositions and permutations hold for $\underline{H}(\underline{E}_{v_i})$. 

We write $\underline{\phi}=\phi^{v_0}\otimes\phi_{v_0}=\otimes_v\phi_v$ for a pure tensor element of $C^\infty_c(\underline{H}(\A_{\underline{E}}), \lambda^{-1})$ for a global unitary character $\lambda:\underline{Z}(\A_{\underline{E}})/\underline{Z}(\underline{E})\to\C^\times$ over places of $\underline{F}$ and similarly $\underline{f}=f^{v_0}\otimes f_{v_0}=\otimes_vf_v$. We write $\phi$ for $\phi_{v_0}$ and $f$ for $f_{v_0}$. 

We will always use the symbol $S_3$ to denote a finite set of places of $\underline{F}$ such that $v_3\in S_3$ and $v_1, v_2\notin S_3$. Finally, we consider triples $(\phi^{v_0}, f^{v_0}, S_3)$ satisfying the following conditions.
\begin{enumerate}[(a)]
    \item At any place $v\notin S_3\cup\{v_1,v_2\}$, the group $\underline{G}\times_{\underline{F}}\underline{F}_v$ and the extension $\underline{E}_v/\underline{F}_v$ are unramified.
    \item The function $f_{v_1}$ (resp. $\phi_{v_1}$) is (up to a scalar) a pullback of \textit{generalized Kottwitz function} on $\underline{H}(\underline{F}_{v_1})$ (resp., $\Tilde{\underline{H}}(\underline{F}_{v_1})$) satisfying the conditions in Lemma \ref{lem Kottwitz function} (a). In particular, $\lambda_{v_1}=\text{triv}$. Moreover, we may assume they are associated, as in \S 8.2 of \cite{haines2009base}. We notice that this construction requires that $E/F$ is inert at $v_1$.
    \item The function $f'_{v_2}$ is a coefficient of a supercuspidal representation, $\phi_{v_2}=f'_{v_2}\otimes\dots\otimes f'_{v_2}$ and $f_{v_2}=f'_{v_2}\ast\dots\ast f'_{v_2}$. Thus $(\phi_{v_2},f_{v_2})$ are associated and have nonvanishing (twisted) orbital integrals at $(\theta)$-elliptic strongly $(\theta)$-regular elements which are close to the identity. Notice that these functions are only compactly supported modulo centers.
    \item For any $v\in S_3$, $f_v$ (resp., $\phi_v$) is supported on the set of strongly regular elements (resp., elements with strongly regular norms), and $(\phi_v, f_v)$ are associated. Moreover, the function $f_{v_3}$ (resp., $\phi_{v_3}$) is supported on the set of elliptic elements with strongly regular elliptic images in $G^*(\underline{F}_{v_3})$ (resp., elements with elliptic norms).
    \item For any other place $v\notin S_3\cup\{v_0, v_1,v_2\}$, the function $f'_v$ (resp., $\phi'_v$) is the unit element of the spherical Hecke algebra of $\underline{H}(\underline{F}_v)$ (resp., $\underline{H}(\underline{E}_v)$). Then $(\phi'_v, f'_v)$ are associated by \cite{kottwitz1986base}. We then form the function $\phi_v$ (resp., $f_v$) to be the pullback of $\phi'_v$ (resp., $f'_v$) with the trivial central action. Therefore $(\phi_v, f_v)$ are associated by Lemma \ref{lem 9.3.4}.
    
    \item At every place $v$ of $\underline{F}$ where $\underline{E}/\underline{F}$ splits, we have $\phi_v=f'_v\otimes\dots\otimes f'_v$ and $f_v=f'_v\ast\dots\ast f'_v$ for an appropriate function $f'_v$ so that $(\phi_v, f_v)$ are associated. 
\end{enumerate}
We note that by the above conditions, the functions $f_v$ and $\phi_v$ are assumed to be associated at every place $v\neq v_0$.

\subsection{Existence of the Local Data}\label{sec 7.3}
We prove that local data adapted to $\text{Irr}_{\chi,\lambda}(G)$ in this subsection.

By abuse of notations, we use $R(\underline{\phi})$ (resp. $R(\underline{f})$) to denote the action of $\underline{\phi}$ (resp., $\underline{f}$) on the cuspidal spectrum $L^2_{\text{cusp}}(\underline{H}(\underline{E})\backslash\underline{H}(\A_{\underline{E}}), \lambda)$ (resp., $L^2_{\text{cusp}}(\underline{H}(\underline{F})\backslash\underline{H}(\A_{\underline{F}}),\lambda)$). Let $I_\theta$ denote the intertwining operator on the same space given by $I_\theta(\psi)(x):=\psi(\theta^{-1}(x))$ for $\psi\in L^2_{\text{cusp}}(\underline{H}(\underline{E})\backslash\underline{H}(\A_{\underline{E}}),\lambda)$ and $x\in\underline{H}(\A_{\underline{E}})$.
The existence of local data comes from the following proposition. We say that $f_{v_0}$ and $\phi_{v_0}$ are associated at strongly regular elliptic norms if they are associated for every strongly regular elliptic semisimple norm in the adjoint group $\gamma=\mathcal{N}\delta$ in $\underline{G}(\underline{F}_{v_0})$ whose image $\bar{\gamma}$ is also strongly regular elliptic semisimple in $\underline{H}(\underline{F}_{v_0})$.
\begin{prop}\label{prop trace=association}
There is a constant $c\neq0$, depending only on $(\underline{G},\underline{E}/\underline{F})$, with the following property: the functions $f_{v_0}$ and $\phi_{v_0}$ are associated at strongly regular norms if and only we have the equality of traces
\begin{equation*}\tag{7.3.1}\label{7.3.1}
\mathrm{tr}(R(\phi^{v_0}\otimes\phi_{v_0})I_\theta)=c\,\mathrm{tr} \,R(f^{v_0}\otimes f_{v_0}),
\end{equation*}
for every triple $(\phi^{v_0}, f^{v_0}, S_3)$ satisfying conditions (a)-(f) in \S \ref{sec 7.2}.
\end{prop}
\begin{proof}
For the "if" part, at the place $v_0$, if we are given $\gamma_0=\mathcal{N}\delta_0$ for (resp. $\theta$-) strongly regular elliptic semisimple element $\gamma_0\in\underline{G}(\underline{F}_{v_0}) $ (resp. $\delta_0\in\underline{G}(\underline{E}_{v_0}) $ ), we choose an appropriate global element $\delta\in\underline{\tilde{G}}(\underline{F})$ such that:
\begin{enumerate}[(1)]
    \item $\delta_{v_0}$ is close to $\delta_0$ at $v_0$ and close to $1$ at $v_2$;
    \item $\delta_{v_i}$ is $\theta$-elliptic and strongly $\theta$-regular at $i=1,2,3$.
\end{enumerate}
Thus $\delta$ itself is $\theta$-elliptic and strongly $\theta$-regular. Then we can choose the set $S_3$ and the associated functions $(f^{v_0}, \phi^{v_0})$ such that, the geometric sides of (\ref{7.3.1}), according to the stabilization process, take the stabilized form of equations (\ref{6.4.1}) and (\ref{6.4.3}), which involve only the term indexed by $\gamma:=\mathcal{N}\delta$. The second sums in those equations, which are the sum of adelic (twisted) orbital integrals can be written as a product over all places of local stable (twisted) orbital integrals, and at all places except $v_0$, these are nonzero (by choice) and matching. Then the identity (\ref{7.3.1}) will force the matching at $v_0$, namely, $SO^{\underline{H}(\underline{E}_{v_0})}_{\delta_{v_0}\theta}(\phi_{v_0})=SO^{\underline{H}(\underline{F}_{v_0})}_{\gamma_{v_0}}(f_{v_0})$. By Lemma \ref{lem 9.3.4}, we then have $SO^{\underline{G}(\underline{E}_{v_0})}_{\delta_{v_0}\theta}(\phi_{v_0})=SO^{\underline{G}(\underline{F}_{v_0})}_{\gamma_{v_0}}(f_{v_0})$. A continuity argument then forces the desired identity $SO^{\underline{G}(\underline{E}_{v_0})}_{\delta_0\theta}(\phi_{v_0})=SO^{\underline{G}(\underline{F}_{v_0})}_{\gamma_0}(f_{v_0})$.

For the "only if" part, we just need to use Proposition \ref{prop trace formula and stable orbital integrals}, since for (\ref{7.3.1}) to be true, by assumption, the only ingredient we need is simply $SO^{\underline{G}(\underline{E}_{v_0})}_{\delta_{v_0}\theta}(\phi_{v_0})=SO^{\underline{G}(\underline{F}_{v_0})}_{\gamma_{v_0}}(f_{v_0})$ for every $\gamma=\mathcal{N}\delta$, which holds by assumption.
\end{proof}

\begin{prop}\label{prop existence of local data}
Local data adapted to $\mathrm{Irr}_{\chi,\lambda}(G)$ in the sense of \S \ref{sec 7.1} exists.
\end{prop}
\begin{proof}
On the spectral side, we can rewrite equation (\ref{7.3.1}) as
\begin{equation*}
\sum_{\Pi}\langle\text{tr}\,\Pi I_\theta,\phi\rangle=c\sum_{\pi}\langle\text{tr}\,\pi ,f\rangle,    
\end{equation*}
where the sum is over automorphic cuspidal representations $\Pi$ (resp., $\pi$) in the decomposition of the space
$L^2_{\text{cusp}}(\underline{H}(\underline{E})\backslash\underline{H}(\A_{\underline{E}}),\lambda)$ (resp., $L^2_{\text{cusp}}(\underline{H}(\underline{F})\backslash\underline{H}(\A_{\underline{F}}),\lambda)$).
Using the product decomposition of $\phi$ and $f$, this is
\begin{equation*}\tag{7.3.2}\label{7.3.2}
\sum_{\Pi}\prod_v \langle\text{tr}\,\Pi_v I_\theta,\phi_v\rangle=c\sum_{\pi}\prod_v\langle\text{tr}\,\pi_v ,f_v\rangle    
\end{equation*}
Therefore the identity (\ref{7.3.1}) is equivalent to the identity (\ref{7.3.2}). Fixing the functions $\phi_v, f_v$ at the places other than $v_0$, for each we get an identity
\begin{equation*}
\sum_{\Pi}\langle\text{tr}\,\Pi^{v_0} I_\theta,\phi^{v_0}\rangle \langle\text{tr}\,\Pi_{v_0} I_\theta,\phi_{v_0}\rangle=\sum_{\pi}\langle\text{tr}\,\pi^{v_0} ,f^{v_0}\rangle\langle\text{tr}\,\pi_{v_0} ,f_{v_0}\rangle.    
\end{equation*}
Assuming $f_{v_0}$ and $\phi_{v_0}$ range only over functions bi-$\rho$-invariant under a fixed Iwahori subgroup, then at $v_0$, the function is bi-invariant under the pro-unipotent of that fixed Iwahori subgroup, outside of $v_0$, the functions must be biinvariant under certain open compact subgroup (\cite{getz2024introduction}, Lemma 5.2.1). That is to say, the level at all places are fixed. Therefore, we can just apply an adapted version of Harish-Chandra's finiteness theorem (Lemma \ref{lem Harish-Chandra finiteness}) for cusp forms to see that, the number of representations $\Pi$ (resp. $\pi$) make nonzero contribution to the above identity is finite. 

Since the sums that we want ultimately in the local data are over automorphic representations of local component at $v_0$, we regroup the summation as 
\begin{equation*}\tag{7.3.3}\label{7.3.3}
\sum_{\Pi_{v_0}}[\sum_{\substack{\Pi'\text{ s.t.}\\ \Pi'_{v_0}\cong\Pi_{v_0}}}\langle\text{tr}\,\Pi'^{v_0} I_\theta,\phi^{v_0}\rangle ]\langle\text{tr}\,\Pi_{v_0} I_\theta,\phi_{v_0}\rangle=\sum_{\pi_{v_0}}[\sum_{\substack{\pi'\text{ s.t.}\\ \pi'_{v_0}\cong\pi_{v_0}}}\langle\text{tr}\,\pi'^{v_0} ,f^{v_0}\rangle]\langle\text{tr}\,\pi_{v_0} ,f_{v_0}\rangle.    
\end{equation*}
Where the outer sums are over isomorphic classes of irreducible cuspidal ($\theta$-stable) automorphic representations of $\underline{G}(\underline{E}_{v_0})$ (resp. $\underline{G}(\underline{F}_{v_0})$) that transforms under the central character $\lambda^{-1}_{r, v_0}$ (resp. $\lambda_{v_0}^{-1}$). In particular, they are in $\text{Irr}^\theta_{\chi_{r,v_0},\lambda_{r, v_0}}(\underline{G}(\underline{E}_{v_0}))$ (resp. $\text{Irr}^\theta_{\chi_{v_0},\lambda_{v_0}}(\underline{G}(\underline{F}_{v_0}))$). Now we set
\begin{equation*}
\begin{split}
    b(\Pi_{v_0})&=\sum_{\substack{\Pi'\text{ s.t.}\\ \Pi'_{v_0}\cong\Pi_{v_0}}}\langle\text{tr}\,\Pi'^{v_0} I_\theta,\phi^{v_0}\rangle\\
    a(\pi_{v_0})&=\sum_{\substack{\pi'\text{ s.t.}\\ \pi'_{v_0}\cong\pi_{v_0}}}\langle\text{tr}\,\pi'^{v_0} ,f^{v_0}\rangle
\end{split}
\end{equation*}
Then (\ref{7.3.3}) becomes
\begin{equation*}\tag{7.3.4}\label{7.3.4}
\sum_{\Pi_{v_0}}b(\Pi_{v_0})\langle\text{tr}\,\Pi_{v_0} I_\theta,\phi_{v_0}\rangle=\sum_{\pi_{v_0}}a(\pi_{v_0})\langle\text{tr}\,\pi_{v_0} ,f_{v_0}\rangle,     \end{equation*}
therefore, the implication (B)$\Rightarrow$(A) in the definition of local data is proved.

Assume that (\ref{7.3.4}) holds for the functions $\phi_{v_0}$ and $f_{v_0}$, since we already fix functions at other places $v$ other than $v_0$, we get (\ref{7.3.3}), thus (\ref{7.3.2}), ultimately (\ref{7.3.1}). Then we get $\phi_{v_0}$ and $f_{v_0}$ are associated at strongly regular norms by Proposition \ref{prop trace=association}, hence the implication (A)$\Rightarrow$(B) is also proved.
\end{proof}
\begin{lem}(Harish-Chandra's finiteness Theorem for cusp forms over function fields)\label{lem Harish-Chandra finiteness}
Let $K$ be an open compact subgroup of $\underline{G}(\A_{\underline{F}})$, let $V_\mathrm{cusp}(\lambda, K)$ denote the space of cuspidal functions $f: \underline{G}(F)\backslash\underline{G}(\A_{\underline{F}})/K\to\C$ such that $f(zg)=\lambda(z)f(g)$ for all $g\in\underline{Z}(\A_{\underline{F}})$. Then:
\begin{enumerate}
    \item There exists a compact subset $C\subset \underline{G}(\A_{\underline{F}})$ such that every function in $V_\mathrm{cusp}(\lambda, K)$ is supported on $\underline{Z}(\A_{\underline{F}})\underline{G}(F)C $.
    \item $\dim V_\mathrm{cusp}(\lambda, K)<\infty $.
\end{enumerate}
\end{lem}
\begin{proof}
For (i), see \cite{harder1974chevalley} Theorem 1.2.1, also \cite{getz2024introduction} Theorem 9.5.1. (ii) follows from (i) by noticing that the value of such $f\in V_\text{cusp}(\lambda, K)$ is determined on the finite cosets $C/K$ (we can always enlarge $C$ to contain $K$).    
\end{proof}
    
\subsection{A Further Reduction}\label{sec 7.4}
In this subsection, we show that the existence of \textit{local data adapted to $\mathrm{Irr}_{\chi,\lambda}(G)$ for all unitary characters $\lambda$} implies the existence of \textit{local data adapted to $\mathrm{Irr}_{\chi}(G)$}, so that we are back to the scenario considered in \cite{haines2012base}. The local data adapted to $\text{Irr}_{\chi}(G)$ is given analogously as follows:
\begin{enumerate}[(a')]
    \item An indexing set $\mathcal{I}$, possibly infinite;
    \item A collection of complex numbers $a_i(\pi)$ for $i\in\mathcal{I}$ and $\pi\in\text{Irr}_{\chi}(G)$;
    \item A collection of complex numbers $b_i(\Pi)$ for $i\in\mathcal{I}$ and $\Pi\in\text{Irr}^\theta_{\chi_r}(G_r)$.
\end{enumerate}
\begin{enumerate}[(1')]

    \item For $i$ fixed, the constants $a_i(\pi)$ and $b_i(\Pi)$ are zero for all but finitely many $\pi$ and $\Pi$.
    \item For $\phi\in\mathcal{H}(G_r,\rho_r)$ and $f\in\mathcal{H}(G,\rho)$, the following statements are equivalent:
    \begin{enumerate}[(A')]
        \item For all $i$, we have $\sum_\pi a_i(\pi)\langle\text{tr}\,\pi,f\rangle=\sum_\Pi b_i(\Pi)\langle\text{tr}\,\Pi I_\theta,\phi\rangle$;
        \item For all strongly regular elliptic semisimple norms $\gamma=\mathcal{N}(\delta)$, we have
        \begin{equation*}
         SO_\gamma(f)=SO_{\delta\theta}(\phi).   
        \end{equation*}
    \end{enumerate}
\end{enumerate}
\begin{prop}\label{prop existence of original local data}
Assume local data adapted to $\mathrm{Irr}_{\chi,\lambda}(G)$ exists for all unitary character $\lambda$ on $Z$, then the local data adapted to $\mathrm{Irr}_{\chi}(G)$ exists.    
\end{prop}
\begin{proof}
Indeed, given that local data exists for every unitary character $\lambda$, denote $\mathcal{I}_\lambda$ to be the index set in the corresponding local data, we take $\mathcal{I}=\displaystyle \coprod_{\lambda}\mathcal{I}_\lambda$. For $\pi\in\mathrm{Irr}_\chi(G)$, if $\pi\in\mathrm{Irr}_{\chi,\lambda}(G)$, then we take $a_i(\pi)$ to be the complex number in the local data adapted to $\mathrm{Irr}_{\chi,\lambda}(G)$, otherwise we set $a_i(\pi)=0$. We set complex numbers $b_i(\Pi)$ analogously. 
Therefore we have data (a'), (b'), (c'), subject to condition (1').

If (2')(A') is satisfied, for any character $\lambda$ and all $i\in \mathcal{I}_\lambda$, we have
\begin{equation*}
  \displaystyle\sum_\pi a_i(\pi)\langle\text{tr}\,\pi,f\rangle=\sum_\Pi b_i(\Pi)\langle\text{tr}\,\Pi I_\theta,\phi\rangle.  
\end{equation*}
It is straightforward to check that for $\pi\in\text{Irr}_{\chi,\lambda}(G)$ (resp. $\Pi\in\text{Irr}_{\chi_r,\lambda_r}(G_r)$), we have $\langle\text{tr}\,\pi,f\rangle=\langle\text{tr}\,\pi,f_\lambda\rangle$ (resp. $\langle\text{tr}\,\Pi I_\theta,\phi\rangle=\langle\text{tr}\,\Pi I_\theta,\phi_{\lambda_r}\rangle$ ) (recall that the latter traces are defined over $G/Z$). Therefore, we have $(f_\lambda, \phi_{\lambda_r})$ are associated for any unitary character $\lambda$ on $Z$. By Lemma \ref{lem 4.4} (i) and the further reduction to unitary characters, we can assert that (2')(B') holds. 

Similarly if (2')(B') is satisfied, we just reverse the process above to get (2')(A'). Therefore we have the existence of local data adapted to $\text{Irr}_\chi(G)$.

\end{proof}
\subsection{Labesse elementary functions and their traces}\label{sec 7.5}
We collect the definitions and properties of Labesse elementary functions that were defined and used in \cite{haines2012base} in this subsection. They will be used to prove (2')(A'). 

Recall that $A^{F_r}$ is a maximal $F_r$-split torus in $G$, whose centralizer $T$ is a maximal $F$-torus. Let $B=TU$ be the $F$-rational Borel subgroup defining the dominant Weyl chamber in $X_*(A)_\R$. Let $\rho$ denote the half-sum of the $B$-positive roots of $G$.

We fix a uniformizer $\varpi$ for the field $F$. We have the following commutative diagram:
\begin{equation*}
\begin{tikzcd}
X_*(A^{F_r}) \arrow[r, "\cdot(\varpi)"] \arrow[d, "N_r"]
& T(F_r) \arrow[d, "N_r"] \\
X_*(A) \arrow[r, "\cdot(\varpi)"]
&  T(F)
\end{tikzcd}
\end{equation*}
Consider a regular dominant cocharacter $\nu\in X_*(A^{F_r})$ and set $u=\nu(\varpi)\in T(F_r)$. Therefore we have $\tau=N_r(\nu)$ is also a regular dominant cocharacter in $X_*(A)$ with $t=\tau(\varpi)=N_r(u)\in T(F)$. Labesse\cite{labesse1990fonctions} constructed $F$-rational parabolic subgroup $P_{u\theta}$ with $F$-rational unipotent radical $N_{u\theta}$ and $F$-rational Levi factor $M_{u\theta}$ to $u\theta\in G_r\rtimes\langle\theta\rangle$. The subgroups $M_{u\theta}$ (resp. $P_{u\theta}$) of $G$ can be characterized as the set of elements $g\in G$ such that $(u\theta)^ng(u\theta)^{-n}$ remains bounded as $n$ ranges over all integers (resp. all positive integers). Since $(u\theta)^r=t$, we have $M_{u\theta}=M_t=T$ and $P_{u\theta}=P_t=B$.

Consider the map
\begin{equation*}\tag{7.5.1}\label{7.5.1}
\begin{split}
T(F_r)_1\backslash I_r\times T(F_r)_1u& \to G_r\\
[k,mu]&\mapsto k^{-1}mu\theta(k)
\end{split}    
\end{equation*}
where $[k,mu]$ denotes the equivalent class of $(k,mu)\in I_r\times T(F_r)_1u$ under the action of $m_0\in T(F_r)_1$ by $m_0\cdot(k, mu)=(m_0k, m_0mu\theta(m_0)^{-1})$. This map is injective whose image is a compact open subset $\mathfrak{L}_u\subset G_r$.

\begin{defn}\label{defn elementary function}
We define the elementary function $\phi_{u,\chi_r}$ on $G_r$ to vanish off of $\mathfrak{L}_u$, and on $\mathfrak{L}_u$, it is given by 
\begin{equation*}
    \phi_{u,\chi_r}(k^{-1}mu\theta(k))=\chi^{-1}_r(m)
\end{equation*}
for $k\in I_r$ and $m\in T(F_r)_1$.
\end{defn}
We summarize the useful properties of $\phi_{u,\chi_r}$ in the following proposition.
\begin{prop}\label{prop. of properties of elementary functions}
\begin{enumerate}
    \item The functions $\phi_{u,\chi_r}$ are well-defined and belong to $\mathcal{H}(G_r,\rho_r)$.
    \item The functions $\phi_{u,\chi_r}$ are supported on the set of strongly $\theta$-regular elements in $G_r$.
\end{enumerate}
\end{prop}
\begin{proof}
(2) is from Lemma 8.1.2 (ii) in \cite{haines2012base}. For (1), the fact that the functions are well-defined is easy to check. To show that they belong to the Hecke algebra, the first step is to show that $\mathfrak{L}_u=I_ruI_r$. It is clear that $\mathfrak{L}_u\subset I_ruI_r$ by \ref{7.5.1}. To show the equality, following \cite{labesse1995} Prop. IV.1.1, since both sides are open compact, it suffices to show that they have the same volume. It is straightforward to see that they are of the same volume $\mathrm{vol}(I_r)\cdot\mathrm{vol}(I_r/uI_ru^{-1})=\mathrm{vol}(I_r)\cdot\delta^{-1}_{B_r}(u)$. To show that it is in the Hecke algebra, the calculations in the proof of \cite{haines2012base} Lemma 8.1.3 is still valid here.
\end{proof}
When $r=1$ and $u=t$, we can define $f_{t,\chi}\in\mathcal{H}(G,\rho)$ analogously. By \S 8.3 in \cite{haines2012base}, we have the following associated results:
\begin{prop}\label{prop. of association of elementary functions}
The functions $f_{t,\chi}$ and $\phi_{u,\chi_r}$ are associated.
\end{prop}
The next step is to calculate the traces of elementary functions. Let 
$\Pi$ denote a $\theta$-stable admissible representation of $G_r$, and fix an intertwiner $I_\theta:\Pi\cong \Pi^\theta$. The locally integrability of the distribution character in characteristic $0$ that was used in \S 8.4-8.5 in \cite{haines2012base} is also available in the restricted form by Proposition 13.1 of \cite{adler2007local}:
\begin{prop}\label{prop of character distribution}
For $\phi\in C^\infty_c(G^{\theta\mathrm{-reg}}(F_r))$, where $G^{\theta\mathrm{-reg}}(F_r)$ is the open subset of $\theta$-regular elements in $G(F_r)$ whose complement has measure $0$, the functional $\phi\mapsto\langle\tr \,\Pi I_\theta, \phi\rangle$ is represented by a locally constant function $\Theta_{\Pi\theta}$ on $G^{\theta\mathrm{-reg}}(F_r)$. That is, 
\begin{equation*}\tag{7.5.2}
\langle\tr \,\Pi I_\theta, \phi\rangle=\int_{G_r} \Theta_{\Pi\theta}(g)\phi(g)\,dg   
\end{equation*}
\end{prop}
The fact that our elementary functions are supported inside the set of $\theta$-regular elements (Proposition \ref{prop. of properties of elementary functions} (2)) allows us to repeat the same calculations as in \cite{haines2012base}. Along the way, we also use the twisted trace identity $\Theta_{\Pi\theta}(mu)=\Theta_{\Pi_U\theta}(mu)$ of Rogawski (\cite{rogawski1988trace}, Prop. 7.4), where $\Pi_U$ is the Jacquet module of $\Pi$ corresponding to the Borel subgroup $B_r=T_rU_r$.

We fix some notations before we get into the traces of elementary functions. We write $W$ (resp. $W_r$) for the relative Weyl group associated to the $F$-split (resp. $F_r$-split) torus $A$ (resp. $A^{F_r}$) in $G$. 
\begin{itemize}
    \item Let $\Xi$ denote the set of characters on $T(F_r)$ which extend some $W_r$-conjugate of $\chi_r$. Let $\Xi(\chi_r)\subset\Xi$ consist of those whose restriction to $T(F_r)_1$ is precisely $\chi_r$. For $\xi'\in\Xi(\chi_r)$, we may write
    \begin{equation*}
        \xi'=\tilde{\chi}^{\varpi}_r\eta'
    \end{equation*}
    for a unique unramified character $\eta'$ on $T_r$ (see Remark \ref{remark extension of chracter} for the definition of $\tilde{\chi}^{\varpi}_r$).
    \item Let $\Xi^\theta$ (resp. $\Xi(\chi_r)^\theta$) denote the subset of $\theta$-fixed elements in $\Xi$ (resp. $\Xi(\chi_r)$).
\end{itemize}
Suppose the supercuspidal support of $\Pi$ is $(T_r,\xi')_{G_r} $ for some extension $\xi'$ of a $W_r$-conjugate of $\chi_r$. Then $\Pi_U$ is a subquotient of 
\begin{equation*}
    (i^{G_r}_{B_r}(\xi'))_U=\delta^{1/2}_{B_r}\bigoplus_{w\in W_r}\C_{^w\xi'}
\end{equation*}
(cf. \cite{casselman1995introduction}, Prop. 6.4.1), where $\C_{^w\xi'}$ is the character on $T_r$ corresponding to $^w\xi'$.
\begin{itemize}
    \item We have a well-defined subset $\Xi(\Pi)\subset\Xi$ and positive multiplicities $a_{\xi'}$ for $\xi\in\Xi(\Pi)$ such that
    \begin{equation*}
\Pi_U=\delta^{1/2}_{B_r}\bigoplus_{w\in \Xi(\Pi)}\C^{a_{\xi'}}_{\xi'}
    \end{equation*}
    \item Set $\Xi(\Pi)^\theta=\Xi^\theta\cap \Xi(\Pi)$ and $\Xi(\Pi,\chi_r)^\theta=\Xi(\chi_r)^\theta\cap\Xi(\Pi)$. 
    \item For $\xi'\in\Xi^\theta$, we set 
    \begin{equation*}
        \tr(I_\theta, \Pi, \xi'):=\langle\tr\, I_\theta; \delta^{1/2}_{B_r}\C^{a_{\xi'}}_{\xi'}\rangle.
    \end{equation*}
    Here we recall that the intertwiner $I_\theta: \Pi\cong \Pi^\theta$ induces an intertwiner $I_\theta: \Pi_U\cong \Pi^\theta_U$.
\end{itemize}
We summarize the traces of elementary functions in the following proposition.
\begin{prop}\label{prop. traces of elementary functions}
\begin{enumerate}
    \item Suppose $\Pi$ is an irreducible and $\theta$-stable object in $\mathfrak{R}(G_r)$. If \\ $\langle\tr\,\Pi I_\theta, \phi_{u,\chi_r}\rangle\neq0$, then $\Pi\in\mathfrak{R}_{\chi_r}(G_r)$.
    \item For $\Pi\in\mathfrak{R}_{\chi_r}(G_r)$, we have \begin{equation*}\tag{7.5.3}\label{7.5.3}
    \langle\tr\,\Pi I_\theta, \phi_{u,\chi_r}\rangle=q^{\langle\rho,\tau\rangle}\sum_{\xi'\in\Xi(\Pi,\chi_r)^\theta} \eta'(u)\tr(I_\theta, \Pi, \xi').
    \end{equation*}
    In particular, when $r=1$, we have 
    \begin{equation*}\tag{7.5.4}\label{7.5.4}
    \langle\tr\,\pi, f_{t,\chi}\rangle=q^{\langle\rho,\tau\rangle}\sum_{\xi'\in\Xi(\pi,\chi)} \eta(t)\dim\C^{a_\xi}_\xi.
    \end{equation*}
    \item Let $e_{\rho_r}\in \mathcal{Z}(G_r,\rho_r)$ be the idempotent to have support $I$ and to take value $\rho_r(k)^{-1}$ at $k\in I_r$. Then we have     \begin{equation*}\tag{7.5.5}\label{7.5.5}
    \langle\tr\,\Pi I_\theta, e_{\rho_r}\rangle= \sum_{\xi'\in\Xi(\Pi,\chi_r)^\theta} \tr(I_\theta, \Pi, \xi'),  
    \end{equation*}
    and 
    \begin{equation*}\tag{7.5.6}\label{7.5.6}
    \langle\tr\,\Pi I_\theta, e_{\rho_r}\rangle=  \langle\tr\,\Pi I_\theta, \phi_{u,\chi_r}\rangle|_{u=1}.  
    \end{equation*}
    In particular, when $r=1$, we have
    \begin{equation*}\tag{7.5.7}\label{7.5.7}
    \langle\tr\,\pi, e_{\rho}\rangle= \sum_{\xi'\in\Xi(\pi,\chi)} \dim\C^{a_\xi}_\xi. 
    \end{equation*}
    
\end{enumerate}    
\end{prop}
\begin{proof}
The calculations in the proof of Lemma 8.4.1 and \S 8.5 in \cite{haines2012base} carry over unchanged.    
\end{proof}

\subsection{End of Proof}\label{sec 7.6}

We fix $\phi\in Z(G_r, \rho_r)$ and $f=b(\phi)\in \mathcal{Z}(G, \rho)$. By the various reduction steps, we may assume $\gamma$ is a (strongly) regular elliptic semisimple norm, i.e., $\gamma=\mathcal{N}(\delta)$ for some $\delta\in G(E)$ $\theta$-(strongly) regular $\theta$-elliptic semisimple. Proposition \ref{prop existence of original local data} guarantees that we are in the situation considered in \cite{haines2012base}, \S 9.2. Thus by Proposition \ref{prop. of association of elementary functions}, we have 
\begin{equation*}\tag{7.6.1}\label{7.6.1}
\sum_\pi a_i(\pi)\langle\text{tr}\,\pi,f_t\rangle=\sum_\Pi b_i(\Pi)\langle\text{tr}\,\Pi I_\theta,\phi_u\rangle    
\end{equation*}
for each pair of functions $(f_t, \phi_u)$, 
where we denote $f_t:=f_{t,\chi}$ and $\phi_u:=\phi_{u,\chi_r}$. As in \textit{loc. cit.}, we can eventually separate (\ref{7.6.1}) into
\begin{equation*}\tag{7.6.2}\label{7.6.2}
\begin{split}
\sum_{\substack{\xi_0\in\Xi(\chi)/W_\chi \\\text{s.t. }\xi_{0r}\in W_{\chi_r}\xi'_0}}\sum_{\pi\in i^G_B(\xi_0)}a(\pi)&\langle\text{tr}\,\pi,e_\rho\rangle = \\
&\sum_{\Pi\in i^{G_r}_{B_r}(\xi'_0)}b(\Pi)\,\langle\text{tr}\,\Pi I_\theta,e_{\rho_r}\rangle 
\end{split}
\end{equation*}
for a fixed $\xi'_0\in\Xi(\chi_r)^\theta$ where both sides will vanish if $W_{\chi_r}\xi'_0$ contains no norm. If $\xi'_0=\xi_{0r}$, we multiply both sides of (\ref{7.6.2}) by $ch_{\xi_0}(b(\phi))=ch_{\xi_{0r}}(\phi)$, if $\xi'_0$ is not a norm, we multiply both sides of (\ref{7.6.2}) by $ch_{\xi'_0}(\phi)$. The upshot is that, after summing over $\xi'_0\in \Xi(\chi_r)/ W_{\chi_r}$, we will have the desired identities of traces
\begin{equation*}\tag{7.6.3}\label{7.6.3}
\sum_\pi a_i(\pi)\langle\text{tr}\,\pi,b(\phi)\rangle=\sum_\Pi b_i(\Pi)\langle\text{tr}\,\Pi I_\theta,\phi\rangle .   
\end{equation*}
Therefore $\phi$ and $b(\phi)$ are associated at all strongly regular elliptic elements, as desired.

\end{document}